\documentclass[10pt]{article}
\usepackage{algorithm}
\usepackage{algpseudocode}
\usepackage{amsfonts}
\usepackage{amsmath}
\usepackage{amssymb} % 実数Rの入力する
\usepackage{amsthm} % 定理を記述する
\usepackage[title]{appendix} % 単語"appendix"を挿入する（https://tex.stackexchange.com/questions/226481/appendix-section-title）
\usepackage{autobreak}
\usepackage{bigdelim} % 表の周りに波括弧とコメントをつける（https://motchy99.blog.fc2.com/blog-entry-78.html）
\usepackage{blkarray}

\makeatletter
\def\BA@fnsymbol#1{\ensuremath{%
  \ifcase#1\or *\or \dagger\or \ddagger\or \mathsection\or \mathparagraph\or \|\or **\or \dagger\dagger\or \ddagger\ddagger \else@ctrerr\fi}}%
\makeatother
\usepackage{bm}
\usepackage{cancel} % 数式の上からキャンセル線を引く
\usepackage{cases}
\usepackage{chngcntr}
\usepackage{colortbl} % 表に色付けする
\usepackage{comment}
\usepackage{empheq} % 連立方程式を書く（https://tex.stackexchange.com/questions/451687/main-label-in-a-system-of-equations）
\usepackage{enumitem}%\usepackage{enumerate}
\usepackage{fancyhdr} % 上側に線を引く
\usepackage{fancyvrb} % verbatimを使う
\usepackage{framed} % 枠囲いをする（その２）
\usepackage{graphicx}% 図表を挿入する [dvipdfmx]
\usepackage{hyperref}
\usepackage[many]{tcolorbox}
\tcbuselibrary{breakable} % tcolorboxでページまたぎを許可
\usepackage[colorinlistoftodos,prependcaption,textsize=footnotesize]{todonotes}

\hypersetup{
  colorlinks,
  linkcolor={red!50!black},
  citecolor={blue!50!black},
  urlcolor={blue!80!black}
}
\usepackage{mathrsfs} % 花文字
\usepackage{multirow} % 表で縦方向に結合する
\usepackage[sort&compress,numbers]{natbib}
\usepackage{setspace} % ダブルスペースを使う
\usepackage{subfig} % 図と表を並置する
\usepackage{pgfplots} % hyperrefを使用するために必要
\usepackage{pxrubrica} % 圏点を入力（\kenten{~~~}）
\usepackage{url}
\usepackage[top=35truemm,bottom=35truemm,left=30truemm,right=30truemm]{geometry} % 余白の設定
\makeatletter

\@addtoreset{equation}{section}
\makeatother
\algtext*{EndWhile}
\algtext*{EndIf}
\algtext*{EndFor}
\theoremstyle{plain}
\newtheorem{theorem}{Theorem}[section]
\newtheorem{proposition}[theorem]{Proposition}%
\newtheorem{corollary}[theorem]{Corollary}
\newtheorem{lemma}[theorem]{Lemma}

\theoremstyle{definition}
\newtheorem{definition}[theorem]{Definition}%
\newtheorem{assumption}[theorem]{Assumption}
\theoremstyle{remark}
\newtheorem{example}[theorem]{Example}%
\newtheorem{remark}[theorem]{Remark}%

\DeclareMathOperator*{\dist}{dist}

\DeclareMathOperator*{\cond}{cond}

\DeclareMathOperator*{\setint}{int}
\DeclareMathOperator*{\setspan}{span}

\DeclareMathOperator*{\tr}{tr}

\DeclareMathOperator*{\Diag}{Diag}

\DeclareMathOperator*{\ri}{ri}

\DeclareMathOperator*{\Feas}{Feas}

\newcommand{\argmin}{\mathop{\rm arg\,min}\limits}

\newcommand{\minimize}{\mathop{\rm minimize}\limits}
\newcommand{\maximize}{\mathop{\rm maximize}\limits}
\newcommand{\subjectto}{\mathrm{subject~to}}
\newcommand{\relmiddle}[1]{\mathrel{}\middle#1\mathrel{}}

\newcommand{\abs}[1]{\lvert #1 \rvert} % 絶対値
\newcommand{\norm}[1]{\lVert #1 \rVert} % ノルム

\newcommand{\bbB}{\mathbb{B}}

\newcommand{\bbN}{\mathbb{N}}

\newcommand{\bbR}{\mathbb{R}}

\newcommand{\bbZ}{\mathbb{Z}}

\newcommand{\calA}{\mathcal{A}}

\newcommand{\calC}{\mathcal{C}}

\newcommand{\calE}{\mathcal{E}}
\newcommand{\calF}{\mathcal{F}}

\newcommand{\calK}{\mathcal{K}}

\newcommand{\calM}{\mathcal{M}}

\newcommand{\calQ}{\mathcal{Q}}
\newcommand{\calR}{\mathcal{R}}
\newcommand{\calS}{\mathcal{S}}

\newcommand{\calV}{\mathcal{V}}
\newcommand{\calW}{\mathcal{W}}
\newcommand{\calX}{\mathcal{X}}

\newcommand{\RNum}[1]{\uppercase\expandafter{\romannumeral #1\relax}} % ローマ数字（大文字）
\newcommand{\Rnum}[1]{\lowercase\expandafter{\romannumeral #1\relax}} % ローマ数字（小文字）

\title{\Large \sf Radial-Type Error Bounds for Semidefinite Feasibility Problems without Strict Feasibility: Qualitative Estimates and Asymptotic Tightness}

\makeatletter
\let\@fnsymbol\@arabic
\makeatother

\author{
 \normalsize
     Mitsuhiro Nishijima\thanks{Department of Industrial and Systems Engineering, Keio University, 3-14-1 Hiyoshi, Kohoku-ku, Yokohama-shi, 2238522, Kanagawa, Japan. ({\tt nishijima@keio.jp}).} %ORCID: 0000-0003-4871-9156
 \and
 \normalsize
Guoyin Li\thanks{Department of Applied Mathematics, University of New South Wales, 2052, Sydney, Australia. ({\tt g.li@unsw.edu.au})} % ORCID: 0000-0002-2099-7974
\and
\normalsize
Bruno F. Louren\c{c}o\thanks{Department of Fundamental Statistical Mathematics, The Institute of Statistical Mathematics, 10-3 Midori-cho, Tachikawa-shi, 1908562, Tokyo, Japan. ({\tt bruno@ism.ac.jp}).} % ORCID: 0000-0003-2369-0107
         }

\begin{document}
\maketitle
\begin{abstract}\noindent
In this paper, we develop a systematic framework for deriving explicit error bounds for semidefinite feasibility problems without assuming strict feasibility (Slater's condition), a setting in which existing results are limited. Our main technical contribution is the introduction of radial-type H\"{o}lder error bounds, where the error bound constant depends explicitly on the norm of the reference matrix through radial modulus functions. By combining facial reduction with recently developed facial residual functions, we obtain explicit descriptions of these modulus functions, yielding qualitative radial-type H\"{o}lder error bounds \emph{without imposing any constraint qualifications}. Our results complement the classical work of Sturm by providing explicit estimates for the constants involved in Sturm's local H\"{o}lder error bounds over bounded sets with a given size.
We further analyze the asymptotic behavior of these bounds as the dimension of the underlying matrix space grows, identifying regimes in which they can be asymptotically tight up to a dimension-free constant. 
As an application, we establish explicit error bounds for the optimality system of semidefinite programs by reformulating them as feasibility problems, a setting where Slater's condition typically fails.
Under the generically satisfied strict complementarity condition, we derive radial-type error bounds without assuming the usual solution uniqueness requirement, and demonstrate their asymptotic tightness through an explicit example.
\end{abstract}


\vspace{0.5cm}
% 15B48: Positive matrices and their generalizations; cones of matrices
% 52A20: Convex sets in n dimensions (including convex hypersurfaces)
% 90C22: Semidefinite programming
% 90C25: Convex programming

\section{Introduction}
The problem of finding an element of the intersection of an affine space $\calV$ and a closed convex cone $\calK$ in a finite-dimensional vector space $\calE$ is called a \emph{conic linear feasibility problem}.
We write this problem as $\Feas(\calV,\calK)$ {and, throughout this paper, we always assume that $\calV \cap \calK \neq \emptyset$, that is, $\Feas(\calV,\calK)$ admits a feasible solution.} 
\emph{Error bounds} provide us upper bounds on the distance between a point $x \in \calE$ and the feasible set $\calV \cap \calK$ by using the distances between $x$ and $\calV$ as well as that between $x$ and $\calK$.
Assuming that $\calE$ is equipped with the norm $\norm{\cdot}$ induced by an inner product, we define $\dist(x,\calC) \coloneqq \min\{\norm{x-y} \mid y\in \calC\}$ as the distance between $x$ and a closed convex set $\calC$ in $\calE$.
Following \cite{Sturm2000,LLP2023}, the problem $\Feas(\calV,\calK)$ is said to admit a \emph{(local) H\"{o}lder error bound} with exponent $\gamma \in (0,1]$ if for any bounded set $B$ (sometimes referred to as test sets), it follows that
\begin{equation}
\dist(x,\calV \cap \calK) \le \kappa_B \max\{\dist(x,\calV), \dist(x,\calK)\}^{\gamma} \text{ for all $x\in B$} \label{eq:Holder_eb}
\end{equation}
for some nonnegative modulus $\kappa_B$.\footnote{In some literature, this property is stated in an equivalent way.
For example, in \cite{DLW2017}, the quantity $\max\{\dist(x,\calV), \dist(x,\calK)\}^{\gamma}$ is replaced by $\dist(x,\calV)^{\gamma}+\dist(x,\calK)^{\gamma}$ and the pair $(\calV,\calK)$ is referred to as being $\gamma$-H\"{o}lder regular.} 
If $\gamma=1$, then \eqref{eq:Holder_eb} is called a \emph{Lipschitz} error bound.
Instead of restricting the set $B$ to be bounded, if \eqref{eq:Holder_eb} holds with $B$ being the whole space, then we say that a \emph{global} error bound holds for the problem.

In the definition of H\"{o}lder error bounds, the quantity $\dist(x,\calV \cap \calK)$ is referred to as the \emph{forward error}, which measures the distance to the solution set of the feasibility problem, and is hard to estimate or compute directly, and hence is regarded as unknown.
On the other hand, in numerical computation, what is often readily available to users is the quantity $\calE_{\rm b}(x) \coloneqq \max\{\dist(x,\calV), \dist(x,\calK)\}$ (or equivalently, $\dist(x,\calV)+ \dist(x,\calK)$), which is called the \emph{backward error}.
Thus, the associated error bounds provide us a good way to quantify the unknown forward error in terms of the computable backward error.  

In the case where $\calK$ is a polyhedral cone, the feasibility problem $\Feas(\calV,\calK)$ has a close connection with linear programming, and error bound results for $\Feas(\calV,\calK)$ have been studied extensively. For example, in this case, it is known that the problem $\Feas(\calV,\calK)$ admits a global Lipschitz error bound~\cite{Hoffman1952}, and its modulus is often called a \emph{Hoffman constant}~\cite{KT1995,Za03,Pena2024}. 
The Hoffman constant is closely related to other important quantities such as the chi measure $\chi$~\cite{Zhang2000} and the distance to ill-posedness~\cite{PVZ20XX_Equivalence,PVZ2021}. 

We refer to a conic linear feasibility problem associated with a positive semidefinite cone as a \emph{semidefinite feasibility problem}.
The semidefinite feasibility problem and its associated error bounds have received considerable attention due to their natural connections with the stability and conditioning of the semidefinite programs (SDPs), as well as the powerful modeling capabilities of SDPs~\cite{WSV2000}. 
For example, for semidefinite feasibility problems satisfying the strict feasibility condition (often known as \emph{Slater's condition}), they admit Lipschitz error bounds; see \cite[Corollary~5]{BBL1999} and \cite[Lemma~3.1]{BT2003}.
If a semidefinite feasibility problem satisfies a stronger version of Slater's condition, the problem admits a global Lipschitz error bound and its modulus can be interpreted geometrically~\cite{DH1999,Zhang2000,Hu2005,LMP2014,LAA+2025}. 

While Slater's condition is a standard { constraint qualification} in convex optimization, verifying it can require nontrivial effort. More importantly, it may easily fail in several key settings, such as SDPs arising from relaxations of quadratic assignment problems and semidefinite feasibility problems describing the optimal solution sets of SDPs (see, for example, the survey \cite{DW2017} and the references therein).

In the absence of Slater's condition, \cite[Theorem~3.3]{Sturm2000} established that semidefinite feasibility problems admit (local) H\"{o}lder error bounds. Notably, \cite{Sturm2000} showed that the H\"{o}lder exponent $\gamma$ can be characterized entirely by the \emph{singular degree}, which is determined by the number of steps required by the facial reduction algorithm~\cite{BW1981}. This exponent also plays a crucial role in quantifying the convergence rate of alternating projection methods for ill-posed semidefinite feasibility problems~\cite{DLW2017}. More recently, Lourenço~\cite{Lourenco2021} and Lindstrom, Lourenço, and Pong~\cite{LLP2023} introduced \emph{facial residual functions} and established error bounds for general conic linear feasibility problems, thereby extending Sturm's error bound approach to other classes of cones \cite{LLP2025,LLP2023,LLL+2025,LLL+2024,WLP20XX}.

However, these results are primarily existential and do not provide explicit or computable estimates of the associated error bound constants. To the best of our knowledge, the problem of systematically estimating the constant $\kappa_B$ in H\"{o}lder error bounds for semidefinite feasibility problems, particularly in the absence of {constraint qualifications} such as Slater's condition, remains largely open. 
Given the success of explicit estimates for Hoffman constants in analyzing linear convergence rates and studying sensitivity in polyhedral convex programs~\cite{LL2010,WL2014,BS2017,NNG2019,AHL+2023}, analogous explicit estimates of error bound constants could similarly facilitate more informative convergence analysis and perturbation bounds for numerical algorithms solving semidefinite feasibility problems or SDPs. 
This naturally leads to the following research question:
\emph{In the absence of {constraint qualifications} such as Slater's condition, can one obtain effective explicit estimates of the error bound constants in H\"{o}lder error bounds for semidefinite feasibility problems? If so, how tight are these estimates, particularly as the dimension of the underlying matrix space increases?} 

Let $\calS_+^n$ denote the cone of positive semidefinite matrices in the space $\calS^n$ of real $n\times n$ symmetric matrices.
The purpose of this paper is to address the above questions by studying the following \emph{radial-type H\"{o}lder error bounds} with exponent $\gamma \in (0,1]$: 
\begin{equation*}
\dist(\bm{X},\calV \cap \mathcal{S}_+^n) \le \mu(\|\bm{X}\|_{\rm F})\, \calE_{\rm b}(\bm{X})^{\gamma}  
\text{ for all  $\bm{X} \in \mathcal{Q}$},
\end{equation*}
or, more generally, with the form 
\begin{equation}
\dist(\bm{X},\calV \cap \mathcal{S}_+^n) \le \sum_{j=0}^{\ell} \mu_j(\|\bm{X}\|_{\rm F}) \, \calE_{\rm b}(\bm{X})^{\gamma_j} \text{ for all $\bm{X} \in \mathcal{Q}$}. \label{eq:radial_Holder_eb} 
\end{equation}
Here, $\|\cdot\|_{\rm F}$ denotes the Frobenius norm, $\mathcal{Q}$ is a given subset of $\mathcal{S}^n$, $\ell$ is a nonnegative integer, $\gamma_j$ is a positive constant with $\gamma_j \le \gamma$ for every $j=0,\ldots,\ell$, and $\mu_j$, $j=0,\ldots,\ell$, are   nonnegative nondecreasing functions, which we call \emph{radial modulus functions}.
In particular, our main aim is to derive radial-type H\"{o}lder error bounds with explicit exponents $\gamma_0,\dots,\gamma_{\ell}$ and explicit form of the radial modulus functions $\mu_0,\dots,\mu_{\ell}$ in the absence of {constraint qualifications}.

It is worth emphasizing that the notion of a radial-type error bound unifies both local and global error bounds.
In particular, when $\ell = 0$, $\gamma_0 = \gamma$, $\mathcal{Q} = \mathcal{S}^n$, and $\mu_0$ is a constant function, the above condition reduces to a global error bound.
On the other hand, since each radial modulus function $\mu_j$ is nondecreasing, the bound immediately yields a local error bound with an explicit constant (depending on $\rho$) over any bounded test set contained in the ball of radius $\rho$.
We also mention that we do not restrict ourselves to the case where $\mathcal{Q}$ is bounded. In fact, our main focus and result are on the cases where the test set $\mathcal{Q}$ is the ambient space $\mathcal{S}^n$.
In line with existing literature, we also consider the cases where $\mathcal{Q}=\{\bm{X}\in \calS^n \mid \calE_{\rm b}(\bm{X}) \le \lambda\}$ where $\lambda$ is a given number. 

The main contributions of this paper are as follows.
\begin{enumerate}[label=(\arabic*)]
\item Firstly, we provide explicit computable forms for the radial modulus functions $\mu_0,\dots,\mu_{\ell}$.
We achieved these by exploiting the facial reduction process together with recently developed facial residual functions \cite{LLP2023}. This approach yields new qualitative radial-type H\"{o}lder error bounds in the form of \eqref{eq:radial_Holder_eb} where the test set is $\mathcal{Q}=\mathcal{S}^n$, without assuming any {constraint qualifications} (Theorems~\ref{thm:eb_alpha_beta} and \ref{thm:error_bound}).
As a consequence, in line with the existing literature, we derive radial-type H\"{o}lder error bounds in the form of \eqref{eq:radial_Holder_eb} where the test set is $\mathcal{Q}=\{\bm{X} \in \calS^n \mid \calE_{\rm b}(\bm{X}) \le 1\}$ (Corollary~\ref{cor:error_bound_constant}).
This complements the classical work of Sturm \cite{Sturm2000} by providing explicit estimates for the constants involved in the H\"{o}lder error bounds over bounded sets with given sizes.

Interestingly, the derived radial modulus functions depend explicitly on a quantity $d^*$ called the \emph{distance to partial polyhedral Slater's (PPS) condition}, which is a variant of the concept of singularity degree in the literature and relates to the number of steps of a certain facial reduction algorithm. 

\item Secondly, we investigate the asymptotic tightness of the derived error bounds.
To formalize this notion, we introduce in Definition~\ref{def:tight_eb} the concept of being \emph{asymptotically tight up to a dimension-free constant}. Consider a family of semidefinite feasibility problems in which the dimension of the underlying matrix space varies.
Roughly speaking, a derived error bound is asymptotically tight up to a dimension-free constant $C \in (0,1]$ if it is essentially optimal up to the multiplicative factor $C$.
This means that one cannot obtain a strictly sharper error bound by uniformly scaling the radial modulus functions by any dimension-independent constant factor strictly smaller than $C$.

Our investigation of asymptotic tightness is based on the values of $d^*$, the distance to the PPS condition of the feasibility problem.
We first focus on instances where $d^*$ are $0$ and $1$, respectively, and show in Sections~\ref{sec:dPPS_0} and \ref{sec:dPPS_1} that the corresponding error bounds can be asymptotically tight up to a dimension-free constant.
On the other hand, for the semidefinite feasibility problem presented in \cite[Example~2]{Sturm2000} (which satisfies $d^*=n-1$), our explicit error bound is not asymptotically tight up to a dimension-free constant. 

\item Thirdly, as an application, we derive explicit error bounds for the optimality system of SDPs by reformulating it as a semidefinite feasibility problem, a setting in which Slater's condition typically fails. {Without assuming solution uniqueness that the existing literature requires, see, for example, \cite[Section~1]{DU2021} and  \cite[Theorem~4.1]{DYC+2021} (and also the related reference  \cite{NO1999})}, we establish explicit radial-type error bounds under the strict complementarity condition, which holds generically~\cite{PT2002,AHO1997}. 
Moreover, we show, through an explicit example, that the derived bound can be asymptotically tight in this setting. 
\end{enumerate}

The organization of this paper is as follows.
In Section~\ref{sec:preliminaries}, we recall and introduce notation and concepts used in the subsequent sections.
In Section~\ref{sec:general}, we derive qualitative radial-type H\"{o}lder error bounds for semidefinite feasibility problems.
In Section~\ref{sec:tight}, we discuss the asymptotic tightness of the error bound presented in the previous section.
In Section~\ref{subsec:optset_SDP}, as an application of the result of Section~\ref{sec:general}, we obtain explicit error bounds for the optimality system of SDPs satisfying the strict complementarity condition.
In Section~\ref{sec:conclusion}, we finish with concluding remarks.

\section{Preliminaries}\label{sec:preliminaries}
In this section, we recall some preliminaries and basic definitions as well as existing tools that will be used in the next few sections. 
\subsection{Notation}
For an integer $n$, we write $\bbZ_{\ge n}$ for the set of integers greater than or equal to $n$.
Let $V$ be a finite-dimensional Euclidean space and denote the corresponding norm by $\norm{\cdot}$.
For $S \subseteq V$, we use $\setint S$, $\ri S$, $\setspan S$, and $S^\perp$ to denote the interior of $S$, the relative interior of $S$, the smallest subspace containing $S$, and the orthogonal complement of $S$, respectively.
For a closed convex set $S$ in $V$ and $x \in V$, we define $\dist(x,S) \coloneqq \min\{\norm{x-y} \mid y\in S\}$ and the projection of $x$ onto $S$ as $P_S(x) \coloneqq \argmin\{\norm{x - y} \mid y\in S\}$.
For $c \in V$ and $r \ge 0$, the closed ball with center $c$ and radius $r$ is denoted by
$\mathbb{B}(c,r) \coloneqq \{x \in V \mid \|x - c\| \le r\}$.
For a linear mapping $\calA$, we use $\calA^*$ to denote the adjoint of $\calA$.

Vectors are denoted by boldfase lowercase letters such as $\bm{a}$.
The space of $n$-dimensional real vectors is denoted by $\bbR^n$ and the set of entrywise nonnegative vectors in $\bbR^n$ is denoted by $\bbR_+^n$.
For a vector $\bm{a} \in \bbR^n$, we write $\bm{a}^\top$ for the transpose of $\bm{a}$ and we use $\norm{\bm{a}}_2 \coloneqq \sqrt{\bm{a}^\top \bm{a}}$ to denote the $2$-norm of $\bm{a}$.

Matrices are denoted by boldface uppercase letters such as $\bm{A}$.
The space of real $m\times n$ matrices is denoted by $\bbR^{m\times n}$ and the space of real $n\times n$ symmetric matrices is denoted by $\calS^n$.
For $\bm{A}\in \calS^n$, we define $\bbR\bm{A} \coloneqq \{t\bm{A} \mid t\in \bbR\}$ and $\bbR_+\bm{A} \coloneqq \{t\bm{A} \mid t\in \bbR_+\}$, respectively.
We write the $(i,j)$th element of a matrix $\bm{A}$ as $A_{i,j}$.
When there is no ambiguity we will abbreviate $A_{i,j}$ as $A_{ij}$.
The zero matrix is denoted by $\bm{O}$ and the identity matrix of order $n$ is denoted by $\bm{I}_n$.
We use $\bm{E}_{i,j}$ or $\bm{E}_{ij}$ to denote the matrix whose $(i,j)$th element is $1$ and all other elements are $0$.
Its order is determined from the context.
For $a_1,\dots,a_n \in \bbR$, we use $\Diag(a_1,\dots,a_n)$ to denote the diagonal matrix whose $(i,i)$th element is $a_i$ for each $i = 1,\dots,n$.
Similarly, for real symmetric matrices $\bm{A}_1,\dots,\bm{A}_n$, the block diagonal matrix whose $i$th block is $\bm{A}_i$ for each $i = 1,\dots,n$ is denoted by $\Diag(\bm{A}_1,\dots,\bm{A}_n)$.
For a subset $\calC$ of $\calS^m$, we define $\calC \oplus \{0\}^n \coloneqq \{\Diag(\bm{A},\bm{O}) \in \calS^{m+n} \mid \bm{A} \in \calC\}$.
The set $\{0\}^n \oplus \calC$ is also defined analogously.
For a matrix $\bm{A} \in \bbR^{m\times n}$, we use $\bm{A}^\top$ to denote the transpose of $\bm{A}$ and write $\tr(\bm{A})$ for the trace of $\bm{A}$.
For matrices $\bm{A},\bm{B} \in \bbR^{m\times n}$, we define $\langle \bm{A},\bm{B} \rangle \coloneqq \tr(\bm{B}^\top \bm{A})$, which is an inner product on $\bbR^{m\times n}$.
The Frobenius norm, which is the norm induced by this inner product, is denoted by $\norm{\cdot}_{\rm F}$.
For a matrix $\bm{A} \in \calS^n$, we use $\lambda_{\rm max}(\bm{A})$ and $\lambda_{\rm min}(\bm{A})$ to denote the maximum and minimum eigenvalues of $\bm{A}$, respectively.
In addition, $\lambda_{\rm min}^+(\bm{A})$ denotes the smallest positive eigenvalue of $\bm{A}$ if $\bm{A}$ has at least one positive eigenvalue.
Otherwise, we let $\lambda_{\rm min}^+(\bm{A}) \coloneqq +\infty$.
Furthermore, we use the convention that $\frac{a}{+\infty} = 0$ for every $a\in \bbR$.

\subsection{Positive semidefinite cones and their faces}
Similar to \cite{Lourenco2021,LLP2023}, our development heavily relies on the facial structure of the underlying cone.
Here we present a brief review of positive semidefinite cones and their faces.

Recall that $\calS_+^n$ denotes the cone of positive semidefinite matrices in $\calS^n$, which we refer to as the positive semidefinite cone.
The formula for the distance from a point $\bm{X} \in \calS^n$ to the cone $\calS_+^n$ is well known~\cite[page~\mbox{399}]{BV2004} and is given by
\begin{equation}
\dist(\bm{X},\calS_+^n) = \sqrt{\sum_{i=1}^n \max\{-\lambda_i,0\}^2}, \label{eq:dist_PSDcone}
\end{equation}
where $\lambda_1,\dots,\lambda_n$ are the eigenvalues of $\bm{X}$.
Moreover, it is easy to verify that, for any matrix $\bm{X}\in \calS^n$ and any principal submatrix $\widehat{\bm{X}} \in \calS^d$, the following inequality holds:
\begin{equation}
\dist(\widehat{\bm{X}},\calS_+^d) \le \dist(\bm{X},\calS_+^n). \label{eq:ineq_dist_PSDcone}
\end{equation}

We recall that a nonempty (closed) convex subcone $\calF$ of $\calS_+^n$ is called a \emph{face} if for all $\bm{A},\bm{B}\in \calS_+^n$, we have $\bm{A},\bm{B} \in \calF$ whenever $\bm{A} + \bm{B} \in \calF$ holds.
For a face $\calF$ of $\calS_+^n$, we define $\calF^*$ as the \emph{dual cone} of $\calF$, i.e., the set of $\bm{A}\in \calS^n$ such that $\langle \bm{A},\bm{B}\rangle \ge 0$ for all $\bm{B} \in \calF$.
A \emph{chain of faces} of $\calS_+^n$ is a sequence of faces $\calF_1,\dots,\calF_{\ell}$ of $\calS_+^n$, denoted by $\calF_{\ell} \subsetneq \cdots \subsetneq \calF_1$, such that the strict inclusion $\calF_{i+1} \subsetneq \calF_i$ holds for every $i = 1,\dots,\ell-1$.

A set of the form $\calS_+^n \cap \{\bm{Z}\}^\perp$ for some $\bm{Z} \in \calS_+^n$ is a face of the cone $\calS_+^n$, and we say that the face is \emph{exposed} by the matrix $\bm{Z}$.
Every face of the positive semidefinite cone is linearly isomorphic to a positive semidefinite cone of smaller order; see \cite[Lemma~4]{BC1975}, \cite[Theorem~3.6]{HW1987}, and \cite[Example~3.2.2]{Pataki2000}.
In particular, for a face $\calF$ of $\calS_+^n$, there exist an integer $d$ with $0\le d\le n$ and an orthogonal matrix $\bm{P}$ of order $n$ such that
\begin{equation*}
\calF = \bm{P} (\calS_+^d \oplus \{0\}^{n-d}) \bm{P}^\top.
\end{equation*}
When $d\ge 1$, by dropping the last $(n-d)$ columns of $\bm{P}$, we can also represent the face $\calF$ as $\bm{Q}\calS_+^d\bm{Q}^\top$, where $\bm{Q} \in \bbR^{n \times d}$ is a matrix whose column vectors are orthonormal.
When $d = 0$, the face $\calF$ is the singleton $\{\bm{O}\}$.

Recall that $\Feas(\calV,\calS_+^n)$ denotes the problem of finding an element of the intersection of an affine space $\calV$ in $\calS^n$ and the positive semidefinite cone $\calS_+^n$.
The following result, which was established in \cite[Proposition~5]{Lourenco2021}, describes a chain of faces associated with a feasible semidefinite feasibility problem.
\begin{theorem}\label{thm:chain_faces}
For any feasible semidefinite feasibility problem $\Feas(\calV,\calS_+^n)$, there exists a chain of faces of $\calS_+^n$, denoted by
\begin{equation}
\calF_{\ell} \subsetneq \cdots \subsetneq \calF_1 = \calS_+^n, \label{eq:chain_faces}
\end{equation}
together with matrices $\bm{Z}_1,\dots,\bm{Z}_{\ell-1} \in \calS^n$ such that the following conditions hold:
\begin{enumerate}[label=(\roman*), ref=\roman*, font=\upshape]
\item $\bm{Z}_i \in \calF_i^* \cap \calV^\perp$ and  $\calF_{i+1} = \calF_i \cap \{\bm{Z}_i\}^\perp$ hold for each $i = 1,\dots,\ell - 1$. \label{enum:cond_Z}
\item The face $\calF_{\ell}$ is polyhedral or the set $\calV \cap \ri\calF_{\ell}$ is nonempty. \label{enum:PPS}
\end{enumerate}
Moreover, any such chain necessarily satisfies $\ell \le n$.
\end{theorem}

The inequality $\ell \le n$ given in Theorem~\ref{thm:chain_faces} follows from \cite[Example~1]{LMT2018}.
For the chain of faces in \eqref{eq:chain_faces}, it follows from \eqref{enum:cond_Z} of Theorem~\ref{thm:chain_faces} that $\calV \cap \calS_+^n = \calV \cap \calF_{\ell}$.
This implies that, after the facial reduction steps, we can find a face $\calF_{\ell}$ such that the {constraint qualification} shown in \eqref{enum:PPS} of Theorem~\ref{thm:chain_faces} is satisfied and the set of feasible solutions is invariant.
The condition in \eqref{enum:PPS} is called the \emph{partial polyhedral Slater's (PPS) condition} in \cite{Lourenco2021,LLP2023}.
The \emph{distance to the PPS condition} of the feasibility problem $\Feas(\calV,\calS_+^n)$, denoted by $d_{\rm PPS}(\calV,\calS_+^n)$, is defined as $\ell-1$, where $\ell$ is the length of a shortest chain of the form \eqref{eq:chain_faces} satisfying the conditions in Theorem~\ref{thm:chain_faces}.
By Theorem~\ref{thm:chain_faces}, the quantity $d_{\rm PPS}(\calV,\calS_+^n)$ satisfies $d_{\rm PPS}(\calV,\calS_+^n) \le n-1$.
In addition, it is a variant of the \emph{singularity degree} of the problem $\Feas(\calV,\calS_+^n)$ used in literature.
More precisely, the singularity degree of the problem $\Feas(\calV,\calS_+^n)$, denoted by $d_{\rm s}$, is defined as $\overline{\ell}-1$, where $\overline{\ell}$ is the length of a shortest chain in \eqref{eq:chain_faces} satisfying the condition in \eqref{enum:cond_Z} and the condition that the set $\calV \cap \ri\calF_{\ell}$ is nonempty.
From the definitions, it is clear that $d_{\rm PPS}(\calV,\calS_+^n) \le d_{\rm s}$ always holds.
{Indeed,} for semidefinite feasibility problems with $n\ge 2$, the equation
\begin{equation}
d_{\rm PPS}(\calV,\calS_+^n) = d_{\rm s} \label{eq:dpps_eq_ds}
\end{equation}
holds as shown in Appendix~\ref{apdx:dpps_eq_ds}. {
Using $d_{\rm PPS}$ is typically more advantageous when the underlying cone is a direct product. Nevertheless, we use $d_{\rm PPS}$ here in order to maintain consistency with \cite{Lourenco2021,LLL+2024}}.
For more details {on the distance to the PPS condition and the singularity degree}, see \cite[Section~2.4.1]{Lourenco2021} and references therein.

For each $i = 1,\dots,\ell-1$, the matrix $\bm{Z}_i$ shown above can be taken to have norm $1$.
By the strict inclusion $\calF_{i+1} \subsetneq \calF_i$, we have $\bm{Z}_i \neq \bm{O}$, i.e., $\norm{\bm{Z}_i}_{\rm F} \neq 0$. 
Then it follows from the condition in \eqref{enum:cond_Z} of Theorem~\ref{thm:chain_faces} that $\bm{Z}_i / \norm{\bm{Z}_i}_{\rm F} \in \calF_i^* \cap \calV^\perp$ and $\calF_{i+1} = \calF_i \cap \{\bm{Z}_i / \norm{\bm{Z}_i}_{\rm F}\}^\perp$ hold.
Therefore, by normalizing if necessary, we may assume that $\norm{\bm{Z}_i}_{\rm F} = 1$ without loss of generality.

Let $\calF$ be a face of $\calS_+^n$ and let $\bm{Z} \in \calF^*$.
A \emph{one-step facial residual function for $\calF$ and $\bm{Z}$} is a function $\psi_{\calF,\bm{Z}}\colon \bbR_+\times \bbR_+ \to \bbR_+$ satisfying the following three conditions:
\begin{enumerate}[label=(\roman*), ref=\roman*]
\item $\psi_{\calF,\bm{Z}}$ is nondecreasing in each argument.
\item $\psi_{\calF,\bm{Z}}(0,t) = 0$ for any $t \in \bbR_+$.
\item For any $\bm{X} \in \setspan\calF$ and any $\epsilon \in \bbR_+$, if $\dist(\bm{X},\calF) \le \epsilon$ and $\langle \bm{Z},\bm{X}\rangle \le \epsilon$, then we have
\begin{equation*}
\dist(\bm{X},\calF \cap \{\bm{Z}\}^\perp) \le \psi_{\calF,\bm{Z}}(\epsilon,\lVert \bm{X}\rVert_{\rm F}).
\end{equation*}
\end{enumerate}
As mentioned in \cite[Example~3.6]{LLP2023}, there exist nonnegative $\alpha$ and  $\beta$, which depend on $\calF$ and $\bm{Z}$, such that
\begin{equation}
\psi_{\calF,\bm{Z}}(s,t) = \alpha s + \beta\sqrt{st} \label{eq:1-FRF_PSD}
\end{equation}
is a one-step facial residual function for $\calF$ and $\bm{Z}$.
In Section~\ref{subsec:1-FRF}, we will compute the constants $\alpha$ and $\beta$ explicitly, and discuss how these constants relate to the choice of $\bm{Z}$.
Indeed, the one-step facial residual function introduced in \cite{LLP2023} is of the form $\alpha s + \alpha\sqrt{st}$ for some positive $\alpha$, corresponding to the case $\alpha = \beta$. Here, we extend this definition by allowing the two constants to be independent, which enables us to derive sharper estimates later.

\subsection{Error bounds}\label{subsec:error_bound}
In this subsection, we review the results of Lipschitz error bounds for polyhedra, and introduce the notion of asymptotic tightness of error bounds.
Without loss of generality, we only consider the case that polyhedra are subsets of $\bbR^n$.
The celebrated Hoffman's error bound result~\cite{Hoffman1952} states that for a given matrix $\bm{A}\in \bbR^{m\times n}$ whose $i$th row is denoted by $\bm{a}_i^\top$, there exists a nonnegative constant $\alpha$ such that for any $\bm{b} = (b_1,\dots,b_m)^\top \in \bm{A}\bbR^n + \bbR_+^m$ (namely $\bm{b}$ such that the linear inequality system $\bm{A}\bm{y} \le \bm{b}$ has a solution), it follows that
\begin{equation}\label{eq:Hoffman0}
\dist(\bm{x}, \{\bm{y}\in \bbR^n \mid \bm{A}\bm{y} \le \bm{b}\}) \le \alpha \sqrt{\sum_{i=1}^m\max\{\bm{a}_i^\top\bm{x} - b_i,0\}^2} \text{ for all $\bm{x} \in \mathbb{R}^n$}.
\end{equation}
We call the smallest nonnegative constant $\alpha$ for which the error bound in \eqref{eq:Hoffman0} holds the \emph{Hoffman constant} and write it as $H(\bm{A})$.

Let $\bm{A}_J$ be the submatrix obtained by extracting the rows of $\bm{A}$ indexed by a subset $J$.
In addition, we use $\bbR^J$ and $\bbR_+^J$ to denote the space of real vectors whose entries are indexed by the elements in $J$ and the set of nonnegative vectors in $\bbR^J$, respectively.
It is shown in \cite[Proposition~2]{PVZ2021} that if we set
\begin{equation}
H(\bm{A}) = \max_{\substack{J \subseteq \{1,\dots,m\},\\ \bm{A}_J\bbR^n + \bbR_+^J = \bbR^J}} \frac{1}{\displaystyle\min_{\substack{\bm{v}\in \bbR_+^J,\\ \norm{\bm{v}}_2 = 1}} \norm{(\bm{A}_J)^\top \bm{v}}_2}, \label{eq:Hoffman_characterization}
\end{equation}
with the convention that the fraction inside the maximum is zero when $J = \emptyset$, then the constant $H(\bm{A})$ satisfies the error bound in \eqref{eq:Hoffman0} and it is indeed the smallest constant for which the error bound in \eqref{eq:Hoffman0} holds.
The computation of the Hoffman constant shown in \eqref{eq:Hoffman_characterization} is also discussed in \cite[Section~3]{PVZ2021}.
Although various explicit representations of a Hoffman constant, for example, \eqref{eq:Hoffman_characterization} and those provided in \cite{KT1995,GHR1995}, are known, calculating the exact value of the Hoffman constant typically involves significant computational difficulty as discussed in \cite{Pena2024}. 
Meanwhile, some studies discuss algorithmic aspects for calculating them, and effective upper bounds can be found in \cite{KT1995,PVZ20XX_An,PVZ2021,Pena2024}.

For linear equality systems, effective upper bounds for Hoffman constants can be computed via the singular values of the associated coefficient matrix. 
For $\bm{A} \in \bbR^{m\times n}$, we define $\sigma_{\rm min}^+(\bm{A})$ as the smallest positive singular value of $\bm{A}$ when $\bm{A}$ is nonzero, and set $\sigma_{\rm min}^+(\bm{A}) \coloneqq +\infty$ when $\bm{A}$ is the zero matrix.
Then for any $\bm{b}\in \bbR^m$ such that the linear equality system $\bm{A}\bm{y} = \bm{b}$ has a solution, it follows that
\begin{equation}
\dist(\bm{x},\{\bm{y}\in \bbR^n \mid \bm{A}\bm{y} = \bm{b}\}) \le \frac{1}{\sigma_{\rm min}^+(\bm{A})}\norm{\bm{A}\bm{x}-\bm{b}}_{\rm 2} \text{ for all $\bm{x}\in \bbR^n$}.\label{eq:Hoffman_eq}
\end{equation}
This inequality follows from the result on the best approximate solution of a linear equality system~\cite{Penrose1956} and the description of the singular values of a generalized inverse; see also \cite[Section~4]{NNG2019}.

It is well known that a global Lipschitz error bound for the intersection of polyhedral sets always holds, e.g., see \cite[Corollary~5.26]{BB1996}.
Let $C_1$ and $C_2$ be polyhedral sets in $\bbR^n$ and suppose that $C_1 \cap C_2 \neq \emptyset$.
Then there exists a nonnegative constant $\alpha$ such that
\begin{equation}
\dist(\bm{x},C_1\cap C_2) \le \alpha\max\{\dist(\bm{x},C_1),\dist(\bm{x},C_2)\} \text{ for all $\bm{x} \in \bbR^n$}.\label{eq:EB_polysets}
\end{equation}
We write $\eta(C_1,C_2)$ for the smallest nonnegative constant $\alpha$ for which the error bound in \eqref{eq:EB_polysets} holds.
Although computing $\eta(C_1,C_2)$ is also computationally challenging, we can obtain computable upper bounds on $\eta(C_1,C_2)$ by exploiting algebraic descriptions of $C_1$ and $C_2$.
We provide a proof of the following lemma in Appendix~\ref{apdx:proof_lemma}.

\begin{lemma}\label{lem:Hoffman}
Let $C_1$ and $C_2$ be polyhedral sets in $\bbR^n$ with $C_1 \cap C_2 \neq \emptyset$. Let $\eta(C_1,C_2)$ be the smallest nonnegative constant $\alpha$ for which the error bound in \eqref{eq:EB_polysets} holds.
\begin{enumerate}[label=(\roman*), ref=\roman*, font=\upshape]
\item Suppose that for each $i = 1,2$, the polyhedral set $C_i$ is described by a linear inequality system $\bm{A}_i\bm{x} \le \bm{b}_i$ for some $\bm{A}_i \in \bbR^{m_i \times n}$ and $\bm{b}_i \in \bbR^{m_i}$.
Let $\bm{A}$ be the matrix obtained by concatenating the matrix $\bm{A}_1$ with the matrix $\bm{A}_2$.
Then we have $\eta(C_1,C_2) \le H(\bm{A})\norm{\bm{A}}_{\rm F}$. \label{enum:eta_inequality}
\item Suppose that for each $i = 1,2$, the polyhedral set $C_i$ is described by a linear equality system $\bm{A}_i\bm{x} = \bm{b}_i$ for some $\bm{A}_i \in \bbR^{m_i \times n}$ and $\bm{b}_i \in \bbR^{m_i}$.
Let $\bm{A}$ be the matrix constructed in a way as in \eqref{enum:eta_inequality}.
Then we have $\eta(C_1,C_2) \le \norm{\bm{A}}_{\rm F} / \sigma_{\rm min}^+(\bm{A})$.\label{enum:eta_equality}
\end{enumerate}
\end{lemma}

\begin{remark}\label{rem:eta}
We note that $\eta(C_1,C_2) \ge 1$ unless $C_1 \cap C_2 = \bbR^n$, i.e., $C_1 = C_2 = \bbR^n$.
Moreover, we also note that $\eta(C_1,C_2) = 1$ if one of the two sets $C_1$ and $C_2$ is included in the other set.
It follows that $\eta(C_1,C_2) = 0$ if $C_1 \cap C_2 = \bbR^n$.
\end{remark}

The error bounds studied in this paper apply to semidefinite feasibility problems of arbitrary matrix order $n$. Consequently, the notion of tightness we adopt below is inherently asymptotic, involving sequences of problems defined over matrix spaces of increasing dimension. Our objective is to characterize the behavior of the associated error bound inequalities as the matrix order $n$ tends to infinity.

More precisely, let $\calV_n$ be an affine space in $\calS^n$.
We consider a family of semidefinite feasibility problems $\Feas(\calV_n,\mathcal{S}^n_+)$ in which the order of the matrix (and so, the dimension of the underlying matrix space) varies.
The corresponding backward error is $\mathcal{E}_{\rm b}^{(n)}(\bm{X})\coloneqq\max\{\dist(\bm{X},\calV_n), \dist(\bm{X},\mathcal{S}^n_+)\}$.
For each $n\in \bbZ_{\ge 1}$, define the residual function $g_n$ by 
\begin{equation*}
g_n(\bm{X}) \coloneqq
\sum_{j=0}^{\ell_n} \mu_{j,n}(\|\bm{X}\|_{\rm F})\,\calE_{\rm b}^{(n)}(\bm{X})^{\gamma_{j,n}},
\end{equation*}
where $\ell_n \in \bbZ_{\ge 0}$, $\mu_{j,n}\colon\mathbb{R}_+ \to \mathbb{R}_{+}$ is a nondecreasing function for each $j=0,\ldots,\ell_n$, and $\gamma_{j,n}>0$.
Suppose that there exist a positive integer $n_0$ and a set $\calQ_n$ in $\calS^n$ parametrized by every $n\in \bbN_{\ge n_0}$ such that the following error bounds hold:
\begin{equation}\label{eq:eb_SFP}
\dist(\bm{X},\calV_n \cap \mathcal{S}^n_+) \le g_n(\bm{X}) \text{ for all $n\in \bbN_{\ge n_0}$ and for all $\bm{X} \in \calQ_n$}.
\end{equation}
Under this assumption, 
$g_n$ is a nonnegative function such that $g_n(\bm{X})=0$ if and only if $\bm{X} \in \mathcal{V}_n \cap \mathcal{S}^n_+$. 
As we will see later in Corollary~\ref{cor:error_bound_constant}, error bounds as in   \eqref{eq:eb_SFP} always hold for {semidefinite feasibility problems}. With that, we have the following definition of tightness.

\begin{definition}[\textbf{Asymptotic tightness for error bounds}]\label{def:tight_eb}
We say that the error bound in \eqref{eq:eb_SFP} is \emph{asymptotically tight up to a dimension-free constant $C_0 \in (0,1]$} if there exist a sequence $(n_k)$ in $\bbZ_{\ge n_0}$ satisfying $\lim_{k\to \infty} n_k = \infty$ and a sequence $(\bm{X}_k)$ with  $\bm{X}_k \in \calQ_{n_k} \backslash (\calV_{n_k} \cap \mathcal{S}^{n_k}_+)$ for every $k$ such that
\begin{equation*}
\liminf_{k \to \infty}\frac{\dist(\bm{X}_k,\calV_{n_k} \cap \mathcal{S}^{n_k}_+)}{g_{n_k}(\bm{X}_k)} \ge C_0.
\end{equation*}
Moreover, we say that the error bound in \eqref{eq:eb_SFP} is \emph{asymptotically tight up to a dimension-free constant} if there exists a constant $C_0 \in (0,1]$, independent of the dimension, such that the bound is asymptotically tight with constant $C_0$.
\end{definition}

\begin{remark}[\textbf{Intuition behind Definition~\ref{def:tight_eb}}]\label{rem:int}
If the error bound in \eqref{eq:eb_SFP} is asymptotically tight up to a dimension-free constant $C_0 \in (0,1]$, this then implies that for any constant $C \in (0,C_0)$, the following inequality cannot be true:
\begin{equation*}
\dist(\bm{X},\calV_n \cap  \mathcal{S}^{n}_+) \le C g_n(\bm{X}) \text{ for all  $n\in \bbZ_{\ge n_0}$  and for all $\bm{X} \in \calQ_n$}.
\end{equation*}
In other words, in this case, for any $C \in (0,C_0)$, an improved version of error bound results with the residual function $g_n$ scaling with a constant contracting factor $C$ is not possible. 
\end{remark}

\section{Radial-type error bounds for semidefinite feasibility problems}\label{sec:general}
In this section, we derive a qualitative radial-type H\"{o}lder error bound for the semidefinite feasibility problem $\Feas(\calV,\calS_+^n)$ \emph{without assuming any constraint qualifications}.
Our qualitative explicit estimate relies on the facial reduction algorithm.
For simplicity, we write $d^*$ for $d_{\rm PPS}(\calV,\calS_+^n)$, the distance to the PPS condition of the feasibility problem $\Feas(\calV,\calS_+^n)$.
Then, from Theorem~\ref{thm:chain_faces} there exists a chain
\begin{equation}\label{eq:chain}
\calF_{d^*+1} \subsetneq \cdots \subsetneq \calF_1
\end{equation}
 of faces of $\calS_+^n$, such that $\calF_1 = \calS_+^n$, $\bm{Z}_i \in \calF_i^* \cap \calV^\perp$ and $\calF_{i+1} = \calF_i \cap \{\bm{Z}_i\}^\perp$ for each $i = 1,\dots,d^*$, and the problem $\Feas(\calV,\calF_{d^*+1})$ satisfies the PPS condition. 
As mentioned in Section~\ref{subsec:error_bound}, we may assume that $\lVert\bm{Z}_i\rVert_{\rm F} = 1$ for all $i = 1,\dots,d^*$.
Throughout this section, we assume that the $\calF_i$'s and $\bm{Z}_i$'s are given in \eqref{eq:chain}, and we shall fix them and utilize them to derive the desired radial-type error bound.
 
The underlying ideas of the derivation are intuitive, though the argument is somewhat technical.
To streamline the exposition, we first present a roadmap of the proof.
\emph{To derive a qualitative radial-type H\"{o}lder error bound, we first consider the facially reduced problem with minimal face, $\Feas(\calV,\calF_{d^*+1})$, and compute an explicit modulus of a Lipschitz error bound for this problem in Section~\ref{subsec:eb_reduced}.
Then we compute the constants associated to a one-step facial residual function for $\calF_i$ and $\bm{Z}_i$ when one moves from a lower layer face $\calF_i$ to the higher layer face $\calF_{i+1}$ for every $i = 1,\dots,d^*$ in Section~\ref{subsec:1-FRF}.
Combining these results, we then obtain the final qualitative radial-type H\"{o}lder error bound for the problem $\Feas(\calV,\calS_+^n)$ in Section~\ref{subsec:eb_originalSDP}}.

\subsection{Radial modulus function for facially reduced problem with minimal face}\label{subsec:eb_reduced}
In this subsection, we consider the problem $\Feas(\calV,\calF_{d^*+1})$, the feasibility problem with respect to the subspace $\calV$ and the minimal face $\calF_{d^*+1}$.
Our main aim in this subsection is to establish an explicit form of a radial modulus function $\kappa$ (see Theorem~\ref{thm:kappaB} later) such that 
\begin{equation}
\dist(\bm{X},\calV \cap \calF_{d^*+1}) \le \kappa(\norm{\bm{X}}_{\rm F})\max\{\dist(\bm{X},\calV),\dist(\bm{X},\calF_{d^*+1})\}  \text{ for all $\bm{X} \in \calS^n$}. \label{eq:eb_reduced}
\end{equation}

Recall that the problem $\Feas(\calV,\calF_{d^*+1})$ satisfies the PPS condition, so $\calV \cap \ri\calF_{d^*+1} \neq \emptyset$ holds or $\calF_{d^*+1}$ is polyhedral.
First, we consider the case where $\calV \cap \ri\calF_{d^*+1} \neq \emptyset$.
Direct verification shows that the following equations hold (this can also be seen by the proof of \cite[Lemma~3.2]{BT2003}):
\begin{align}
(\setspan\calF_{d^*+1}) \cap (\calF_{d^*+1} + \calF_{d^*+1}^\perp) &= \calF_{d^*+1}, \label{eq:Fl_span}\\
(\setspan\calF_{d^*+1}) \cap \setint(\calF_{d^*+1} + \calF_{d^*+1}^\perp) &= \ri\calF_{d^*+1}. \label{eq:riFl_span}
\end{align}

\begin{lemma}\label{lem:radius_in_Fl_plus_Flperp}
Let $\bm{U} \in \setint(\calF_{d^*+1} + \calF_{d^*+1}^\perp)$.
Then we have
\begin{equation}
\lambda_{\rm min}^+(P_{\calF_{d^*+1}}(\bm{U})) = \max\{\epsilon \mid \mathbb{B}(\bm{U},\epsilon) \subseteq \calF_{d^*+1} + \calF_{d^*+1}^\perp\} \in (0,+\infty]. \label{eq:lambda+}
\end{equation}
\end{lemma}

\begin{proof}
When $\calF_{d^* + 1} = \{\bm{O}\}$, both the left- and right-hand sides of \eqref{eq:lambda+} are $+\infty$, and the desired equality holds.
In what follows, we assume that $\calF_{d^* + 1} \neq \{\bm{O}\}$.
Since $\calF_{d^*+1}$ is a nonzero face of $\calS_+^n$, there exist an orthogonal matrix $\bm{P}$ of order $n$ and a positive integer $r$ such that
\begin{equation}
\calF_{d^*+1} = \bm{P}(\calS_+^r \oplus \{0\}^{n-r})\bm{P}^\top. \label{eq:face_ell}
\end{equation}
Then we see that
\begin{align}
\calF_{d^*+1}^\perp &= \bm{P}\left\{
\begin{pmatrix}
\bm{O} & \bm{A}_{[12]}\\
\bm{A}_{[12]}^\top & \bm{A}_{[22]}
\end{pmatrix} \relmiddle| \bm{A}_{[12]} \in \bbR^{r\times (n-r)},\ \bm{A}_{[22]} \in \calS^{n-r}\right\}\bm{P}^\top, \nonumber\\
\calF_{d^*+1} + \calF_{d^*+1}^\perp &= \bm{P}\left\{
\begin{pmatrix}
\bm{A}_{[11]} & \bm{A}_{[12]}\\
\bm{A}_{[12]}^\top & \bm{A}_{[22]}
\end{pmatrix} \relmiddle| \bm{A}_{[11]} \in \calS_+^r,\ \bm{A}_{[12]} \in \bbR^{r\times (n-r)},\ \bm{A}_{[22]} \in \calS^{n-r}\right\}\bm{P}^\top. \label{eq:F+Fperp}
\end{align}
By $\bm{U} \in \setint(\calF_{d^*+1} + \calF_{d^*+1}^\perp)$, there exist $\bm{A}_{[11]} \in \setint\calS_+^r$, $\bm{A}_{[12]} \in \bbR^{r\times (n-r)}$, and $\bm{A}_{[22]} \in \calS^{n-r}$ such that
\begin{equation*}
\bm{U} = \bm{P}\begin{pmatrix}
\bm{A}_{[11]} & \bm{A}_{[12]}\\
\bm{A}_{[12]}^\top & \bm{A}_{[22]}
\end{pmatrix}\bm{P}^\top.
\end{equation*}
Then we have $P_{\calF_{d^*+1}}(\bm{U}) = \bm{P}\Diag(\bm{A}_{[11]},\bm{O})\bm{P}^\top$ and $\lambda_{\rm min}^+(P_{\calF_{d^*+1}}(\bm{U})) = \lambda_{\rm min}(\bm{A}_{[11]}) \in (0,+\infty)$.

First, we show the inclusion $\bbB(\bm{U}, \lambda_{\rm min}(\bm{A}_{[11]})) \subseteq \calF_{d^*+1} + \calF_{d^*+1}^\perp$.
Suppose that $\bm{X} \in \calS^n$ satisfies $\|\bm{U} - \bm{X}\|_{\rm F} \le \lambda_{\rm min}(\bm{A}_{[11]})$.
We divide $\bm{P}^\top\bm{X}\bm{P}$ into
\begin{equation*}
\bm{P}^\top\bm{X}\bm{P} = \begin{pmatrix}
\bm{Z}_{[11]} & \bm{Z}_{[12]}\\
\bm{Z}_{[12]}^\top & \bm{Z}_{[22]}
\end{pmatrix},
\end{equation*}
where $\bm{Z}_{[11]} \in \calS^r$, $\bm{Z}_{[12]} \in \bbR^{r\times (n-r)}$, and $\bm{Z}_{[22]} \in \calS^{n-r}$.
Then, since
\begin{align*}
\lVert\bm{A}_{[11]} - \bm{Z}_{[11]}\rVert_{\rm F} & \le \left\lVert \begin{pmatrix}
\bm{A}_{[11]} & \bm{A}_{[12]}\\
\bm{A}_{[12]}^\top & \bm{A}_{[22]}
\end{pmatrix} - \begin{pmatrix}
\bm{Z}_{[11]} & \bm{Z}_{[12]}\\
\bm{Z}_{[12]}^\top & \bm{Z}_{[22]}
\end{pmatrix}\right\rVert_{\rm F} 
= \lVert\bm{U}-\bm{X}\rVert_{\rm F}
\le \lambda_{\rm min}(\bm{A}_{[11]}),
\end{align*}
we have $\lvert\lambda_{\rm max}(\bm{A}_{[11]} - \bm{Z}_{[11]})\rvert \le \lambda_{\rm min}(\bm{A}_{[11]})$.
For any $\bm{x} \in \bbR^r$ with $\lVert\bm{x}\rVert_2 = 1$, we see that
\begin{equation*}
\bm{x}^\top\bm{Z}_{[11]}\bm{x} = \bm{x}^\top\bm{A}_{[11]}\bm{x} - \bm{x}^\top(\bm{A}_{[11]} - \bm{Z}_{[11]})\bm{x} \ge \lambda_{\rm min}(\bm{A}_{[11]}) - \lambda_{\rm max}(\bm{A}_{[11]} - \bm{Z}_{[11]}) \ge 0,
\end{equation*}
from which we obtain $\bm{Z}_{[11]} \in \calS_+^r$.
Therefore, from \eqref{eq:F+Fperp}, we have $\bm{X} \in \calF_{d^*+1} + \calF_{d^*+1}^\perp$.

Next, we show the inclusion $\bbB(\bm{U},\epsilon) \subseteq \calF_{d^*+1} + \calF_{d^*+1}^\perp$ does not hold for any $\epsilon$ such that $\epsilon > \lambda_{\rm min}(\bm{A}_{[11]})$.
Let $\bm{Q}\Diag(\lambda_1,\dots,\lambda_r)\bm{Q}^\top$ be an eigendecomposition of $\bm{A}_{[11]}$, where $\bm{Q}$ is an orthogonal matrix of order $r$ and $\lambda_1,\dots,\lambda_r$ are the eigenvalues of $\bm{A}_{[11]}$ satisfying $\lambda_1 \ge \dots \ge \lambda_r = \lambda_{\rm min}(\bm{A}_{[11]})$.
We define
\begin{align*}
\bm{R} &\coloneqq \bm{P} \Diag(\bm{Q}\Diag(0,\dots,0,\epsilon)\bm{Q}^\top,\bm{O})\bm{P}^\top,\\
\bm{X} &\coloneqq \bm{U} - \bm{R} = \bm{P} \begin{pmatrix}
\bm{Q}\Diag(\lambda_1,\dots,\lambda_{r-1},\lambda_r-\epsilon)\bm{Q}^\top & \bm{A}_{[12]} \\
\bm{A}_{[12]}^\top & \bm{A}_{[22]}
\end{pmatrix}\bm{P}^\top.
\end{align*}
On the one hand, we have $\|\bm{U}-\bm{X}\|_{\rm F} = \|\bm{R}\|_{\rm F} = \epsilon$, which implies that $\bm{X}\in \bbB(\bm{U},\epsilon)$.
On the other hand, since $\lambda_r - \epsilon < 0$, the matrix $\bm{Q}\Diag(\lambda_1,\dots,\lambda_{r-1},\lambda_r-\epsilon)\bm{Q}^\top$ is not positive semidefinite and $\bm{X}$ does not belong to $\calF_{d^*+1} + \calF_{d^*+1}^\perp$.
\end{proof}

In the next lemma, we show that the distance from a given matrix $\bm{X}$ to $\calV \cap \calF_{d^*+1}$ can be bounded by the maximum distances to the two sets $\calV \cap \setspan\calF_{d^*+1}$ and $\calF_{d^*+1} + \calF_{d^*+1}^\perp$.

\begin{lemma}\label{lem:dist_X_VcapFl}
Suppose that $\calV \cap \ri\calF_{d^*+1} \neq \emptyset$.
For every $\bm{X} \in \calS^n$, we have
\begin{equation}
\dist(\bm{X},\calV \cap \calF_{d^*+1}) \le \left(1 + 2\inf_{\bm{U}\in \calV \cap \ri\calF_{d^*+1}}\frac{\lVert\bm{U} - \bm{X}\rVert_{\rm F}}{\lambda_{\rm min}^+(\bm{U})}\right)\max\{\dist(\bm{X},\calV \cap \setspan\calF_{d^*+1}),\dist(\bm{X},\calF_{d^*+1} + \calF_{d^*+1}^\perp)\}. \label{eq:eb_V_cap_F}
\end{equation}
\end{lemma}

\begin{proof}
When $\calF_{d^*+1} = \{\bm{O}\}$, both the left- and right-hand sides of \eqref{eq:eb_V_cap_F} are $\norm{\bm{X}}_{\rm F}$, and the desired inequality holds.
In what follows, we assume that $\calF_{d^*+1} \neq \{\bm{O}\}$.
For convenience, we define
$\delta(\bm{X}) \coloneqq \max\{\dist(\bm{X},\calV\cap\setspan\calF_{d^*+1}),\dist(\bm{X},\calF_{d^*+1} + \calF_{d^*+1}^\perp)\}$.
Let $\bm{U} \in \calV \cap \ri\calF_{d^*+1}$ be arbitrary and let
\begin{equation}
\bm{V} \coloneqq \bm{U} + \frac{\lambda_{\rm min}^+(\bm{U})}{2\delta(\bm{X})}(P_{\calV \cap \setspan\calF_{d^*+1}}(\bm{X}) - P_{\calF_{d^*+1} + \calF_{d^*+1}^\perp}(P_{\calV \cap \setspan\calF_{d^*+1}}(\bm{X}))). \label{eq:def_V}
\end{equation}
We note that $\lambda_{\rm min}^+(\bm{U})$ is a real number since $\bm{U}$ is a nonzero positive semidefinite matrix.
Then $\norm{\bm{V} - \bm{U}}_{\rm F} = \frac{\lambda_{\rm min}^+(\bm{U})}{2\delta(\bm{X})} \dist(P_{\calV \cap \setspan\calF_{d^*+1}}(\bm{X}),\calF_{d^*+1} + \calF_{d^*+1}^\perp)$, and hence, the matrix $\bm{V}$ satisfies
\begin{align*}
\norm{\bm{V} - \bm{U}}_{\rm F}
&\le \frac{\lambda_{\rm min}^+(\bm{U})}{2\delta(\bm{X})} \norm{P_{\calV \cap \setspan\calF_{d^*+1}}(\bm{X}) - P_{\calF_{d^*+1} + \calF_{d^*+1}^\perp}(\bm{X})}_{\rm F}\\
&\le \frac{\lambda_{\rm min}^+(\bm{U})}{2\delta(\bm{X})}(\dist(\bm{X},\calV \cap \setspan\calF_{d^*+1}) + \dist(\bm{X},\calF_{d^*+1} + \calF_{d^*+1}^\perp)) \\
&\le \lambda_{\rm min}^+(\bm{U})\\
&= \lambda_{\rm min}^+(P_{\calF_{d^*+1}}(\bm{U})),
\end{align*}
i.e., $\bm{V} \in \bbB(\bm{U},\lambda_{\rm min}^+(P_{\calF_{d^*+1}}(\bm{U})))$, where the first inequality follows from $P_{\calF_{d^*+1} + \calF_{d^*+1}^\perp}(\bm{X}) \in \calF_{d^*+1} + \calF_{d^*+1}^\perp$, the second inequality follows from the triangle inequality, the third inequality follows from the definition of $\delta(\bm{X})$, and the last equality follows from $\bm{U} \in \calF_{d^*+1}$.
In addition, it follows from \eqref{eq:riFl_span} that $\bm{U} \in \setint(\calF_{d^*+1} + \calF_{d^*+1}^\perp)$.
Therefore, from Lemma~\ref{lem:radius_in_Fl_plus_Flperp}, we see that $\bm{V} \in \calF_{d^*+1} + \calF_{d^*+1}^\perp$.

We let
\begin{equation*}
\bm{Y} \coloneqq \frac{2\delta(\bm{X})}{\lambda_{\rm min}^+(\bm{U}) + 2\delta(\bm{X})}\bm{V} + \frac{\lambda_{\rm min}^+(\bm{U})}{\lambda_{\rm min}^+(\bm{U}) + 2\delta(\bm{X})}P_{\calF_{d^*+1} + \calF_{d^*+1}^\perp}(P_{\calV \cap \setspan\calF_{d^*+1}}(\bm{X})).
\end{equation*}
The matrix $\bm{Y}$ belongs to $\calF_{d^*+1} + \calF_{d^*+1}^\perp$ since it is a convex combination of the two elements in $\calF_{d^*+1} + \calF_{d^*+1}^\perp$.
In addition, it also belongs to $\calV \cap \setspan\calF_{d^*+1}$ since by substituting $\bm{V}$ defined in \eqref{eq:def_V} into the definition of $\bm{Y}$, the matrix $\bm{Y}$ can be also represented as
\begin{equation}
\bm{Y} = \frac{2\delta(\bm{X})}{\lambda_{\rm min}^+(\bm{U}) + 2\delta(\bm{X})}\bm{U} + \frac{\lambda_{\rm min}^+(\bm{U})}{\lambda_{\rm min}^+(\bm{U}) + 2\delta(\bm{X})}P_{\calV \cap \setspan\calF_{d^*+1}}(\bm{X}), \label{eq:Y_conv_U_ProjX}
\end{equation}
which is a linear combination of the two elements in $\calV \cap \setspan\calF_{d^*+1}$.
Then it follows from \eqref{eq:Fl_span} that $\bm{Y} \in \calV \cap \calF_{d^*+1}$, and so we have
\begin{align}
\dist(\bm{X},\calV \cap \calF_{d^*+1}) &\le \norm{\bm{X} - \bm{Y}}_{\rm F} \nonumber\\
&\le \dist(\bm{X},\calV \cap \setspan\calF_{d^*+1}) + \norm{\bm{Y} - P_{\calV \cap \setspan\calF_{d^*+1}}(\bm{X})}_{\rm F}\nonumber\\
&\le \delta(\bm{X}) + \frac{2\delta(\bm{X})}{\lambda_{\rm min}^+(\bm{U})}\norm{
\bm{U} - P_{\calV \cap \setspan\calF_{d^*+1}}(\bm{X})}_{\rm F}\nonumber\\
&\le \left(1 + \frac{2\lVert\bm{U} - \bm{X}\rVert_{\rm F}}{\lambda_{\rm min}^+(\bm{U})}\right)\delta(\bm{X}), \label{eq:dist_X_VcapFl_ub_noinf}
\end{align}
where the second inequality follows from the triangle inequality, the third inequality follows from the definition of $\delta(\bm{X})$ and \eqref{eq:Y_conv_U_ProjX}, and the fourth inequality follows from $\bm{U} \in \calV \cap \setspan\calF_{d^*+1}$ and the nonexpansiveness of the projection.
Since $\bm{U} \in \calV \cap \ri\calF_{d^*+1}$ is arbitrary, by taking the infimum over $\bm{U} \in \calV \cap \ri\calF_{d^*+1}$ in \eqref{eq:dist_X_VcapFl_ub_noinf}, we obtain the desired result.
\end{proof}

The proof of Lemma~\ref{lem:dist_X_VcapFl} follows the general strategy of \cite[Lemma~3.1]{BT2003}, which addresses convex feasibility problems. In \cite{BT2003}, one fixes a matrix $\bm{U} \in \setint(\calF_{d^*+1} + \calF_{d^*+1}^\perp)$ and chooses $\epsilon > 0$ such that $\mathbb{B}(\bm{U},\epsilon) \subseteq \calF_{d^*+1} + \calF_{d^*+1}^\perp$, leading to an error bound that depends on this particular choice.
In our setting, however, Lemma~\ref{lem:radius_in_Fl_plus_Flperp} allows us to compute the maximal admissible $\epsilon$ explicitly. Moreover, by taking the infimum over all $\bm{U} \in \setint(\calF_{d^*+1} + \calF_{d^*+1}^\perp)$, we obtain the following sharper estimate.

\begin{proposition}\label{prop:eb_reduced}
Suppose that $\calV \cap \ri\calF_{d^*+1} \neq \emptyset$.
For each nonnegative $\rho$, we let
\begin{equation}
\kappa(\rho) \coloneqq \max\{\eta(\calV,\setspan\calF_{d^*+1}),1\}\theta(\rho), \label{eq:def_kappa_Slater}
\end{equation}
where $\eta$ is given as in Lemma~\ref{lem:Hoffman} and $\theta\colon\mathbb{R}_+ \to [1,+\infty)$ is given by 
\begin{equation}
\theta(\rho) \coloneqq 1 + 2\sup_{\bm{Y} \in \bbB(\bm{O},\rho)}\inf_{\bm{U}\in \calV \cap \ri\calF_{d^*+1}}\frac{\lVert\bm{U} -  \bm{Y}\rVert_{\rm F}}{\lambda_{\rm min}^+(\bm{U})}. \label{eq:thetaB}
\end{equation}
Then we have
\begin{equation*}
\dist(\bm{X},\calV\cap\calF_{d^*+1}) \le \kappa(\|\bm{X}\|_{\rm F})\max\{\dist(\bm{X},\calV),\dist(\bm{X},\calF_{d^*+1})\} \text{ for all $\bm{X} \in \calS^n$}.
\end{equation*}
\end{proposition}

\begin{proof}
Using Lemma~\ref{lem:dist_X_VcapFl}, we have
\begin{align}
\dist(\bm{X},\calV \cap \calF_{d^*+1}) 
&\le  \left(1 + 2\inf_{\bm{U}\in \calV \cap \ri\calF_{d^*+1}}\frac{\lVert\bm{U} - \bm{X}\rVert_{\rm F}}{\lambda_{\rm min}^+(\bm{U})}\right)\max\{\dist(\bm{X},\calV \cap \setspan\calF_{d^*+1}),\dist(\bm{X},\calF_{d^*+1} + \calF_{d^*+1}^\perp)\} \nonumber \\
&\le \theta(\norm{\bm{X}}_{\rm F})\max\{\dist(\bm{X},\calV \cap \setspan\calF_{d^*+1}),\dist(\bm{X},\calF_{d^*+1} + \calF_{d^*+1}^\perp)\}. \label{eq:eb_thetaB}
\end{align}
In addition, since $\calV$ and $\setspan\calF_{d^*+1}$ are polyhedral, it follows from the definition of $\eta(\calV,\setspan\calF_{d^*+1})$ that
\begin{equation}
\dist(\bm{X},\calV \cap \setspan\calF_{d^*+1}) \le \eta(\calV,\setspan\calF_{d^*+1})\max\{\dist(\bm{X},\calV),\dist(\bm{X},\setspan\calF_{d^*+1})\}.\label{eq:eb_V_spanFl}
\end{equation}
From \eqref{eq:eb_thetaB} and \eqref{eq:eb_V_spanFl}, we obtain
\begin{align*}
&\dist(\bm{X},\calV \cap \calF_{d^*+1})\\
\le{}& \max\{\eta(\calV,\setspan\calF_{d^*+1}),1\}\theta(\norm{\bm{X}}_{\rm F})\max\{\dist(\bm{X},\calV),\dist(\bm{X},\setspan\calF_{d^*+1}),\dist(\bm{X},\calF_{d^*+1} + \calF_{d^*+1}^\perp)\}\\
\le{}& \max\{\eta(\calV,\setspan\calF_{d^*+1}),1\}\theta(\norm{\bm{X}}_{\rm F})\max\{\dist(\bm{X},\calV),\dist(\bm{X},\calF_{d^*+1})\} \\
={}& \kappa(\norm{\bm{X}}_{\rm F})\max\{\dist(\bm{X},\calV),\dist(\bm{X},\calF_{d^*+1})\},
\end{align*}
where the second inequality holds since $\calF_{d^*+1}$ is included in both $\setspan\calF_{d^*+1}$ and $\calF_{d^*+1} + \calF_{d^*+1}^\perp$, and the equality follows from the definition of the function $\kappa$.
Therefore, we obtain the desired result.
\end{proof}

In general, it may be difficult to directly compute the exact value of $\theta(\rho)$. 
However, an upper bound can be obtained through the triangle inequality as follows: 
\begin{equation}
\theta(\rho) \le 1 + 2\inf_{\bm{U}\in \calV \cap \ri\calF_{d^*+1}}\frac{\norm{\bm{U}}_{\rm F} + \rho}{\lambda_{\rm min}^+(\bm{U})}. \label{eq:theta_ub}
\end{equation}
In addition, by taking $\bm{U}_0 \in \calV \cap \calF_{d^*+1}$, it follows from \eqref{eq:theta_ub} that
\begin{equation}
\theta(\rho) \le \overline{\theta}(\rho;\bm{U}_0) \coloneqq 1 + 2\frac{\norm{\bm{U}_0}_{\rm F} + \rho}{\lambda_{\rm min}^+(\bm{U}_0)}. \label{eq:bar_theta}
\end{equation}

We now discuss situations where the radial modulus function $\kappa$ can be bounded uniformly by a constant, which leads to global error bounds.
Recall that the function $\kappa$ defined in \eqref{eq:def_kappa_Slater} is composed of the constant $\eta(\calV,\setspan\calF_{d^*+1})$ and the function $\theta$.
Firstly, we note that the constant  $\eta(\calV,\setspan\calF_{d^*+1})$ can be effectively {upper bounded} by computing the smallest positive singular value of a linear mapping determined by the two affine spaces $\calV$ and $\setspan\calF_{d^*+1}$, as shown in \eqref{enum:eta_equality} of Lemma~\ref{lem:Hoffman}. 
By \cite[Corollary~\mbox{11}]{BBL1999}, we can also bound it by a constant using the angle between these two affine spaces.
Secondly, we discuss how $\theta(\rho)$ can be bounded by a more concise quantity.
As shown in the following corollary, the upper bound in \eqref{eq:theta_ub} can be bounded by an explicit quantity that does not depend on $\rho$ when $\mathcal{V}$ is a subspace, so that Proposition~\ref{prop:eb_reduced} implies a global Lipschitz error bound for $\Feas(\calV,\calF_{d^*+1})$.
We define the \emph{condition number} of the set 
$\mathcal{V} \cap \ri\calF_{d^*+1}$ by
\begin{equation*}
\cond(\mathcal{V} \cap \ri\mathcal{F}_{d^*+1}) \coloneqq \inf\left\{ \frac{\lambda_{\max}(\bm{X})}{\lambda_{\rm min}^+(\bm{X})} \relmiddle|  \bm{X} \in \mathcal{V} \cap \ri\calF_{d^*+1}\right\}.
\end{equation*}
Note that various versions of the condition number for a general conic system have been introduced and studied in the literature, see for example \cite{BF2009,PR2020, Renegar1994}.
The one we used here is a variant of those in the literature, restricted to the setting of a semidefinite system, and expressed explicitly in terms of the extreme eigenvalues of the matrices.

\begin{corollary}
Let $\calV$ be a subspace, and let the matrix $\bm{P}$ and the face $\calF_{d^*+1}$ be as in \eqref{eq:face_ell}, in which case $\calF_{d^*+1}$ is linearly isomorphic to $\calS_+^r$.
Suppose that $\calV \cap \ri\calF_{d^*+1} \neq \emptyset$. 
Then, for all $\rho \in \bbR_+$, 
\begin{equation*}
\theta(\rho) \le \theta_0 \coloneqq 1 + 2\sqrt{r} \cond(\mathcal{V} \cap \ri\calF_{d^*+1}).
\end{equation*}
Moreover, if there exists $\alpha > 0$ such that $\bm{P}\Diag(\alpha\bm{I}_r,\bm{O})\bm{P}^\top \in \calV \cap \ri\calF_{d^*+1}$, 
then $\theta_0=1+2 \sqrt{r}$, and so, for every $\rho \in \bbR_+$, we have
\begin{equation}
\theta(\rho) \le 1 + 2\sqrt{r}. \label{eq:theta_tight_ub}
\end{equation}
\end{corollary}

\begin{proof}
To begin with, we show that 
\begin{equation}
\inf_{\bm{U}\in \calV \cap \ri\calF_{d^*+1}}\frac{\norm{\bm{U}}_{\rm F} + \rho}{\lambda_{\rm min}^+(\bm{U})} = \inf_{\bm{U}\in \calV \cap \ri\calF_{d^*+1}}\frac{\norm{\bm{U}}_{\rm F}}{\lambda_{\rm min}^+(\bm{U})}. \label{eq:frac_norm_lambda}
\end{equation}
Since $\rho$ is nonnegative, we see that the right-hand side of \eqref{eq:frac_norm_lambda} is less than or equal to the left-hand side of \eqref{eq:frac_norm_lambda}.
To prove the equality, let $\bm{U}_0 \in \calV \cap \ri\calF_{d^*+1}$ be arbitrary.
The set $\calV \cap \ri\calF_{d^*+1}$ is a cone (not necessarily containing the origin) since $\calV$ is a subspace, so $t\bm{U}_0 \in \calV \cap \ri\calF_{d^*+1}$ holds for all $t > 0$.
Then we have for every $t > 0$
\begin{equation*}
\inf_{\bm{U}\in \calV \cap \ri\calF_{d^*+1}}\frac{\norm{\bm{U}}_{\rm F} + \rho}{\lambda_{\rm min}^+(\bm{U})} \le \frac{\norm{t\bm{U}_0}_{\rm F} + \rho}{\lambda_{\rm min}^+(t\bm{U}_0)} = \frac{\norm{\bm{U}_0}_{\rm F}}{\lambda_{\rm min}^+(\bm{U}_0)} + \frac{\rho}{t\lambda_{\rm min}^+(\bm{U}_0)}.
\end{equation*}
Since $t > 0$ is arbitrary, taking the limit $t \to \infty$, we have
\begin{equation*}
\inf_{\bm{U}\in \calV \cap \ri\calF_{d^*+1}}\frac{\norm{\bm{U}}_{\rm F} + \rho}{\lambda_{\rm min}^+(\bm{U})} \le \frac{\norm{\bm{U}_0}_{\rm F}}{\lambda_{\rm min}^+(\bm{U}_0)}.
\end{equation*}
Since $\bm{U}_0 \in \calV \cap \ri\calF_{d^*+1}$ is also arbitrary, we obtain
\begin{equation*}
\inf_{\bm{U}\in \calV \cap \ri\calF_{d^*+1}}\frac{\norm{\bm{U}}_{\rm F} + \rho}{\lambda_{\rm min}^+(\bm{U})} \le \inf_{\bm{U}\in \calV \cap \ri\calF_{d^*+1}}\frac{\norm{\bm{U}}_{\rm F}}{\lambda_{\rm min}^+(\bm{U})},
\end{equation*}
and the equality in \eqref{eq:frac_norm_lambda} holds.

For any $\bm{U}\in \calV \cap \ri\calF_{d^*+1}$, since $\norm{\bm{U}}_{\rm F}$ is the square root of the sum of squares of the nonzero eigenvalues of $\bm{U}$ and the number of its nonzero eigenvalues is $r$, it follows from the definition of $\cond(\mathcal{V} \cap \ri\calF_{d^*+1})$ that
\begin{equation}
\inf_{\bm{U}\in \calV \cap \ri\calF_{d^*+1}}\frac{\norm{\bm{U}}_{\rm F}}{\lambda_{\rm min}^+(\bm{U})} \le \sqrt{r} \cond(\mathcal{V} \cap \ri\calF_{d^*+1}). \label{eq:ub_by_cond}
\end{equation}
Therefore, we obtain
\begin{equation*}
\theta(\rho) \overset{\scriptsize \text{(a)}}\le 1 + 2\inf_{\bm{U}\in \calV \cap \ri\calF_{d^*+1}}\frac{\norm{\bm{U}}_{\rm F} + \rho}{\lambda_{\rm min}^+(\bm{U})} \overset{\scriptsize \text{(b)}}= 1 + 2\inf_{\bm{U}\in \calV \cap \ri\calF_{d^*+1}}\frac{\norm{\bm{U}}_{\rm F}}{\lambda_{\rm min}^+(\bm{U})} \overset{\scriptsize \text{(c)}}\le 1 + 2\sqrt{r} \cond(\mathcal{V} \cap \ri\calF_{d^*+1}),
\end{equation*}
where we use \eqref{eq:theta_ub} to derive (a), use \eqref{eq:frac_norm_lambda} to derive (b), and use \eqref{eq:ub_by_cond} to derive (c).

Finally, suppose, in addition,  that  there exists $\alpha>0$ such that $\bm{P}\Diag(\alpha\bm{I}_r,\bm{O})\bm{P}^\top \in \calV \cap \ri\calF_{d^*+1}$.
Then the condition number $\cond(\mathcal{V} \cap \ri\calF_{d^*+1})$ equals $1$, and so, the conclusion follows.  
\end{proof}

As we will see in Section~\ref{sec:tight}, the radial modulus function in the error bound for the problem $\Feas(\calV,\calS_+^n)$ cannot, in general, be uniformly bounded by a constant. Moreover, the asymptotic tightness of the derived error bound depends on the quantity $d^*$, which measures the distance to the PPS condition of the semidefinite feasibility problem.

In contrast, the following example demonstrates that one can construct simple instances for which the radial modulus function in the error bound for $\Feas(\calV,\calF_{d^*+1})$, defined in \eqref{eq:def_kappa_Slater}, is uniformly bounded by a dimension-free constant. In this case, the resulting estimate is asymptotically tight in the sense of Definition~\ref{def:tight_eb}. A more detailed investigation of tightness in a broader setting will be presented in Section~\ref{sec:tight}.

\begin{example}[\textbf{An instance where the error bound for the facially reduced problem is asymptotically tight}]
Let $n \in \bbZ_{\ge 3}$ and
\begin{align*}
\calK_n &\coloneqq \calS_+^2 \oplus \{0\}^{n-2},\\
\calV_n &\coloneqq \setspan\calK_n = \calS^2 \oplus \{0\}^{n-2}.
\end{align*}
Note that $\calK_n$ is a face of $\calS_+^n$ exposed by the matrix $\frac{1}{\sqrt{n-2}}\sum_{i=3}^n\bm{E}_{ii} \in \calS_+^n \cap \calV_n^\perp$ and
\begin{equation*}
\calV_n \cap \ri\calK_n = (\setint\calS_+^2) \oplus \{0\}^{n-2} \neq \emptyset.
\end{equation*}
Then Proposition~\ref{prop:eb_reduced} implies that
\begin{equation}
\dist(\bm{X},\calV_n\cap \calK_n) \le \kappa(\norm{\bm{X}}_{\rm F}) \max\{\dist(\bm{X},\calV_n),\dist(\bm{X},\calK_n)\} \text{ for all $\bm{X} \in \calS^n$}, \label{eq:eb_reduced_Slater}
\end{equation}
where $\kappa(\norm{\bm{X}}_{\rm F}) = \max\{\eta(\calV_n,\setspan\calK_n),1\}\theta(\norm{\bm{X}}_{\rm F})$.

To see the asymptotic tightness of the error bound in \eqref{eq:eb_reduced_Slater}, we take a sequence $(n_k) \subseteq \bbZ_{\ge 2}$ satisfying $\lim_{k\to \infty} n_k = \infty$ and a sequence $(\bm{X}_k)$ satisfying $\bm{X}_k \in \calS^{n_k}$ for all $k$ arbitrarily.
It follows from $\calV_{n_k} = \setspan\calK_{n_k}$ and Remark~\ref{rem:eta} that $\eta(\calV_{n_k},\setspan\calK_{n_k}) = 1$.
In addition, since $\calV_{n_k}$ is a subspace and $\Diag(\bm{I}_2,\bm{O}) \in \calV_{n_k} \cap \ri\calK_{n_k}$ holds, it follows from \eqref{eq:theta_tight_ub} that $\theta(\norm{\bm{X}_k}_{\rm F}) \le 1 + 2\sqrt{2}$, and so $\kappa(\norm{\bm{X}_k}_{\rm F}) \le 1 + 2\sqrt{2}$.
Therefore, we have
\begin{align*}
\frac{\dist(\bm{X}_k,\calV_{n_k} \cap \calK_{n_k})}{\kappa(\norm{\bm{X}_k}_{\rm F}) \max\{\dist(\bm{X}_k,\calV_{n_k}),\dist(\bm{X}_k,\calK_{n_k})\}} &\ge \frac{\dist(\bm{X}_k,\calK_{n_k})}{(1 + 2\sqrt{2})\dist(\bm{X}_k,\calK_{n_k})} = \frac{1}{1 + 2\sqrt{2}}.
\end{align*}
Thus, we see that the error bound in
\eqref{eq:eb_reduced_Slater} is asymptotically tight with a dimension-free constant $\frac{1}{1 + 2\sqrt{2}}$.
\end{example}

Considering the case where $\calF_{d^*+1}$ is polyhedral leads to the following theorem.

\begin{theorem}[\textbf{Radial modulus function for facially reduced problem with minimal face}]\label{thm:kappaB}
Let $\theta$ be defined as in \eqref{eq:thetaB} and let $\eta$ be given as in Lemma~\ref{lem:Hoffman}. 
Then
\begin{equation*}
\dist(\bm{X},\calV \cap \calF_{d^*+1}) \le \kappa(\norm{\bm{X}}_{\rm F})\max\{\dist(\bm{X},\calV),\dist(\bm{X},\calF_{d^*+1})\}  \text{ for all $\bm{X} \in \calS^n$},
\end{equation*}
where the function $\kappa(\rho)$ can be set to 
\begin{align*}
\kappa(\rho) = \begin{cases}
\max\{\eta(\calV,\setspan\calF_{d^*+1}),1\}\theta(\rho) & \text{(if $\calV \cap \ri\calF_{d^*+1} \neq \emptyset$)},\\
\eta(\calV,\calF_{d^*+1}) & \text{(if $\calF_{d^*+1}$ is polyhedral)}.
\end{cases}
\end{align*}
\end{theorem}

\begin{proof}
If $\calV \cap \ri\calF_{d^*+1} \neq \emptyset$ holds, then the statement holds by Proposition~\ref{prop:eb_reduced}.
If $\calF_{d^*+1}$ is polyhedral, it follows from the definition of $\eta(\calV,\calF_{d^*+1})$ that
\begin{equation*}
\dist(\bm{X},\calV \cap \calF_{d^*+1}) \le \eta(\calV,\calF_{d^*+1})\max\{\dist(\bm{X},\calV),\dist(\bm{X},\calF_{d^*+1})\} \text{ for all $\bm{X} \in \calS^n$}.
\end{equation*}
Therefore, \eqref{eq:eb_reduced} holds by setting $\kappa(\rho) \coloneqq \eta(\calV,\calF_{d^*+1})$ for any $\rho \in \bbR_+$.
\end{proof}

When $\calV \cap \ri\calF_{d^*+1} \neq \emptyset$, the upper bound for the function $\theta$ given in \eqref{eq:bar_theta} yields an upper bound for the function $\kappa$.
Indeed, by taking $\bm{U}_0 \in \calV \cap \calF_{d^*+1}$ and letting 
\begin{equation}
\overline{\kappa}(\rho;\bm{U}_0) \coloneqq \max\{\eta(\calV,\setspan\calF_{d^*+1}),1\}\overline{\theta}(\rho;\bm{U}_0), \label{eq:def_barkappa}
\end{equation}
we then obtain $\kappa(\rho) \le \overline{\kappa}(\rho;\bm{U}_0)$. 

In passing, we note that the constant $\eta(\calV,\calF_{d^*+1})$ can also be upper bounded by using a Hoffman constant (as shown in \eqref{enum:eta_inequality} of Lemma~\ref{lem:Hoffman}), and so, can be effectively estimated.
Also, the two situations $\calV \cap \ri\calF_{d^*+1} \neq \emptyset$ and $\calF_{d^*+1}$ being polyhedral can hold simultaneously; in that case, we can set 
\begin{equation}
\kappa(\rho) \coloneqq \min\{\max\{\eta(\calV,\setspan\calF_{d^*+1}),1\}\theta(\rho),\eta(\calV,\calF_{d^*+1})\}. \label{eq:kappa_min}
\end{equation}

\subsection{Constants for one-step facial residual functions}\label{subsec:1-FRF}
Recall that a one-step facial residual function for a face of the positive semidefinite cone can be written as \eqref{eq:1-FRF_PSD}.
This means that for $i = 1,\dots,d^*$, a one-step facial residual function $\psi_{\calF_i,\bm{Z}_i}(s,t)$ for $\calF_i$ and $\bm{Z}_i$ takes the form of
\begin{equation*}
\psi_{\calF_i,\bm{Z}_i}(s,t) = \alpha_i s + \beta_i \sqrt{st}
\end{equation*}
for some nonnegative $\alpha_i$  and $\beta_i$, where the $\calF_i$'s and $\bm{Z}_i$'s are as in \eqref{eq:chain}.
In this subsection, we compute the constants $\alpha_i$ and $\beta_i$ explicitly.
We also discuss how these constants relate to the choices of $\calF_i$ and $\bm{Z}_i$.

The following lemma provides a one-step facial residual function for the positive semidefinite cone $\calS_+^n$ associated with a block diagonal matrix having one positive definite block and all other blocks equal to zero.

\begin{lemma}\label{lem:1FRF_partialPD}
Let $r < n$ and $\bm{\Lambda} = \Diag(\lambda_1,\dots,\lambda_{n-r})$ be a diagonal matrix with positive diagonal elements $\lambda_1 \ge \cdots \ge \lambda_{n-r} > 0$.
Let $\bm{Z} \coloneqq \Diag(\bm{O},\bm{\Lambda}) \in \calS_+^n$.
Define
\begin{align*}
\gamma &\coloneqq \frac{\sqrt{\sum_{j=1}^{n-r} \lambda_j^2} + 1}{\lambda_{n-r}} + 1 = \frac{\norm{\bm{\Lambda}}_{\rm F} + 1}{\lambda_{\rm min}(\bm{\Lambda})} + 1,\\
\alpha &\coloneqq \sqrt{1 + 2(\gamma + 1)r + \gamma^2}, \\
\beta  &\coloneqq \sqrt{2(\gamma + 1)\sqrt{r}}.
\end{align*}
Then for any $\bm{X} \in \calS^n$ and $\epsilon \in \bbR_+$ such that $\dist(\bm{X},\calS_+^n) \le \epsilon$ and $\langle \bm{Z},\bm{X}\rangle \le \epsilon$, we have
\begin{equation}
\dist(\bm{X},\calS_+^n \cap \{\bm{Z}\}^\perp) \le \alpha\epsilon + \beta\sqrt{\epsilon\lVert\bm{X}\rVert_{\rm F}}. \label{eq:eb_positive_diagonal}
\end{equation}
In other words, for $\alpha$ and $\beta$ defined above, $\alpha s + \beta \sqrt{st}$ is a one-step facial residual function for $\calS_+^n$ and $\bm{Z}$.
\end{lemma}

\begin{proof}
We partition $\bm{X}$ into
\begin{equation*}
\bm{X} = \begin{pmatrix}
\bm{X}_{[11]} & \bm{X}_{[12]}\\
\bm{X}_{[12]}^\top & \bm{X}_{[22]}
\end{pmatrix},
\end{equation*}
where $\bm{X}_{[11]}\in \calS^r$, $\bm{X}_{[12]}\in \bbR^{r\times (n-r)}$, and $\bm{X}_{[22]}\in \calS^{n-r}$.
To prove \eqref{eq:eb_positive_diagonal}, we show the following inequalities:
\begin{align}
\dist(\bm{X}_{[11]},\calS_+^r) &\le \epsilon, \label{eq:0416_X11}\\
\lVert\bm{X}_{[22]}\rVert_{\rm F} &\le \gamma\epsilon, \label{eq:0416_X22}\\
\lVert\bm{X}_{[12]}\rVert_{\rm F} &\le \sqrt{(\gamma + 1) (\lVert\bm{X}\rVert_{\rm F} + \sqrt{r}\epsilon)\sqrt{r}\epsilon}. \label{eq:0416_X12}
\end{align}
Indeed, if \eqref{eq:0416_X11}, \eqref{eq:0416_X22}, and \eqref{eq:0416_X12} hold, then it follows that
\begin{align*}
\dist(\bm{X},\calS_+^n \cap \{\bm{Z}\}^\perp) &= \dist(\bm{X},\calS_+^{r}\oplus\{0\}^{n-r})\\
&= \sqrt{\dist(\bm{X}_{[11]},\calS_+^r)^2 + \lVert\bm{X}_{[22]}\rVert_{\rm F}^2 + 2\lVert\bm{X}_{[12]}\rVert_{\rm F}^2}\\
&\le \sqrt{\epsilon^2 + \gamma^2\epsilon^2 + 2(\gamma + 1)(\lVert\bm{X}\rVert_{\rm F} + \sqrt{r}\epsilon)\sqrt{r}\epsilon}\\
&= \sqrt{(1 + 2(\gamma + 1)r + \gamma^2)\epsilon^2 + 2(\gamma + 1)\sqrt{r}\epsilon\lVert\bm{X}\rVert_{\rm F}}\\
&\le \alpha\epsilon + \beta\sqrt{\epsilon\lVert\bm{X}\rVert_{\rm F}},
\end{align*}
so \eqref{eq:eb_positive_diagonal} holds.
We next turn to the proofs of \eqref{eq:0416_X11}, \eqref{eq:0416_X22}, and \eqref{eq:0416_X12}. 

\noindent
\fbox{Proof of \eqref{eq:0416_X11}}
\eqref{eq:0416_X11} follows from $\dist(\bm{X},\calS_+^n) \le \epsilon$ and \eqref{eq:ineq_dist_PSDcone}.

\noindent
\fbox{Proof of \eqref{eq:0416_X22}}
It follows from $\dist(\bm{X},\calS_+^n) \le \epsilon$ and \eqref{eq:ineq_dist_PSDcone} that $\dist(\bm{X}_{[22]},\calS_+^{n-r}) \le \epsilon$.
$\bm{R} \coloneqq - \bm{X}_{[22]} + P_{\calS_+^{n-r}}(\bm{X}_{[22]})$ satisfies $\bm{X}_{[22]} + \bm{R} \in \calS_+^{n-r}$ and $\|\bm{R}\|_{\rm F} \le \epsilon$.
In addition, we see that
\begin{equation*}
\|\bm{X}_{[22]}+\bm{R}\|_{\rm F} \le  \tr(\bm{X}_{[22]}+\bm{R}) \le \frac{\langle\bm{\Lambda},\bm{X}_{[22]}+\bm{R}\rangle}{\lambda_{n-r}},
\end{equation*}
where we use $\bm{X}_{[22]} + \bm{R} \in \calS_+^{n-r}$ to derive the first inequality.
Therefore, we have
\begin{equation*}
\|\bm{X}_{[22]}\|_{\rm F} \le \|\bm{R}\|_{\rm F} + \frac{\langle\bm{\Lambda},\bm{X}_{[22]}\rangle}{\lambda_{n-r}} + \frac{\langle\bm{\Lambda},\bm{R}\rangle}{\lambda_{n-r}} \le \|\bm{R}\|_{\rm F} + \frac{\epsilon}{\lambda_{n-r}} + \frac{\|\bm{\Lambda}\|_{\rm F}\|\bm{R}\|_{\rm F}}{\lambda_{n-r}} \le \gamma\epsilon,
\end{equation*}
where we use $\langle\bm{\Lambda},\bm{X}_{[22]}\rangle =\langle\bm{Z},\bm{X}\rangle \le \epsilon$ and the Cauchy--Schwarz inequality to derive the second inequality and use $\|\bm{R}\|_{\rm F} \le \epsilon$ to derive the third inequality.

\noindent
\fbox{Proof of \eqref{eq:0416_X12}}
By the assumption $\dist(\bm{X},\calS_+^n) \le \epsilon$ and the distance formula in \eqref{eq:dist_PSDcone}, every eigenvalue of $\bm{X}$ is at least $-\epsilon$, and hence $\bm{X} + \epsilon\bm{I}_n \in \calS_+^n$.
This implies that
\begin{equation*}
\bm{X} + (\delta + \epsilon)\bm{I}_n = \begin{pmatrix}
\bm{X}_{[11]} + (\delta+\epsilon)\bm{I}_r & \bm{X}_{[12]}\\
\bm{X}_{[12]}^\top & \bm{X}_{[22]} + (\delta+\epsilon)\bm{I}_{n-r}
\end{pmatrix} \in \setint\calS_+^n
\end{equation*}
for all $\delta > 0$.
Therefore, the Schur complement
\begin{equation}
\bm{X}_{[11]} + (\delta+\epsilon)\bm{I}_r - \bm{X}_{[12]}(\bm{X}_{[22]} + (\delta+\epsilon)\bm{I}_{n-r})^{-1}\bm{X}_{[12]}^\top \label{eq:0416_schur_complement}
\end{equation}
is positive definite.
Taking the inner product of \eqref{eq:0416_schur_complement} with $\bm{I}_r$, we have
\begin{align}
\langle\bm{X}_{[11]} + (\delta+\epsilon)\bm{I}_r,\bm{I}_r\rangle &\ge \langle\bm{X}_{[12]}(\bm{X}_{[22]} + (\delta+\epsilon)\bm{I}_{n-r})^{-1}\bm{X}_{[12]}^\top,\bm{I}_r\rangle \nonumber\\
&= \langle(\bm{X}_{[22]} + (\delta+\epsilon)\bm{I}_{n-r})^{-1},\bm{X}_{[12]}^\top\bm{X}_{[12]}\rangle\nonumber\\
&\ge \lambda_{\rm min}((\bm{X}_{[22]} + (\delta+\epsilon)\bm{I}_{n-r})^{-1})\tr(\bm{X}_{[12]}^\top\bm{X}_{[12]})\nonumber\\
&= \frac{\lVert\bm{X}_{[12]}\rVert_{\rm F}^2}{\lambda_{\rm max}(\bm{X}_{[22]} + (\delta+\epsilon)\bm{I}_{n-r})} \nonumber\\
&\ge \frac{\lVert\bm{X}_{[12]}\rVert_{\rm F}^2}{\lVert\bm{X}_{[22]}\rVert_{\rm F} + \delta+\epsilon} \nonumber\\
&\ge \frac{\lVert\bm{X}_{[12]}\rVert_{\rm F}^2}{(\gamma + 1)\epsilon + \delta}, \label{eq:0416_X12_inner_product}
\end{align}
where we use the following well-known inequality $\langle\bm{A},\bm{B}\rangle \ge \lambda_{\rm min}(\bm{A})\tr(\bm{B})$ for $\bm{A} \in \calS^n$ and $\bm{B} \in \calS_+^n$ (see, for example, \cite[Proposition~8.4.13]{Ber2009}) to derive the second inequality and use \eqref{eq:0416_X22} to derive the last inequality.
In addition, the left-hand side of \eqref{eq:0416_X12_inner_product} can be bounded by
\begin{equation*}
\langle\bm{X}_{[11]} + (\delta+\epsilon)\bm{I}_r,\bm{I}_r\rangle \le \lVert\bm{X}_{[11]} + (\delta+\epsilon)\bm{I}_r\rVert_{\rm F}\sqrt{r} \le (\lVert\bm{X}\rVert_{\rm F} + \sqrt{r}(\delta+\epsilon))\sqrt{r}.
\end{equation*}
Therefore, we have
\begin{equation*}
\lVert\bm{X}_{[12]}\rVert_{\rm F} \le \sqrt{((\gamma + 1)\epsilon + \delta)(\lVert\bm{X}\rVert_{\rm F} + \sqrt{r}(\delta+\epsilon))\sqrt{r}}.
\end{equation*}
Since $\delta > 0$ is arbitrary, we have
\begin{equation*}
\lVert\bm{X}_{[12]}\rVert_{\rm F} \le \sqrt{(\gamma + 1)(\lVert\bm{X}\rVert_{\rm F} + \sqrt{r}\epsilon)\sqrt{r}\epsilon}.
\end{equation*}
\end{proof}

Now, for each $i = 1,\dots,d^*$, we provide a one-step facial residual function for the face $\calF_i$ and the matrix $\bm{Z}_i \in \calF_i^*$.
Each face $\calF_i$ of $\calS_+^n$ is linearly isomorphic to $\calS_+^{r_i}$ for some  $r_i \leq n$.
When $\calF_i = \{\bm{O}\}$, the index $i$ must be $d^* + 1$ and $r_{d^*+1}$ is set to $0$.
Overall, for every $i = 1,\dots, d^*$, there exists $\bm{Q}_i \in \bbR^{n\times r_i}$ whose column vectors are orthonormal such that 
\begin{equation}
\calF_i = \bm{Q}_i\calS_+^{r_i}\bm{Q}_i^\top. \label{eq:Fi}
\end{equation}
In addition, since $\calF_{i+1}$ is also a face of $\calF_i$, there exists an orthogonal matrix $\bm{P}$ of order $r_i$ such that
\begin{equation}
\calF_{i+1} = \bm{Q}_i\bm{P}(\calS_+^{r_{i+1}} \oplus \{0\}^{r_i - r_{i+1}})\bm{P}^\top\bm{Q}_i^\top. \label{eq:Fi+1}
\end{equation}

\begin{lemma}\label{lem:rank_QtopZQ}
Let $\bm{Q}_i \in \bbR^{n\times r_i}$ be given as in \eqref{eq:Fi}.
Then for any $\bm{Z}\in \ri(\calF_i^* \cap \calF_{i+1}^\perp)$,
the matrix $\bm{Q}_i^\top\bm{Z}\bm{Q}_i$ is positive semidefinite and has rank $r_i - r_{i+1}$.
\end{lemma}

\begin{proof}
From \eqref{eq:Fi+1}, it follows that 
\begin{equation}
\calF_{i+1}^\perp = \left\{\bm{X} \in \calS^n \relmiddle| \bm{P}^\top\bm{Q}_i^\top\bm{X}\bm{Q}_i\bm{P} \in \left\{\begin{pmatrix}
\bm{O} & \bm{A}_{[12]}\\
\bm{A}_{[12]}^\top & \bm{A}_{[22]}
\end{pmatrix} \relmiddle|
\begin{aligned}
&\bm{A}_{[12]} \in \bbR^{r_{i+1} \times (r_i - r_{i+1})},\\
&\bm{A}_{[22]} \in \calS^{r_i - r_{i+1}}
\end{aligned}\right\}\right\}.
\label{eq:Fi+1_perp}
\end{equation}
In addition, we see that
\begin{equation}
\calF_i^* = \{\bm{X} \in \calS^n \mid \bm{Q}_i^\top \bm{X}\bm{Q}_i \in \calS_+^{r_i}\} = \{\bm{X} \in \calS^n \mid \bm{P}^\top\bm{Q}_i^\top \bm{X}\bm{Q}_i\bm{P} \in \calS_+^{r_i}\}. \label{eq:Fi_dual}
\end{equation}
From \eqref{eq:Fi+1_perp} and \eqref{eq:Fi_dual}, we have 
\begin{equation}
\ri(\calF_i^* \cap \calF_{i+1}^\perp) = \{\bm{X} \in \calS^n \mid \bm{P}^\top\bm{Q}_i^\top\bm{X}\bm{Q}_i\bm{P} \in \{0\}^{r_{i+1}} \oplus \setint\calS_+^{r_i - r_{i+1}}\}. \label{eq:ri_Fidual_cap_Fi+1perp}
\end{equation}

Let $\bm{Z} \in \ri(\calF_i^* \cap \calF_{i+1}^\perp)$.
By \eqref{eq:ri_Fidual_cap_Fi+1perp}, there exists $\bm{A}_{[22]} \in \setint\calS_+^{r_i - r_{i+1}}$ such that $\bm{P}^\top\bm{Q}_i^\top\bm{Z}\bm{Q}_i\bm{P} = \Diag(\bm{O},\bm{A}_{[22]})$.
Then it follows that $\bm{Q}_i^\top\bm{Z}\bm{Q}_i = \bm{P}\Diag(\bm{O},\bm{A}_{[22]})\bm{P}^\top$.
By the choice of $\bm{A}_{[22]}$, the matrix $\bm{Q}_i^\top\bm{Z}\bm{Q}_i$ is positive semidefinite and has rank $r_i - r_{i+1}$.
Therefore, we obtain the desired result.
\end{proof}

In particular, by the following lemma, Lemma~\ref{lem:rank_QtopZQ} can be applied to the matrices $\bm{Z}_1,\dots,\bm{Z}_{d^*}$.

\begin{lemma}\label{lem:Zi_in_ri}
For every $i = 1,\dots,d^*$, $\bm{Z}_i \in \ri(\calF_i^* \cap \calF_{i+1}^\perp)$ holds.
\end{lemma}

\begin{proof}
{We prove this lemma using the notion of \emph{niceness}.
Recall that a closed convex cone $\calK$ in $\calS^n$ is said to be nice if $\calK^* + \calF^\perp$ is closed for every face $\calF$ of $\calK$.

We note that $\calF_i^\perp$ is the lineality space of $\calF_i^*$, namely, the largest subspace included in $\calF_i^*$.
It is known that $\calF_i^*$ can be written as the sum of its lineality space $\calF_i^\perp$ and its pointed component $\calF_i^* \cap (\calF_i^\perp)^\perp$~\cite[page~165]{Rockafellar1970}.
Moreover, $\calF_i^*$ is nice if and only if its pointed component $\calF_i^* \cap (\calF_i^\perp)^\perp$ is nice~\cite[Proposition~2.6.\Rnum{2}]{RT2019}.
We observe that
\begin{equation*}
\calF_i^* \cap (\calF_i^\perp)^\perp = \calF_i^* \cap \setspan\calF_i = \calF_i = \calS_+^n \cap \setspan\calF_i.
\end{equation*}
Here, the second equality follows from the self-duality of $\calF_i$ in the subspace $\setspan\calF_i$.
This self-duality follows from the fact that $\calF_i$ is linearly isomorphic to $\calS_+^{r_i}$ as shown in \eqref{eq:Fi} and that $\calS_+^{r_i}$ is self-dual.
Since $\calS_+^n$ is nice~\cite[Section~2.5]{Pataki2007} and the subspace $\setspan\calF_i$ is nice as a polyhedral cone~\cite[Section~1.1]{Pataki2007}, \cite[Proposition~5]{Pataki2013_On} implies that their intersection $\calS_+^n \cap \setspan\calF_i$, which is equal to $\calF_i^* \cap (\calF_i^\perp)^\perp$, is nice.
This means that the entire cone $\calF_i^*$ is also nice.
Therefore, by \cite[Proposition~1]{Lourenco2021}, we conclude that $\bm{Z}_i \in \ri(\calF_i^* \cap \calF_{i+1}^\perp)$.}
\end{proof}

\begin{proposition}[\textbf{Constants for the one-step facial residual functions}] \label{prop:1FRF}
For each $i = 1,\dots,d^*$, define the following constants:
\begin{align}
\gamma_i &\coloneqq \frac{\norm{\bm{Q}_i^\top\bm{Z}_i\bm{Q}_i}_{\rm F}+1}{\lambda_{\rm min}^+(\bm{Q}_i^\top\bm{Z}_i\bm{Q}_i)} + 1, \label{eq:def_gammai}\\
\alpha_i &\coloneqq \sqrt{1 + 2(\gamma_i + 1)r_{i+1} + \gamma_i^2}, \label{eq:def_alphai}\\
\beta_i  &\coloneqq \sqrt{2(\gamma_i + 1)\sqrt{r_{i+1}}}. \label{eq:def_betai}
\end{align} 
Then, for any $\bm{X}\in \setspan\calF_i$ and $\epsilon \in \bbR_+$ such that $\dist(\bm{X},\calF_i)\le \epsilon$ and $\langle\bm{Z}_i,\bm{X}\rangle \le \epsilon$,
we have $\dist(\bm{X},\calF_{i+1}) \le \alpha_i\epsilon + \beta_i\sqrt{\epsilon\lVert\bm{X}\rVert_{\rm F}}$.
In other words, for $\alpha_i$ and $\beta_i$ defined above, $\alpha_i s + \beta_i \sqrt{st}$ is a one-step facial residual function for $\calF_i$ and $\bm{Z}_i$.
\end{proposition}

\begin{proof}
Since $\bm{Q}_i^\top\bm{Z}_i\bm{Q}_i$ is a positive semidefinite matrix of rank $r_i-r_{i+1}$ by Lemmas~\ref{lem:rank_QtopZQ} and \ref{lem:Zi_in_ri}, there exist an orthogonal matrix $\bm{P}$ of order $r_i$ and a diagonal matrix $\bm{\Lambda}$ of order $r_i-r_{i+1}$ with positive diagonal elements such that $\bm{Q}_i^\top\bm{Z}_i\bm{Q}_i = \bm{P}\Diag(\bm{O},\bm{\Lambda})\bm{P}^\top$.

Let $\bm{X}\in \setspan\calF_i$ and $\epsilon \in \bbR_+$ be such that $\dist(\bm{X},\calF_i)\le \epsilon$ and $\langle\bm{Z}_i,\bm{X}\rangle \le \epsilon$. We have
\begin{align}
\dist(\bm{P}^\top\bm{Q}_i^\top\bm{X}\bm{Q}_i\bm{P},\calS_+^{r_i}) &= \dist(\bm{X},\calF_i) \le \epsilon, \nonumber\\
\langle\Diag(\bm{O},\bm{\Lambda}),\bm{P}^\top\bm{Q}_i^\top\bm{X}\bm{Q}_i\bm{P}\rangle &= \langle\bm{Z}_i,\bm{X}\rangle \le \epsilon,\nonumber\\
\norm{\bm{P}^\top\bm{Q}_i^\top\bm{X}\bm{Q}_i\bm{P}}_{\rm F} &= \norm{\bm{X}}_{\rm F}. \label{eq:norm_PQXQP}
\end{align}
Applying Lemma~\ref{lem:1FRF_partialPD} to $\bm{P}^\top\bm{Q}_i^\top\bm{X}\bm{Q}_i\bm{P}$ and $\Diag(\bm{O},\bm{\Lambda})$,  the resulting constants coincide with $\gamma_i$, $\alpha_i$, $\beta_i$ as defined in 
\eqref{eq:def_gammai}, \eqref{eq:def_alphai}, \eqref{eq:def_betai} and we have
\begin{align*}
\dist(\bm{X},\calF_{i+1}) &= \dist(\bm{X},\calF_i \cap \{\bm{Z}_i\}^\perp)\\
&=\dist(\bm{P}^\top\bm{Q}_i^\top\bm{X}\bm{Q}_i\bm{P},\calS_+^{r_i}\cap\{\Diag(\bm{O},\bm{\Lambda})\}^\perp)\\
&\le \alpha_i\epsilon + \beta_i\sqrt{\epsilon\lVert\bm{P}^\top\bm{Q}_i^\top\bm{X}\bm{Q}_i\bm{P}\rVert_{\rm F}}\\
&= \alpha_i\epsilon + \beta_i\sqrt{\epsilon\lVert\bm{X}\rVert_{\rm F}},
\end{align*}
where the final equality follows from \eqref{eq:norm_PQXQP}.
\end{proof}

\begin{remark}\label{rem:alpha_beta}
We discuss lower bounds for the constants $\alpha_i$ and $\beta_i$ introduced in \eqref{eq:def_alphai} and \eqref{eq:def_betai}, respectively.
From \eqref{eq:def_alphai} and \eqref{eq:def_betai}, we see that $\alpha_i \geq \max\{\beta_i,1\}$.
The constant
$\beta_i$ depends on how the affine space $\calV$ intersects  $\calS_+^n$.
When $\calV \cap \calS_+^n = \{\bm{O}\}$, it follows from the proof of \cite[Proposition~27]{Lourenco2021} that we can take a matrix in $\calS_+^n \cap \calV^\perp$ that exposes the face $\{\bm{O}\}$ of $\calS_+^n$, so that $d^* \le 1$.
If $d^* = 1$, then we have $r_2 = 0$, and \eqref{eq:def_betai} implies $\beta_1 = 0$.
When $\calV \cap \calS_+^n \neq \{\bm{O}\}$, for every $i = 1,\dots,d^* + 1$, the face $\calF_i$ strictly includes $\{\bm{O}\}$ and $r_i \ge 1$ holds, so we see from \eqref{eq:def_betai} that $\beta_i \ge 1$.
\end{remark}

\begin{remark}[\textbf{Comparison to earlier work}]
In \cite[Theorem~35]{Lourenco2021} it was shown that facial residual functions for symmetric cones can be taken to be of the form $\alpha s + \beta \sqrt{s t}$ for nonnegative constants $\alpha$ and $\beta$.
This result also applies to $\calS^n_+$ since it is a symmetric cone~\cite[Section~\RNum{1}.2]{FK1994}.
However, the constants $\alpha$ and $\beta$ are never computed explicitly in \cite{Lourenco2021}.
In contrast, Proposition~\ref{prop:1FRF}
leads to a computable formula for $\alpha_i$ and $\beta_i$ and clarifies how the face $\calF_i$ and the matrix $\bm{Z}_i$ affect these  constants.

Specifically, the constant $\gamma_i$ defined in \eqref{eq:def_gammai} depends on the choice of the matrix $\bm{Z}_i$. 
Once the faces $\calF_1,\dots,\calF_{d^*+1}$ are fixed, any matrix $\bm{Z} \in \ri(\calF_i^* \cap \calF_{i+1}^\perp)$ satisfies $\calF_{i+1} = \calF_i \cap \{\bm{Z}\}^\perp$.
Conversely, any matrix $\bm{Z} \in \calF_i^*$ satisfying $\calF_{i+1} = \calF_i \cap \{\bm{Z}\}^\perp$ belongs to $\ri(\calF_i^* \cap \calF_{i+1}^\perp)$.
The latter can be shown by the same argument as in the proof of Lemma~\ref{lem:Zi_in_ri}.
Hence, the matrix 
$\bm{Z}_i \in \ri(\calF_i^* \cap \calF_{i+1}^\perp) \cap \calV^\perp \cap \{\bm{Z}\in \calS^n \mid \lVert\bm{Z}\rVert_{\rm F} = 1\}$ 
can be chosen in order to minimize $\gamma_i$. The following example shows that the ``deeper'' $\bm{Z}_i$ lies in this set, the smaller the resulting  $\gamma_i$ becomes. 

Recall that $\bm{E}_{11}$ is the matrix whose $(1,1)$th element is $1$ and the other elements are $0$.
Let $\calF_1 \coloneqq \calS_+^n$, $\calF_2 \coloneqq \bbR_+\bm{E}_{11}$, and $\calV \coloneqq \{\bm{Z}\in \calS^n \mid Z_{22} = \cdots = Z_{nn} = 0\}$.
Then, $\calF_2$ is a face of $\calF_1$, and we have
\begin{equation*}
\ri(\calF_1^* \cap \calF_2^\perp) \cap \calV^\perp \cap \{\bm{Z}\in \calS^n \mid \lVert\bm{Z}\rVert_{\rm F} = 1\} = \{\Diag(0,Z_{22},\dots,Z_{nn}) \mid Z_{22},\dots,Z_{nn} > 0,\ {\textstyle \sqrt{Z_{22}^2 + \cdots +  Z_{nn}^2}} = 1\}.
\end{equation*}
For each $\bm{Z} \in \ri(\calF_1^* \cap \calF_2^\perp) \cap \calV^\perp \cap \{\bm{Z}\in \calS^n \mid \lVert\bm{Z}\rVert_{\rm F} = 1\}$,  $\gamma_1$ becomes
\begin{equation}
\frac{2}{\min\{Z_{22},\dots,Z_{nn}\}} + 1. \label{eq:hkappa_ex}
\end{equation}
Let us consider \eqref{eq:hkappa_ex} as a function of $Z_{22},\dots,Z_{nn}$. It attains its global minimum at $Z_{22} = \dots = Z_{nn} = \frac{1}{\sqrt{n-1}}$ under the condition that $Z_{22},\dots,Z_{nn} > 0$ and $\sqrt{Z_{22}^2 + \cdots +  Z_{nn}^2} = 1$.

We can interpret this result in a geometrical way.
For each $Z_{22},\dots,Z_{nn}$ such that $Z_{22},\dots,Z_{nn} > 0$ and $\sqrt{Z_{22}^2 + \cdots +  Z_{nn}^2} = 1$, the smallest perturbation that makes the vector $(Z_{22},\dots,Z_{nn})^\top$ not be in $\setint\bbR_+^{n-1}$ is
\begin{equation*}
\inf_{\bm{a} \in \bbR^{n-1}}\{\lVert\bm{a}\rVert_2 \mid (Z_{22},\dots,Z_{nn})^\top + \bm{a} \not\in \setint\bbR_+^{n-1}\} = \min\{Z_{22},\dots,Z_{nn}\},
\end{equation*}
which is maximized at $Z_{22} = \dots = Z_{nn} = \frac{1}{\sqrt{n-1}}$.
\end{remark}

\subsection{Qualitative error bound for the underlying semidefinite feasibility problem}\label{subsec:eb_originalSDP} 
In this subsection, by combining the ingredients derived in Sections~\ref{subsec:eb_reduced} and \ref{subsec:1-FRF}, we obtain the final qualitative radial-type H\"{o}lder error bound for the problem $\Feas(\calV,\calS_+^n)$.
For each $i = 1,\dots,d^*$, let $\alpha_i s + \beta_i\sqrt{s t}$ be the one-step facial residual function for $\calF_i$ and $\bm{Z}_i$ derived in Proposition~\ref{prop:1FRF} and let
\begin{equation*}
\widehat{\psi}_i(s,t) \coloneqq \begin{cases}
\alpha_1 s + \beta_1\sqrt{st} & \text{($i=1$)},\\
(2\alpha_i + 1)s + \beta_i\sqrt{2st} & \text{($i\ge 2$)}.
\end{cases}
\end{equation*}
For each $i = 0,\dots,d^*$, we define $\phi_i \colon \bbR_+\times \bbR_+ \to \bbR_+$ recursively as follows:\footnote{The operation of constructing $\phi_i$ for $i\ge 2$ from $\widehat{\psi}_i$ and $\phi_{i-1}$ is called \emph{diamond composition} in \cite{LLP2023}.}
\begin{equation}
\phi_i(s,t) \coloneqq \begin{cases}
s & \text{($i = 0$)},\\
\widehat{\psi}_1(s,t) & \text{($i = 1$)},\\
\widehat{\psi}_i(s + \phi_{i-1}(s,t),t) & \text{($i \ge 2$)}.
\end{cases} \label{eq:0416_phi}
\end{equation}
When $i\ge 1$, the nonnegativity of $\phi_i(s,t)$ follows from that of $\widehat{\psi}_i(s,t)$.
In the following lemma, we show that the function $\phi_i$ defined in \eqref{eq:0416_phi} satisfies $\phi_i(s,t) \ge s$. 

\begin{lemma}\label{lem:lb_phi_i}
For each $i = 0,\dots,d^*$, we have $\phi_i(s,t) \ge s$ for all $s,t\in \bbR_+$.
\end{lemma}

\begin{proof}
When $i = 0$, by definition, we have $\phi_0(s,t) = s$.
When $i = 1$, it follows that
\begin{equation*}
\phi_1(s,t) = \widehat{\psi}_1(s,t) = \alpha_1 s + \beta_1 \sqrt{st} \ge s,
\end{equation*}
where we use $\alpha_1 \ge 1$ and $\beta_1 \ge 0$ (see Remark~\ref{rem:alpha_beta}) to derive the inequality.
When $i\ge 2$, we see that
\begin{equation*}
\phi_i(s,t) = \widehat{\psi}_i(s + \phi_{i-1}(s,t) ,t) = (2\alpha_i+1)(s + \phi_{i-1}(s,t)) + \beta_i\sqrt{2(s + \phi_{i-1}(s,t))t} \ge s,
\end{equation*}
where we use $\alpha_i,\beta_i \ge 0$ and $\phi_{i-1}(s,t) \ge 0$ to derive the inequality.
\end{proof}

In the following theorem, we show that the distance to the face $\calF_{d^*+1}$ can be bounded using the function $\phi_{d^*}$.
To ease the notation, we recall the definition of $\kappa$ in Theorem~\ref{thm:kappaB}, and define 
\begin{equation}\label{eq:dx}
\calE_{\rm b}(\bm{X}) \coloneqq \max\{\dist(\bm{X},\calV),\dist(\bm{X},\calS_+^n)\}.
\end{equation}

\begin{theorem}
Let the function $\phi_i$ be given as in \eqref{eq:0416_phi} and $\calE_{\rm b}(\bm{X})$ be defined as in \eqref{eq:dx}. 
Then, for every $i = 0,\dots,d^*$, we have
\begin{equation}
\dist(\bm{X},\calF_{i+1}) \le \phi_i(\calE_{\rm b}(\bm{X}),\norm{\bm{X}}_{\rm F}) \text{ for all $\bm{X} \in \calS^n$}. \label{eq:dist_Fl_bound}
\end{equation}
In particular, it follows that
\begin{equation}
\dist(\bm{X},\calV\cap\calS_+^n)\le 
\kappa(\norm{\bm{X}}_{\rm F})\phi_{d^*}(\calE_{\rm b}(\bm{X}),\norm{\bm{X}}_{\rm F}) \text{ for all $\bm{X} \in \calS^n$}. \label{eq:eb_general}
\end{equation}
\end{theorem}

\begin{proof}
First, we prove \eqref{eq:dist_Fl_bound} by  induction on $i\ge 0$.
Firstly, it follows from the definition of $\phi_0$ that
\begin{equation}
\dist(\bm{X},\calF_1) = \dist(\bm{X},\calS_+^n) \le \calE_{\rm b}(\bm{X}) = \phi_0(\calE_{\rm b}(\bm{X}),\norm{\bm{X}}_{\rm F}) \label{eq:dist_X_to_F1_ub}
\end{equation}
holds for all $\bm{X} \in \calS^n$, so the inequality in \eqref{eq:dist_Fl_bound} holds for $i = 0$.
Secondly, we consider the case $i = 1$.
It follows from $\bm{Z}_1 \in \calV^\perp$ that $\calV \subseteq \{\bm{Z}_1\}^\perp$.
Then we have
\begin{equation}
\langle\bm{Z}_1,\bm{X}\rangle \le \abs{\langle\bm{Z}_1,\bm{X}\rangle} = \dist(\bm{X},\{\bm{Z}_1\}^\perp) \overset{\scriptsize \text{(a)}}\le \dist(\bm{X},\calV) \overset{\scriptsize \text{(b)}}\le \calE_{\rm b}(\bm{X}), \label{eq:ip_X_Z1_ub}
\end{equation}
where we use $\calV \subseteq \{\bm{Z}_1\}^\perp$ to derive (a) and use the definition of $\calE_{\rm b}(\bm{X})$ to derive (b).
Combining \eqref{eq:dist_X_to_F1_ub} and \eqref{eq:ip_X_Z1_ub} with the one-step facial residual function for $\calF_1$ and $\bm{Z}_1$, we have
\begin{equation*}
\dist(\bm{X},\calF_2) \le \alpha_1 \calE_{\rm b}(\bm{X}) + \beta_1 \sqrt{\calE_{\rm b}(\bm{X}) \norm{\bm{X}}_{\rm F}} = \widehat{\psi}_1(\calE_{\rm b}(\bm{X}),\norm{\bm{X}}_{\rm F}) = \phi_1(\calE_{\rm b}(\bm{X}),\norm{\bm{X}}_{\rm F}).
\end{equation*}
Therefore, the inequality in \eqref{eq:dist_Fl_bound} holds for $i = 1$.

Now, we assume that the inequality in \eqref{eq:dist_Fl_bound} holds for $i = k-1 \in \{1,\dots,d^*-1\}$.
Let $\bm{X} \in \calS^n$.
Then we have
\begin{equation}
\dist(P_{\setspan\calF_k}(\bm{X}),\calF_k) \le \norm{P_{\setspan\calF_k}(\bm{X}) - P_{\calF_k}(\bm{X})}_{\rm F} \le \dist(\bm{X},\setspan\calF_k) + \dist(\bm{X},\calF_k) \le 2\dist(\bm{X},\calF_k) + \calE_{\rm b}(\bm{X}), \label{eq:dist_PspanFiX_Fi}
\end{equation}
where we use $P_{\calF_k}(\bm{X}) \in \calF_k$ to derive the first inequality, use the triangle inequality to derive the second inequality, and use $\calF_k \subseteq \setspan\calF_k$ to derive the third inequality.
Moreover, we have
\begin{equation}
\dist(P_{\setspan\calF_k}(\bm{X}),\calV) \le \norm{P_{\setspan\calF_k}(\bm{X}) - P_{\calV}(\bm{X})}_{\rm F} \le \dist(\bm{X},\setspan\calF_k) + \dist(\bm{X},\calV) \le 2\dist(\bm{X},\calF_k) + \calE_{\rm b}(\bm{X}), \label{eq:dist_PspanFiX_V}
\end{equation}
where we use $P_{\calV}(\bm{X}) \in \calV$ to derive the first inequality, use the triangle inequality to derive the second inequality, and use $\calF_k \subseteq \setspan\calF_k$ and the definition of $\calE_{\rm b}(\bm{X})$ to derive the third inequality.
By \eqref{eq:dist_PspanFiX_V} and $\bm{Z}_k \in \calV^\perp$, in a manner similar to that used in \eqref{eq:ip_X_Z1_ub}, it follows that
\begin{equation}
\langle \bm{Z}_k, P_{\setspan\calF_k}(\bm{X})\rangle \le \dist(P_{\setspan\calF_k}(\bm{X}),\{\bm{Z}_k\}^\perp)  \le 2\dist(\bm{X},\calF_k) + \calE_{\rm b}(\bm{X}). \label{eq:ip_PspanFiX_Zi}
\end{equation}
Combining \eqref{eq:dist_PspanFiX_Fi} and \eqref{eq:ip_PspanFiX_Zi} with the one-step facial residual function for $\calF_k$ and $\bm{Z}_k$, we have
\begin{align}
\dist(P_{\setspan\calF_k}(\bm{X}),\calF_{k+1}) &\le \alpha_k(2\dist(\bm{X},\calF_k) + \calE_{\rm b}(\bm{X})) + \beta_k \sqrt{(2\dist(\bm{X},\calF_k) + \calE_{\rm b}(\bm{X}))\norm{P_{\setspan\calF_k}(\bm{X})}_{\rm F}} \nonumber \\
&\le \alpha_k(2\dist(\bm{X},\calF_k) + \calE_{\rm b}(\bm{X})) + \beta_k \sqrt{(2\dist(\bm{X},\calF_k) + \calE_{\rm b}(\bm{X}))\norm{\bm{X}}_{\rm F}}, \label{eq:dist_PspanFiX_Fi+1}
\end{align}
where we use the fact that the projection to the subspace $\setspan\calF_k$ is nonexpansive to derive the second inequality.
Then it follows that
\begin{align*}
\dist(\bm{X},\calF_{k+1}) &\overset{\scriptsize \text{(a)}}\le \norm{\bm{X} - P_{\calF_{k+1}}(P_{\setspan\calF_k}(\bm{X}))}_{\rm F} \\
&\overset{\scriptsize \text{(b)}}\le \dist(\bm{X},\setspan\calF_k) + \dist(P_{\setspan\calF_k}(\bm{X}),\calF_{k+1})\\
&\overset{\scriptsize \text{(c)}}\le \dist(\bm{X},\calF_k) + \alpha_k(2\dist(\bm{X},\calF_k) + \calE_{\rm b}(\bm{X})) + \beta_k\sqrt{(2\dist(\bm{X},\calF_k) + \calE_{\rm b}(\bm{X}))\norm{\bm{X}}_{\rm F}}\\
&\le (2\alpha_k + 1)(\dist(\bm{X},\calF_k) + \calE_{\rm b}(\bm{X})) + \beta_k\sqrt{2(\dist(\bm{X},\calF_k) + \calE_{\rm b}(\bm{X}))\norm{\bm{X}}_{\rm F}} \\
&\overset{\scriptsize \text{(d)}}= \widehat{\psi}_k(\dist(\bm{X},\calF_k) + \calE_{\rm b}(\bm{X}),\norm{\bm{X}}_{\rm F})\\
&\overset{\scriptsize \text{(e)}}\le \widehat{\psi}_k(\calE_{\rm b}(\bm{X}) + \phi_{k-1}(\calE_{\rm b}(\bm{X}),\norm{\bm{X}}_{\rm F}),\norm{\bm{X}}_{\rm F})\\
&\overset{\scriptsize \text{(f)}}= \phi_k(\calE_{\rm b}(\bm{X}),\norm{\bm{X}}_{\rm F}),
\end{align*}
where we use $P_{\calF_{k+1}}(P_{\setspan\calF_k}(\bm{X})) \in \calF_{k+1}$ to derive (a), use the triangle inequality to derive (b), use $\calF_k \subseteq \setspan \calF_k$ and \eqref{eq:dist_PspanFiX_Fi+1} to derive (c), use the definition of $\widehat{\psi}_k$ to derive (d), use the inductive hypothesis to derive (e), and use the definition of $\phi_k$ to derive (f).
Therefore, the inequality in \eqref{eq:dist_Fl_bound} also holds for $i = k$.

Next, we prove \eqref{eq:eb_general}.
It follows that
\begin{align*}
\dist(\bm{X},\calV \cap\calS_+^n)   &\overset{\scriptsize \text{(a)}}= \dist(\bm{X},\calV \cap\calF_{d^*+1}) \\
& \overset{\scriptsize \text{(b)}}\le \kappa(\norm{\bm{X}}_{\rm F})\max\{\dist(\bm{X},\calV),\dist(\bm{X},\calF_{d^*+1})\} \\
&\overset{\scriptsize \text{(c)}}\le \kappa(\norm{\bm{X}}_{\rm F})\max\{\calE_{\rm b}(\bm{X}),\phi_{d^*}(\calE_{\rm b}(\bm{X}),\lVert\bm{X}\rVert_{\rm F})\}\\
&\overset{\scriptsize \text{(d)}}\le \kappa(\norm{\bm{X}}_{\rm F})\phi_{d^*}(\calE_{\rm b}(\bm{X}),\lVert\bm{X}\rVert_{\rm F}),
\end{align*}
where (a) holds because $\calV \cap\calS_+^n = \calV \cap\calF_{d^*+1}$, (b) follows from Theorem~\ref{thm:kappaB}, (c) results from \eqref{eq:dist_Fl_bound} and $\dist(\bm{X},\calV) \le \calE_{\rm b}(\bm{X})$, and (d) is a consequence of Lemma~\ref{lem:lb_phi_i}.
\end{proof}

In what follows, we bound the function $\phi_{d^*}(s,t)$ by explicit power functions, thereby deriving  a qualitative radial-type H\"{o}lder error bound for the problem $\Feas(\calV,\calS_+^n)$.
When $d^* = 0$, since $\phi_0(s,t) = s$ by definition, see \eqref{eq:0416_phi}, the error bound in \eqref{eq:eb_general} reduces to
\begin{equation}
\dist(\bm{X},\calV\cap\calS_+^n)\le \kappa(\norm{\bm{X}}_{\rm F})\calE_{\rm b}(\bm{X}) \text{ for all $\bm{X}\in \calS^n$}, \label{eq:eb_dPPS=0}
\end{equation}
which is a Lipschitz error bound.

When $d^* \ge 1$, the radial-type error bound in \eqref{eq:eb_general} is described in terms of the function $\phi_{d^*}$, which is constructed recursively by using the functions $\widehat{\psi}_1,\dots,\widehat{\psi}_{d^*}$.
In what follows, we derive upper bounds of the function $\phi_{i}(s,t)$ for every $i=1,\ldots,d^*$ by simple bivariate power functions with respect to $s$ and $t$. 
To do this, for $i = 1,\dots,d^*$, we let
\begin{equation*}
A_i \coloneqq \begin{cases}
\alpha_1 & (i = 1),\\
2\alpha_i + 1 &  (i\ge 2)
\end{cases} \text{ and }  B_i \coloneqq \begin{cases}
\beta_1 & (i = 1),\\
\sqrt{2}\beta_i & (i \ge 2).
\end{cases}
\end{equation*}

\begin{lemma}\label{lem:ub_hi}
Let $d^* \ge 1$.
For $i = 1,\dots,d^*$ and $j = 0,\dots,i$, we define $c_{i,j}$ recursively as 
\begin{equation*}
\left\{ \begin{aligned}
c_{1,0} &\coloneqq A_1 = \alpha_1, &&\\
c_{1,1} &\coloneqq B_1 = \beta_1, &&\\
c_{i,0} &\coloneqq A_i(1 + c_{i-1,0}) && (i\ge 2),\\
c_{i,1} &\coloneqq A_i c_{i-1,1} + B_i(1 + \sqrt{c_{i-1,0}}) && (i\ge 2),\\
c_{i,j} &\coloneqq A_ic_{i-1,j} + B_i\sqrt{c_{i-1,j-1}} && (i\ge 2,\ 2\le j\le i-1),\\
c_{i,i} &\coloneqq B_i\sqrt{c_{i-1,i-1}} && (i\ge 2).
\end{aligned}\right.
\end{equation*}
Then, for every $i = 1,\dots,d^*$, we have
\begin{equation}
\phi_i(s,t) \le \sum_{j=0}^i c_{i,j} s^{\frac{1}{2^j}}t^{1-\frac{1}{2^j}}. \label{eq:hi_ub}
\end{equation}
\end{lemma}
\begin{proof}
We note that $\widehat{\psi}_i(s,t) = A_i s + B_i \sqrt{st}$ holds for every $i = 1,\dots,d^*$ by the definitions of $A_i$ and $B_i$.
We prove \eqref{eq:hi_ub} by induction on $i\ge 1$.
When $i = 1$, both the left- and right-hand sides of \eqref{eq:hi_ub} are $\alpha_1 s + \beta_1 \sqrt{st}$, so \eqref{eq:hi_ub} holds.
Next, we assume that \eqref{eq:hi_ub} holds for $i = k-1\ge 1$.
For $i = k$, we have
\begin{align*}
\phi_k(s,t) &= A_k(s + \phi_{k-1}(s,t)) + B_k\sqrt{(s + \phi_{k-1}(s,t))t} \\
&\overset{\scriptsize \text{(a)}}\le A_k\Biggl(s + \sum_{j=0}^{k-1}c_{k-1,j}s^{\frac{1}{2^j}}t^{1-\frac{1}{2^j}}\Biggr) + B_k\sqrt{\Biggl(s + \sum_{j=0}^{k-1}c_{k-1,j}s^{\frac{1}{2^j}}t^{1-\frac{1}{2^j}}\Biggr)t} \\
&\overset{\scriptsize \text{(b)}}\le A_k  s + \sum_{j=0}^{k-1}A_kc_{k-1,j}s^{\frac{1}{2^j}}t^{1-\frac{1}{2^j}} + B_k \sqrt{st} + \sum_{j=0}^{k-1}B_k\sqrt{c_{k-1,j}}s^{\frac{1}{2^{j+1}}}t^{1-\frac{1}{2^{j+1}}} \\
&= A_k(1 + c_{k-1,0})s + (A_kc_{k-1,1} + B_k(1 + \sqrt{c_{k-1,0}}))\sqrt{st} \\
&\quad + \sum_{j=2}^{k-1}(A_kc_{k-1,j} + B_k\sqrt{c_{k-1,j-1}})s^{\frac{1}{2^j}}t^{1-\frac{1}{2^j}} + B_k\sqrt{c_{k-1,k-1}}s^{\frac{1}{2^k}}t^{1-\frac{1}{2^k}}\\
&\overset{\scriptsize \text{(c)}}= \sum_{j=0}^k c_{k,j} s^{\frac{1}{2^j}}t^{1-\frac{1}{2^j}},
\end{align*}
where (a) is a consequence of the inductive hypothesis, (b) holds by the subadditivity of the square root function, and (c) follows from the definition of $c_{k,j}$.
Thus, \eqref{eq:hi_ub} also holds for $i = k$.
\end{proof}

In the following lemma, we provide an explicit upper bound for the constants $c_{i,j}$ in \eqref{eq:hi_ub} using both $\alpha_1,\dots,\alpha_i$ and $\beta_1,\dots,\beta_i$.

\begin{lemma}\label{lem:ub_cij}
Let $d^* \ge 1$.
For each $i = 1,\dots,d^*$, 
we define
\begin{equation*}
\alpha^{(i)} \coloneqq \max_{1\le k\le i}\alpha_k,\ \beta^{(i)} \coloneqq \max_{1\le k\le i}\beta_k.
\end{equation*}
Then, for $i = 1,\dots,d^*$ and $j = 0,\dots,i$, we have
\begin{equation}
c_{i,j} \le 6^{i-1}(\alpha^{(i)})^{i-j}(\beta^{(i)})^{2-\frac{1}{2^{j-1}}}.\label{eq:cij_ub}
\end{equation}
\end{lemma}

\begin{proof}
For simplicity, let $p_j \coloneqq 2-\frac{1}{2^{j-1}}$.
We prove that \eqref{eq:cij_ub} holds for any $j = 0,\dots,i$ by induction on $i \ge 1$.
For $i = 1$, we have
\begin{align*}
c_{1,0} &= \alpha_1 = 6^{1-1}(\alpha^{(1)})^{1-0}(\beta^{(1)})^{p_0},\\
c_{1,1} &= \beta_1 = 6^{1-1}(\alpha^{(1)})^{1-1}(\beta^{(1)})^{p_1},
\end{align*}
so \eqref{eq:cij_ub} holds for $i = 1$ and $j = 0,1$.

In what follows, we suppose that $d^* \ge 2$, because the proof finishes at the previous paragraph when $d^* = 1$.
Under the inductive hypothesis that \eqref{eq:cij_ub} holds for $i = k-1 \ge 1$  and $j = 0,\dots,k-1$, we shall prove that \eqref{eq:cij_ub} holds for $i = k$ and  $j = 0,\dots,k$.
The discussion in Remark~\ref{rem:alpha_beta} tells us that if  $d^* \ge 2$ holds, then
$\calV \cap \calS_+^n \neq \{\bm{O}\}$ holds as well. Furthermore, in this case, $\alpha_i,\beta_i \ge 1$ hold for all $i = 1,\dots,d^*$.
By the definitions of $A_k$ and $B_k$, we have
\begin{equation}
A_k \le 3\alpha_k,\ B_k \le \sqrt{2}\beta_k. \label{eq:ub_Ai_Bi}
\end{equation}
Combining $\alpha_i,\beta_i \ge 1$ with the recursive definition of $c_{k-1,j}$, we see that
\begin{equation}
c_{k-1,j} \ge 1 \label{eq:cij_ge_1}
\end{equation}
for any $j = 0,\dots,k-1$.

Firstly, we have
\begin{equation*}
c_{k,0} = A_k(1 + c_{k-1,0}) \le 6\alpha_k c_{k-1,0} \le 6^{k-1}\alpha_k(\alpha^{(k-1)})^{k-1}(\beta^{(k-1)})^{p_0} \le 6^{k-1}(\alpha^{(k)})^k (\beta^{(k)})^{p_0},
\end{equation*}
where we use \eqref{eq:ub_Ai_Bi} and \eqref{eq:cij_ge_1} to derive the first inequality, use the inductive hypothesis to derive the second inequality, and use $\alpha_k,\alpha^{(k-1)} \le \alpha^{(k)}$ to derive the third inequality.
Secondly, we have
\begin{align*}
c_{k,1} &= A_k c_{k-1,1} + B_k(1 + \sqrt{c_{k-1,0}})\\
&\le 3\alpha_k c_{k-1,1} + 2\sqrt{2} \beta_k\sqrt{c_{k-1,0}}\\
&\le 3\cdot 6^{k-2} \alpha_k(\alpha^{(k-1)})^{k-2}(\beta^{(k-1)})^{p_1} + 2\sqrt{2}\beta_k\sqrt{6^{k-2}(\alpha^{(k-1)})^{k-1}(\beta^{(k-1)})^{p_0}}\\
&\le (3 + 2\sqrt{2})6^{k-2}(\alpha^{(k)})^{k-1}(\beta^{(k)})^{p_1}\\
&\le 6^{k-1}(\alpha^{(k)})^{k-1}(\beta^{(k)})^{p_1},
\end{align*}
where {the first inequality holds by  \eqref{eq:ub_Ai_Bi} and \eqref{eq:cij_ge_1}, the second inequality follows by the inductive hypothesis, and the third inequality holds because $\alpha_k \le \alpha^{(k)}$,  $\beta_k\le \beta^{(k)}$, $(\alpha^{(i)})_i$ and $(\beta^{(i)})_i$ are nondecreasing sequences.}
Thirdly, for every $j = 2,\dots,i-1$, we have
\begin{align*}
c_{k,j} &=  A_k c_{k-1,j} + B_k\sqrt{c_{k-1,j-1}} \\
&\le 3\alpha_k c_{k-1,j} + \sqrt{2}\beta_k\sqrt{c_{k-1,j-1}}\\
&\le 3\cdot 6^{k-2}\alpha_k (\alpha^{(k-1)})^{k-1-j}(\beta^{(k-1)})^{p_j} + \sqrt{2}\beta_k \sqrt{6^{k-2}(\alpha^{(k-1)})^{k-j}(\beta^{(k-1)})^{p_{j-1}}}\\
&\le (3 + \sqrt{2})6^{k-2}(\alpha^{(k)})^{k-j}(\beta^{(k)})^{p_j}\\
&\le 6^{k-1}(\alpha^{(k)})^{k-j}(\beta^{(k)})^{p_j},
\end{align*}
where we use $\alpha_k,\alpha^{(k-1)} \le \alpha^{(k)}$, $\beta_k,\beta^{(k-1)} \le \beta^{(k)}$, and $p_j = 1 + p_{j-1}/2$ to derive the third inequality.
Fourthly, we have
\begin{equation*}
c_{k,k} = B_k\sqrt{c_{k-1,k-1}} \le \sqrt{2}\beta_k\sqrt{6^{k-2}(\beta^{(k-1)})^{p_{k-1}}} \le 6^{k-1}(\beta^{(k)})^{p_k},
\end{equation*}
where we use $\beta_k,\beta^{(k-1)} \le \beta^{(k)}$ and $p_k = 1 + p_{k-1}/2$ to derive the final inequality.
Therefore, \eqref{eq:cij_ub} with $i = k$ also holds for any $j = 0,\dots,k$.
\end{proof}

Applying Lemmas~\ref{lem:ub_hi} and \ref{lem:ub_cij} to the error bound in \eqref{eq:eb_general}, we obtain a radial-type H\"{o}lder error bound of the form \eqref{eq:radial_Holder_eb} involving $\alpha_1,\dots,\alpha_{d^*}$ and $\beta_1,\dots,\beta_{d^*}$.

\begin{theorem}[\textbf{Radial-type error bound over the whole space}]\label{thm:eb_alpha_beta}
Let $d^* \ge 1$.
Then
\begin{equation*}
\dist(\bm{X},\calV\cap\calS_+^n) \le \kappa(\norm{\bm{X}}_{\rm F})6^{d^*-1} \sum_{j=0}^{d^*} (\alpha^{(d^*)})^{d^*-j}(\beta^{(d^*)})^{2-\frac{1}{2^{j-1}}} \calE_{\rm b}(\bm{X})^{\frac{1}{2^j}}\norm{\bm{X}}_{\rm F}^{1-\frac{1}{2^j}} \text{ for all $\bm{X}\in \calS^n$}.
\end{equation*}
In other words, a radial-type H\"{o}lder error bound in \eqref{eq:radial_Holder_eb} holds with 
\begin{equation*}
\mu_j(\rho) \coloneqq 6^{d^*-1} (\alpha^{(d^*)})^{d^*-j}(\beta^{(d^*)})^{2-\frac{1}{2^{j-1}}} \kappa(\rho) \, \rho^{1-\frac{1}{2^j}} \text{ and } \gamma_j=\frac{1}{2^j} \text{ for every $j=0,\ldots,\ell$, and } \ell=d^*.
\end{equation*}
\end{theorem}

Next, we show that the function $\phi_i$ can be alternatively bounded by another simpler bivariate power function with fractional exponents, up to an explicit constant $C_i$ defined recursively.

\begin{lemma}\label{lem:0416_phi_bound}
Let $d^* \ge 1$.
For each $i = 1,\dots,d^*$, we define
\begin{equation*}
C_i \coloneqq \begin{cases}
\alpha_1 & (i = 1),\\
2(C_{i-1} + 1)(2\alpha_i + 1) & (i \ge 2).
\end{cases}
\end{equation*}
Then we have $1 \le C_i \le 12^{i-1}\alpha_1\cdots \alpha_i$ and
\begin{equation}
\phi_i(s,t) \le C_i\sum_{j=0}^i s^{\frac{1}{2^j}}t^{1-\frac{1}{2^j}}. \label{eq:ub_phii}
\end{equation}
\end{lemma}

\begin{proof}
First, we derive bounds for the constant $C_i$.
By Remark~\ref{rem:alpha_beta}, $\alpha_i \geq 1$ holds for $i = 1,\dots,d^*$.
Then, the definition of $C_i$ implies  $C_i \ge 1$ for $i = 1,\dots,d^*$.

We next prove $C_i \le 12^{i-1}\alpha_1\cdots \alpha_i$ by induction on $i \ge 1$.
When $i = 1$ both the left- and right-hand sides are $\alpha_1$, so the inequality holds.
We assume that the inequality holds for $i = k-1 \ge 1$.
For $i = k$, it follows from the definition of $C_k$ that
\begin{equation*}
C_k = 2(C_{k-1} + 1)(2\alpha_k + 1) \le 12C_{k-1}\alpha_k \le 12^{k-1}\alpha_1\cdots \alpha_k,
\end{equation*}
where we use $C_{k-1} \ge 1$ (and so, $C_{k-1} + 1 \le 2C_{k-1}$) and $\alpha_k \ge 1$ (and so, $2\alpha_k + 1 \le 3 \alpha_k)$ to derive the first inequality, and use the inductive hypothesis to derive the second inequality.
Therefore, $C_i \le 12^{i-1}\alpha_1\cdots \alpha_i$ holds for $i = k$.

Next, we show \eqref{eq:ub_phii} by induction on $i\ge 1$. 
When $i=1$, it follows from $\alpha_1 \ge \beta_1$ (Remark~\ref{rem:alpha_beta}) that
\begin{equation*}
\phi_1(s,t) = \widehat{\psi}_1(s,t) = \alpha_1 s + \beta_1\sqrt{st} \le C_1\sum_{j=0}^{1}s^{\frac{1}{2^j}}t^{1-\frac{1}{2^j}},
\end{equation*}
so \eqref{eq:ub_phii} holds.
We assume that \eqref{eq:ub_phii} holds for $i = k-1 \ge 1$.
Recalling \eqref{eq:0416_phi}, when $i = k$, we have
\begin{align*}
\phi_k(s,t) &= (2\alpha_k + 1)(s+\phi_{k-1}(s,t))
+ \beta_k\sqrt{2(s + \phi_{k-1}(s,t))t}\\
&\le (2\alpha_k+1)\Biggl((C_{k-1} + 1)s + C_{k-1}\sum_{j=1}^{k-1}s^{\frac{1}{2^j}}t^{1-\frac{1}{2^j}}\Biggr)
 + \alpha_k\sqrt{2\Biggl((C_{k-1} + 1)s + C_{k-1}\sum_{j=1}^{k-1}s^{\frac{1}{2^j}}t^{1-\frac{1}{2^j}}\Biggr)t}\\
&\le (C_{k-1} + 1)(2\alpha_k + 1)s + C_{k-1}(2\alpha_k + 1)\sum_{j=1}^{k-1}s^{\frac{1}{2^j}}t^{1-\frac{1}{2^j}}\\
&\quad + \sqrt{2(C_{k-1}+ 1)}\alpha_k\sqrt{st} + \sqrt{2C_{k-1}}\alpha_k\sum_{j=1}^{k-1}s^{\frac{1}{2^{j+1}}}t^{1-\frac{1}{2^{j+1}}}\\
&= (C_{k-1} + 1)(2\alpha_k + 1)s + \underbrace{(C_{k-1}(2\alpha_k + 1)}_{\leq (C_{k-1}+1)(2\alpha_k +1)} + \underbrace{\sqrt{2(C_{k-1} + 1)}\alpha_k}_{\leq (C_{k-1}+1)(2\alpha_k +1)})\sqrt{st}\\
&\quad + \underbrace{(C_{k-1}(2\alpha_k + 1)}_{\leq (C_{k-1}+1)(2\alpha_k +1)} + \underbrace{\sqrt{2C_{k-1}}\alpha_k}_{\leq (C_{k-1}+1)(2\alpha_k +1)})\sum_{j=2}^{k-1}s^{\frac{1}{2^j}}t^{1-\frac{1}{2^j}} + \sqrt{2C_{k-1}}\alpha_k s^{\frac{1}{2^k}}t^{1-\frac{1}{2^k}}\\
&\le C_k\sum_{j=0}^k s^{\frac{1}{2^j}}t^{1-\frac{1}{2^j}},
\end{align*}
where the first inequality is a consequence of the inductive hypothesis and $\alpha_k \geq \beta_k$, the second inequality holds by the subadditivity of the square root function, and the third inequality follows from the fact that $C_{k-1} \geq 1$ holds, which implies that the coefficients of $s^{\frac{1}{2^j}}t^{1-\frac{1}{2^j}}$ for every $j = 0,\dots,k$ are bounded by $2(C_{k-1} + 1)(2\alpha_k + 1)$.
Therefore, \eqref{eq:ub_phii} also holds for $i = k$.
\end{proof}

Applying Lemma~\ref{lem:0416_phi_bound} to the error bound in \eqref{eq:eb_general}, we obtain an alternative radial-type H\"{o}lder error bound of the form \eqref{eq:radial_Holder_eb} involving only $\alpha_1,\dots,\alpha_{d^*}$. We note that this form is, in general, independent and simpler than the one given in Theorem~\ref{thm:eb_alpha_beta}. 

\begin{theorem}[\textbf{Alternative radial-type error bound over the whole space}]\label{thm:error_bound}
Let $d^* \ge 1$.
Then
\begin{equation*}
\dist(\bm{X},\calV\cap\calS_+^n) \le \kappa(\norm{\bm{X}}_{\rm F})12^{d^*-1}\alpha_1\cdots\alpha_{d^*}\displaystyle \sum_{j=0}^{d^*}\calE_{\rm b}(\bm{X})^{\frac{1}{2^j}}\lVert\bm{X}\rVert_{\rm F}^{1-\frac{1}{2^j}} \text{ for all $\bm{X}\in \calS^n$}.
\end{equation*}
In other words, the radial-type H\"{o}lder error bound in \eqref{eq:radial_Holder_eb} holds with 
\begin{equation*}
\mu_j(\rho)\coloneqq 12^{d^*-1}\alpha_1\cdots\alpha_{d^*} \kappa(\rho) \, \rho^{1-\frac{1}{2^j}} \text{ and } \gamma_j=\frac{1}{2^j} \text{ for every $j=0,\ldots,\ell$, and } \ell=d^*.
\end{equation*}
\end{theorem}

Next, for the ease of comparison and discussion for asymptotic tightness in Section~\ref{sec:tight}, we summarize and simplify the derived error bounds in the corollary below, according to the value of $d^* \in \{0,1,\dots,n-1\}$.
In particular, in line with Sturm's error bound \cite{Sturm2000}, we consider the radial-type error bound where the test set $\calR_n$ is given by 
\begin{equation}
\calR_n \coloneqq \{\bm{X} \in \calS^n \mid \calE_{\rm b}(\bm{X}) \le 1\}. \label{eq:Rn}
\end{equation}
Naturally, the conclusion can be extended to  test sets of the form $\{\bm{X} \in \calS^n \mid \calE_{\rm b}(\bm{X}) \le \lambda\}$ for some  $\lambda>0$.

\begin{corollary}[\textbf{Radial-type error {bounds} over test {sets} with a fixed upper bound of the backward {error}}]\label{cor:error_bound_constant}
For all $\bm{X}\in \calR_n$, we have
\begin{align}
\dist(\bm{X},\calV\cap\calS_+^n) 
\le &  
\begin{cases}
\kappa(\norm{\bm{X}}_{\rm F})\,\calE_{\rm b}(\bm{X}) & (d^* = 0), \\
\kappa(\norm{\bm{X}}_{\rm F})(\alpha_1 {\calE_{\rm b}(\bm{X})} + \beta_1  \sqrt{\calE_{\rm b}(\bm{X})\norm{\bm{X}}_{\rm F}}) & (d^* = 1), \\
\kappa(\norm{\bm{X}}_{\rm F})12^{d^*-1}\alpha_1\cdots\alpha_{d^*}(d^*
\max\{\norm{\bm{X}}_{\rm F}^{1-\frac{1}{2^{d^*}}},\sqrt{\norm{\bm{X}}_{\rm F}}\} + 1)\calE_{\rm b}(\bm{X})^{\frac{1}{2^{d^*}}} & (d^* \ge 2),
\end{cases} \label{eq:computable_eb}
\end{align} 
where the function $\kappa$ is given as in Theorem~\ref{thm:kappaB}.
\end{corollary}

\begin{proof}
The result follows from \eqref{eq:eb_dPPS=0} when $d^* = 0$ and from Theorem~\ref{thm:eb_alpha_beta} when $d^* = 1$.
In what follows, we consider the case where $d^* \ge 2$.
Let $\bm{X} \in \calR_n$ be arbitrary.
It follows from $\calE_{\rm b}(\bm{X}) \le 1$ that
\begin{equation}
\sum_{j=0}^{d^*}\calE_{\rm b}(\bm{X})^{\frac{1}{2^j}}\lVert\bm{X}\rVert_{\rm F}^{1-\frac{1}{2^j}} \le \left(\sum_{j=0}^{d^*}\norm{\bm{X}}_{\rm F}^{1-\frac{1}{2^j}}\right)\calE_{\rm b}(\bm{X})^{\frac{1}{2^{d^*}}} = \left(\sum_{j=1}^{d^*}\norm{\bm{X}}_{\rm F}^{1-\frac{1}{2^j}} + 1\right)\calE_{\rm b}(\bm{X})^{\frac{1}{2^{d^*}}}. \label{eq:sum_dX_times_normX_ub}
\end{equation}
If $\norm{\bm{X}}_{\rm F} \ge 1$, then since $\norm{\bm{X}}_{\rm F}^{1-\frac{1}{2^j}} \le \norm{\bm{X}}_{\rm F}^{1-\frac{1}{2^{d^*}}}$ for all $j = 1,\dots,d^*$, we have
\begin{equation}
\sum_{j=1}^{d^*}\norm{\bm{X}}_{\rm F}^{1-\frac{1}{2^j}} \le d^*\norm{\bm{X}}_{\rm F}^{1-\frac{1}{2^{d^*}}}. \label{eq:sum_dX_times_normX_ub_normX>=1}
\end{equation}
If $\norm{\bm{X}}_{\rm F} < 1$, then since $\norm{\bm{X}}_{\rm F}^{1-\frac{1}{2^j}} \le \sqrt{\norm{\bm{X}}_{\rm F}}$ for all $j = 1,\dots,d^*$, we have
\begin{equation}
\sum_{j=1}^{d^*}\norm{\bm{X}}_{\rm F}^{1-\frac{1}{2^j}} \le d^*\sqrt{\norm{\bm{X}}_{\rm F}}. \label{eq:sum_dX_times_normX_ub_normX<1}
\end{equation}
By \eqref{eq:sum_dX_times_normX_ub_normX>=1} and \eqref{eq:sum_dX_times_normX_ub_normX<1}, we have
\begin{equation*}
\eqref{eq:sum_dX_times_normX_ub} \le (d^*
\max\{\norm{\bm{X}}_{\rm F}^{1-\frac{1}{2^{d^*}}},\sqrt{\norm{\bm{X}}_{\rm F}}\} + 1)\calE_{\rm b}(\bm{X})^{\frac{1}{2^{d^*}}}.
\end{equation*}
Therefore, the conclusion directly follows from Theorem~\ref{thm:error_bound}.
\end{proof}

\begin{remark}[\textbf{Comparison with Sturm's (local) H\"{o}lder error bound}]
Let $d_{\rm s}$ be the singularity degree of the semidefinite feasibility problem $\Feas(\calV,\calS_+^n)$.
In a pioneering work,  Sturm~\cite{Sturm2000} showed that for any $\rho>0$ there exists $\alpha_\rho>0$ such that
\begin{equation*}
\dist(\bm{X},\calV\cap\calS_+^n) \le \alpha_{\rho} \epsilon^{\frac{1}{2^{d_{\rm s}}}} \text{ for any $\bm{X} \in \mathcal{S}^n$ with $\|\bm{X}\|_{\rm F} \le \rho$ and $\calE_{\rm b}(\bm{X}) \le \epsilon$},
\end{equation*}
where $\epsilon \in [0,1]$. 
 
On the other hand, Corollary~\ref{cor:error_bound_constant} implies that  for any $\bm{X} \in \mathcal{S}^n$ with $\|\bm{X}\|_{\rm F} \le \rho$ and $\calE_{\rm b}(\bm{X}) \le \epsilon$ where $\epsilon \in [0,1]$, one has 
\begin{equation*}
\dist(\bm{X},\calV\cap\calS_+^n) \le \alpha_{\rho} \epsilon^{\frac{1}{2^{d^*}}},
\end{equation*}
where $\alpha_{\rho}$ is given by
\begin{equation*}
\alpha_{\rho}=\begin{cases}
\kappa(\rho) & (d^* = 0), \\
\kappa(\rho)(\alpha_1 + \beta_1\sqrt{\rho}) & (d^* = 1), \\
\kappa(\rho)12^{d^*-1}\alpha_1\cdots\alpha_{d^*}(d^*
\max\{\rho^{1-\frac{1}{2^{d^*}}},\sqrt{\rho}\} + 1) & (d^* \ge 2).
\end{cases}  
\end{equation*}
Since $d^*=d_{\rm PPS}(\calV,\calS_+^n) \le d_{\rm s}$ always holds, we complement Sturm's result by providing explicit expressions for the constant $\alpha_{\rho}$.
\end{remark}

\section{Asymptotic tightness of the derived radial-type  error bound}\label{sec:tight}
In this section, we discuss the asymptotic tightness of the derived error bound. For ease of comparison, we focus on the most explicit form of radial-type error bounds given in \eqref{eq:computable_eb}.
We consider three cases depending on the quantity $d^*=d_{\rm PPS}(\calV,\calS_+^n)$, the distance to the PPS condition, of the problem $\Feas(\mathcal{V},\mathcal{S}^n_+)$.

We present two complementary perspectives on this aspect. On the one hand, we first examine the cases where $d^* \in \{0,1\}$, for which $d^*$ is independent of the ambient dimension.
In these cases, we provide families of instances where the error bound derived in the previous section is asymptotically tight up to a dimension-free constant.
 On the other hand, when $d^*=n-1$, corresponding to the most singular case, we identify a family of instances with varying dimension for which the derived error bound need not be asymptotically tight, as we will see at the end of the section.
\subsection{The case  $d^*= 0$}\label{sec:dPPS_0}
In this subsection, we examine a family of semidefinite feasibility problems $\Feas(\calV,\calS_+^n)$ with $d^*\coloneqq d_{\rm PPS}(\calV,\calS_+^n) = 0$ for which the error bound in \eqref{eq:computable_eb} is asymptotically tight with respect to a dimension-free constant.

Let $n\in \bbZ_{\ge 2}$ and
\begin{equation*}
\calV_n \coloneqq \{\Diag(-t + 1, t + 1,\dots,t + 1) \mid  t \in \bbR\}.
\end{equation*}
Noting that
\begin{align}
\calV_n \cap \calS_+^n &= \{\Diag(-t+1, t+1,\dots, t+1) \mid -1\le t \le 1\}, \label{eq:sd0_V_cap_PSD}\\
\calV_n \cap \setint\calS_+^n &= \{\Diag(-t+1,t + 1,\dots,t + 1) \mid -1< t < 1\}, \nonumber
\end{align}
we see that the problem $\Feas(\calV_n,\calS_+^n)$ satisfies Slater's condition, so $d^* = 0$ holds.
Moreover, Corollary~\ref{cor:error_bound_constant} implies that $\dist(\bm{X},\calV_n\cap\calS_+^n) \le 
\kappa(\norm{\bm{X}}_{\rm F})\calE_{\rm b}(\bm{X})$ for all $\bm{X} \in \calR_n$.
For convenience, let $n_k \coloneqq k + 1$ and
\begin{equation*}
\bm{X}_k \coloneqq \Diag\left(-\frac{1}{k+1},2 + \frac{1}{k+1},\dots,2 + \frac{1}{k+1}\right) \in \calS^{n_k}
\end{equation*}
for each $k \in \bbZ_{\ge 1}$. 

For this example, {the radial modulus function  $\kappa$} scales with the norm of the matrix (and so, cannot be a constant function).
Using \eqref{eq:dist_PSDcone}, we have $\dist(\bm{X}_k,\calV_{n_k}) = 0$ and $\dist(\bm{X}_k,\calS_+^{n_k}) = \frac{1}{k+1}$.
In particular, $\calE_{\rm b}(\bm{X}_k)= \frac{1}{k+1} \le 1$. Moreover, by \eqref{eq:sd0_V_cap_PSD}, 
\begin{equation}
\dist(\bm{X}_k,\calV_{n_k} \cap \calS_+^{n_k}) = \frac{1}{\sqrt{k+1}}. \label{eq:sd0_left-hand}
\end{equation}
Thus, for an error bound as in the first case of \eqref{eq:computable_eb} to hold,  $\kappa(\|\bm{X}_k\|_{\rm F})$ must be at least on the order of $\sqrt{k}$.

Next, let us compute an upper bound for $\kappa(\|\bm{X}_k\|_{\rm F})$ according to \eqref{eq:def_barkappa}.
It follows from $\calV_{n_k} \subseteq \setspan\calS_+^{n_k}$ and Remark~\ref{rem:eta} that $\max\{\eta(\calV_{n_k},\setspan\calS_+^{n_k}),1\} = 1$.
In addition, letting $\bm{U}_0 \coloneqq \bm{I}_{n_k} \in \calV_{n_k} \cap \calS_+^{n_k}$ and using the bound in \eqref{eq:bar_theta}, we have
\begin{equation*}
\theta(\norm{\bm{X}_k}_{\rm F}) \le \overline{\theta}(\norm{\bm{X}_k}_{\rm F};\bm{U}_0) = 1 + 2\sqrt{k+1} + 2\sqrt{4(k+1) - \frac{3}{k+1}}.
\end{equation*}
It then follows from \eqref{eq:def_barkappa} that
\begin{equation*}
\kappa(\|\bm{X}_k\|_{\rm F}) \le \overline{\kappa}(\|\bm{X}_k\|_{\rm F};\bm{U}_0) = 1 + 2\sqrt{k+1} + 2\sqrt{4(k+1) - \frac{3}{k+1}}.
\end{equation*}
{Thus, combining this with the fact that $\kappa(\norm{\bm{X}_k}_{\rm F})$ is at least on the order of $\sqrt{k}$, we conclude that $\kappa(\|\bm{X}_k\|_{\rm F})$ grows proportionally to $\sqrt{k}$.}

Next, we show that the error bound estimate is asymptotically tight.
Recall that $\calE_{\rm b}(\bm{X}_k)=\frac{1}{k+1}$. Then the right-hand side of \eqref{eq:computable_eb} is bounded by
\begin{equation}
\kappa(\norm{\bm{X}_k}_{\rm F})\, \calE_{\rm b}(\bm{X}_{n_k}) \le \left(1 + 2\sqrt{k+1} + 2\sqrt{4(k+1) - \frac{3}{k+1}}\right)\frac{1}{k+1}. \label{eq:eb_sd0_rhs}
\end{equation}
It then follows from \eqref{eq:sd0_left-hand} and \eqref{eq:eb_sd0_rhs} that
\begin{equation*}
\liminf_{k\to \infty} \frac{\dist(\bm{X}_k,\calV_{n_k}\cap\calS_+^{n_k})}{\kappa(\norm{\bm{X}_k}_{\rm F})\calE_{\rm b}(\bm{X}_k)} \ge \liminf_{k\to \infty}\frac{\sqrt{k+1}}{1 + 2\sqrt{k+1} + 2\sqrt{4(k+1) - \frac{3}{k+1}}} = \frac{1}{6}.
\end{equation*}
Thus, the derived error bound is asymptotically tight up to the dimension-free constant $1/6$ for this instance.

\subsection{The case  $d^*= 1$}\label{sec:dPPS_1}
For every $n \in \bbZ_{\ge 3}$, we define
\begin{equation}
\calV_n \coloneqq \left\{\bm{A} \in \calS^n \relmiddle| \sum_{i=1}^n A_{ii} = 1,\ \sum_{i=3}^n A_{ii} = 0\right\}.\label{eq:dPPS=1_V}
\end{equation}
We note that $\calV_n \cap \calS_+^n = \{\Diag(\bm{A},\bm{O} ) \in \calS^n \mid \bm{A} \in \calS_+^2,\ A_{11} + A_{22} = 1\}$.
The cone $\calS_+^n$ is not polyhedral since $n\ge 3$. Furthermore, $\calV_n \cap \setint\calS_+^n$ is empty.
Let
\begin{align}
\bm{Z}_1 &\coloneqq \Diag(0,0,1,\dots,1) \in \calS_+^n \cap \calV_n^\perp, \nonumber\\
\calF_2 &\coloneqq \calS_+^n \cap \{\bm{Z}_1\}^\perp = \calS_+^2 \oplus \{0\}^{n-2}. \label{eq:d*=1_F2}
\end{align}
Then, the problem $\Feas(\calV_n,\calF_2)$ satisfies the PPS condition because $\mathcal{V}_n \cap \ri\calF_2 \neq \emptyset$.
Therefore, $d^* = 1$ holds.
For this problem, we show that the error bound in \eqref{eq:computable_eb} is asymptotically tight up to a dimension-free constant.
 
For every $k\in \bbZ_{\ge 1}$, let $n_k \coloneqq k + 3$ and 
\begin{equation}
\bm{X}_k \coloneqq \begin{pmatrix}
1 & 0 & 0 & 0           & 0           & \cdots      & 0           & 0\\
0 & 0 & 0 & 0           & 0           & \cdots      & 0           & 0\\
0 & 0 & 0 & \frac{1}{k} & \frac{1}{k} & \cdots      & \frac{1}{k} & \frac{1}{k} \\
0 & 0 & \frac{1}{k}     & 0           & \frac{1}{k} & \cdots & \frac{1}{k} & \frac{1}{k} \\
0 & 0 & \frac{1}{k}     & \frac{1}{k} & 0           & \cdots & \frac{1}{k} & \frac{1}{k} \\
0 & 0 & \vdots          & \vdots      & \vdots      & \ddots & \vdots      & \vdots\\ 
0 & 0 & \frac{1}{k}     & \frac{1}{k} & \frac{1}{k} & \cdots & 0           & \frac{1}{k} \\
0 & 0 & \frac{1}{k}     & \frac{1}{k} & \frac{1}{k} & \cdots & \frac{1}{k} & 0
\end{pmatrix} \in \calS^{n_k}. \label{eq:Xk_infeasible}
\end{equation}
Since the upper left $2\times 2$ block of
$\bm{X}_k$ is positive semidefinite and $(\bm{X}_k)_{11} + (\bm{X}_k)_{22} = 1$ holds, the left-hand side of \eqref{eq:computable_eb} is
\begin{equation}
\dist(\bm{X}_k,\calV_{n_k} \cap \calS_+^{n_k}) = \sqrt{\sum_{i,j=3}^{n_k}(\bm{X}_k)_{ij}^2} = \sqrt{\frac{k+1}{k}} \to 1 \qquad (\text{as}\,\,k\to\infty). \label{eq:lhs_eb_infeasible}
\end{equation}
We next evaluate the right-hand side of \eqref{eq:computable_eb}.
Since $\bm{X}_k \in \calV_{n_k}$ and $\bm{X}_k$ has eigenvalues $1$ with multiplicity $2$, $0$ with multiplicity $1$, and $-1/k$ with multiplicity $k$, by \eqref{eq:dist_PSDcone}, we have $\dist(\bm{X}_k,\calV_{n_k}) = 0$ and $\dist(\bm{X}_k,\calS_+^{n_k}) = 1/\sqrt{k}$.
Therefore, $\calE_{\rm b}(\bm{X}_k) = 1/\sqrt{k}$.
Moreover, by choosing $\bm{X}^\star \coloneqq \Diag(1/2,1/2,0,\dots,0) \in \calV_{n_k} \cap \ri\calF_2$, it follows from \eqref{eq:def_barkappa} that $\kappa(\norm{\bm{X}_k}_{\rm F}) \le \overline{\kappa}(\norm{\bm{X}_k}_{\rm F};\bm{X}^\star)$.
As shown in Appendix~\ref{apdx:calc_comps}, we have
$\overline{\kappa}(\rho;\bm{X}^\star) = \sqrt{2}(1 + 2\sqrt{2} + 4\rho)$,
$\alpha_1 = \sqrt{n + 8\sqrt{n-2} + 15}$, and
$\beta_1 = \sqrt{\smash[b]{2\sqrt{2(n-2)} + 6\sqrt{2}}}$,
where $n = n_k = k+3$.
Hence, by combining these estimates, the right-hand side of \eqref{eq:computable_eb} at $\bm{X} = \bm{X}_k$ is bounded by
\begin{align}
&\kappa(\norm{\bm{X}_k}_{\rm F})(\alpha_1 \calE_{\rm b}(\bm{X}_k) + \beta_1 \sqrt{\calE_{\rm b}(\bm{X}_k)\norm{\bm{X}_k}_{\rm F}}) \nonumber\\
\le{}&
\overline{\kappa}(\norm{\bm{X}_k}_{\rm F};\bm{X}^\star)
(\alpha_1 \calE_{\rm b}(\bm{X}_k) + \beta_1 \sqrt{\calE_{\rm b}(\bm{X}_k)\norm{\bm{X}_k}_{\rm F}})\nonumber\\
\le{}&
\sqrt{2}\left(1 + 2\sqrt{2} + 4\sqrt{\frac{2k+1}{k}}\right)
\left(\sqrt{\frac{k + 8\sqrt{k+1} + 18}{k}} + \sqrt{\frac{(2\sqrt{2(k+1)} + 6\sqrt{2})\sqrt{2k+1}}{k}}\right)\nonumber\\
\to{}& 3(12 + \sqrt{2}) \qquad (\text{as}\,\, k\to \infty). \label{eq:rhs_eb_infeasible}
\end{align}
It follows from \eqref{eq:lhs_eb_infeasible} and \eqref{eq:rhs_eb_infeasible} that
\begin{equation*}
\liminf_{k\to \infty} \frac{\dist(\bm{X}_k,\calV_{n_k} \cap \calS_+^{n_k})}{\kappa(\norm{\bm{X}_k}_{\rm F})(\alpha_1 \calE_{\rm b}(\bm{X}_k) + \beta_1 \sqrt{\calE_{\rm b}(\bm{X}_k)\norm{\bm{X}_k}_{\rm F}})} \ge \frac{1}{3(12 + \sqrt{2})}.
\end{equation*}
Thus, in this instance, the error bound in \eqref{eq:computable_eb} is asymptotically tight up to the dimension-free constant $\frac{1}{3(12 + \sqrt{2})}$.

We will revisit this instance in Example~\ref{ex:infeasible_X} as the set of optimal solutions of an SDP satisfying the strict complementarity condition.

\subsection{The case $d^*= n-1$}\label{sec:dPPS_n-1}
For every $n \in \bbZ_{\ge 3}$, let
\begin{equation}
\calV_n \coloneqq \{\bm{A}\in \calS^n \mid A_{nn} = 0,\ A_{ii} = A_{1,i+1}\ (i = 2,\dots,n-1)\}. \label{eq:vecsp_dPPSmax}
\end{equation}
The vector space $\calV_n$ defined in \eqref{eq:vecsp_dPPSmax} is obtained from the instance in \cite[Example~2]{Sturm2000} by reordering rows and columns.
The singularity degree of this instance is $n-1$, and hence \eqref{eq:dpps_eq_ds} yields $d^* = n-1$.
After facial reduction, we have a chain $\calF_n \subsetneq \cdots \subsetneq \calF_1$ of faces of $\calS_+^n$ such that $\calF_n = \mathbb{R}_+\bm{E}_{11}$.
In what follows, we show that, for this particular family of instances $\mathcal{V}_n$, the error bound in \eqref{eq:computable_eb} is not asymptotically tight up to a dimension-free constant. 

Take a sequence $(n_k) \subseteq \bbZ_{\ge 3}$ satisfying $\lim_{k\to \infty} n_k = \infty$ and a sequence $(\bm{X}_k)$ satisfying $\bm{X}_k \in \calR_{n_k}$ arbitrarily, where we recall that $\calR_{n_k}$ is defined in \eqref{eq:Rn}.
For these sequences, we prove that the limit inferior of the ratio of the left-hand side of \eqref{eq:computable_eb} to the right-hand side of \eqref{eq:computable_eb}, i.e.,
\begin{equation*}
\rho_k \coloneqq \frac{\dist(\bm{X}_k,\calV_{n_k} \cap \calS_+^{n_k})}{\kappa(\norm{\bm{X}_k}_{\rm F})12^{n_k-2}\alpha_1\cdots \alpha_{n_k-1}((n_k-1)\max\{\norm{\bm{X}_k}_{\rm F}^{1-\frac{1}{2^{n_k-1}}},\sqrt{\norm{\bm{X}_k}_{\rm F}}\} + 1)\calE_{\rm b}(\bm{X}_k)^{\frac{1}{2^{n_k-1}}}}
\end{equation*}
is $0$.
For simplicity, we write $\epsilon_k$ for $\calE_{\rm b}(\bm{X}_k)$.

First, we provide a lower bound for the right-hand side of \eqref{eq:computable_eb}.
The value of $\kappa(\norm{\bm{X}_k}_{\rm F})$ is
\begin{equation*}
\kappa(\norm{\bm{X}_k}_{\rm F}) = \min\{\max\{\eta(\calV_{n_k},\setspan\calF_{n_k}),1\}\theta(\norm{\bm{X}_k}_{\rm F}),\eta(\calV_{n_k},\calF_{n_k})\} = 1,
\end{equation*}
where the first equality holds since $\calV_{n_k} \cap \ri\calF_{n_k} \neq \emptyset$ and $\calF_{n_k} = \mathbb{R}_+\bm{E}_{11}$ is polyhedral (see \eqref{eq:kappa_min}), and the second equality follows from $\eta(\calV_{n_k},\calF_{n_k}) = 1$, which is implied by $\calF_{n_k} \subseteq \calV_{n_k}$. 
In addition, it follows from Remark~\ref{rem:alpha_beta} that $\alpha_i \ge 1$ for each $i = 1,\dots,n_k-1$.
Therefore, we have
\begin{align}
&\kappa(\norm{\bm{X}_k}_{\rm F})12^{n_k-2}\alpha_1\cdots \alpha_{n_k-1}((n_k-1)\max\{\norm{\bm{X}_k}_{\rm F}^{1-\frac{1}{2^{n_k-1}}},\sqrt{\norm{\bm{X}_k}_{\rm F}}\} + 1)\epsilon_k^{\frac{1}{2^{n_k-1}}} \nonumber\\
\ge{}& 12^{n_k-2}(\max\{\norm{\bm{X}_k}_{\rm F}^{1-\frac{1}{2^{n_k-1}}},\sqrt{\norm{\bm{X}_k}_{\rm F}}\} + 1)\epsilon_k^{\frac{1}{2^{n_k-1}}}. \label{eq:lb_ub_eb_dPPSmax}
\end{align}

Next, we provide an upper bound for $\dist(\bm{X}_k,\calV_{n_k} \cap \calS_+^{n_k})$, the left-hand side of \eqref{eq:computable_eb}.
By $\calV_{n_k} \cap \calS_+^{n_k} = \bbR_+\bm{E}_{11}$ and the subadditivity of the square root function, we see that
\begin{align}
\dist(\bm{X}_k,\calV_{n_k}\cap \calS_+^{n_k}) &\le \abs{(\bm{X}_k)_{n_k n_k}} + \dist((\bm{X}_k)_{11},\bbR_+) + \sqrt{2}\sum_{s=0}^{n_k-2}\abs{(\bm{X}_k)_{1,n_k-s}} \nonumber \\
&\quad + \sum_{s=0}^{n_k-3}\abs{(\bm{X}_k)_{n_k-s-1,n_k-s-1}} + \sqrt{2}\sum_{0\le t < s\le n_k-2}\abs{(\bm{X}_k)_{n_k-s,n_k-t}}. \label{eq:dist^2_to_V_cap_PSDcone_ub}
\end{align}
We bound each term of the right-hand side of \eqref{eq:dist^2_to_V_cap_PSDcone_ub}.
Firstly, since 
\begin{equation*}
\dist(\bm{X}_k,\calV_{n_k}) = \sqrt{\frac{2}{3}\sum_{i=2}^{n_k-1}((\bm{X}_k)_{ii} - (\bm{X}_k)_{1,i+1})^2 + (\bm{X}_k)_{n_k n_k}^2}
\end{equation*}
and it does not exceed $\epsilon_k$, we have
\begin{align}
\sqrt{2/3}\abs{(\bm{X}_k)_{ii} - (\bm{X}_k)_{1,i+1}} &\le \epsilon_k\ (i = 2,\dots,n_k-1), \label{eq:ub_diff_Xii_X1i+1}\\
\abs{(\bm{X}_k)_{n_k n_k}} &\le \epsilon_k. \label{eq:ub_Xnn}
\end{align}
In particular, \eqref{eq:ub_Xnn} implies that, by using $\epsilon_k \le 1$ and $1\le (\norm{\bm{X}_k}_{\rm F} + 1)^{1-\frac{1}{2^{n_k-1}}}$, the first term of the right-hand side of \eqref{eq:dist^2_to_V_cap_PSDcone_ub}, i.e., $\abs{(\bm{X}_k)_{n_k n_k}}$ can be bounded by
\begin{equation}
\abs{(\bm{X}_k)_{n_k n_k}} \le (\norm{\bm{X}_k}_{\rm F} + 1)^{1-\frac{1}{2^{n_k-1}}}\epsilon_k^{\frac{1}{2^{n_k-1}}}. \label{eq:ub_Xnn_2}
\end{equation}
Secondly, the second term of the right-hand side of \eqref{eq:dist^2_to_V_cap_PSDcone_ub} admits the bound
\begin{equation}
\dist((\bm{X}_k)_{11},\bbR_+) \le \epsilon_k \le (\norm{\bm{X}_k}_{\rm F} + 1)^{1-\frac{1}{2^{n_k-1}}}\epsilon_k^{\frac{1}{2^{n_k-1}}},\label{eq:dist_X11_to_nno}
\end{equation}
where the first inequality holds by $\dist(\bm{X}_k,\calS_+^{n_k}) \le \epsilon_k$ and \eqref{eq:ineq_dist_PSDcone}, and the second inequality follows in the same manner as \eqref{eq:ub_Xnn_2}.
In the following lemma, we bound the third and fourth terms of the right-hand side of \eqref{eq:dist^2_to_V_cap_PSDcone_ub}.

\begin{lemma}\label{lem:ub_Xabs}
It follows that
\begin{align}
\lvert (\bm{X}_k)_{1,n_k-s}\rvert &\le (2 + (\sqrt{3/2} + 1)s)(\norm{\bm{X}_k}_{\rm F} + 1)^{1 - \frac{1}{2^{s+1}}}\epsilon_k^{\frac{1}{2^{s+1}}}\ (s = 0,\dots,n_k-2), \label{eq:ub_absX1n-k}\\
\lvert (\bm{X}_k)_{n_k-s-1,n_k-s-1}\rvert &\le (2 + \sqrt{3/2} + (\sqrt{3/2} + 1)s)(\norm{\bm{X}_k}_{\rm F} + 1)^{1 - \frac{1}{2^{s+1}}}\epsilon_k^{\frac{1}{2^{s+1}}}\ (s = 0,\dots,n_k-3). \label{eq:ub_absXn-k-1}
\end{align}
In particular, there exist constants  $a_{n_k}$ and $b_{n_k}$ (whose dependency on $n_k$ is polynomial) such that
\begin{align}
\sqrt{2}\sum_{s=0}^{n_k-2}\abs{X_{1,n_k-s}} &\le a_{n_k}(\norm{\bm{X}_k}_{\rm F} + 1)^{1 - \frac{1}{2^{n_k-1}}}\epsilon_k^{\frac{1}{2^{n_k-1}}}, \label{eq:ub_sum_absX1n-k}\\
\sum_{s=0}^{n_k-3}\abs{X_{n_k-s-1,n_k-s-1}} &\le b_{n_k}(\norm{\bm{X}_k}_{\rm F} + 1)^{1 - \frac{1}{2^{n_k-1}}}\epsilon_k^{\frac{1}{2^{n_k-1}}}. \label{eq:ub_sub_absXn-k-1}
\end{align}
\end{lemma}

\begin{proof}
We note that \begin{equation}
\lvert (\bm{X}_k)_{1i}\rvert \le \sqrt{((\bm{X}_k)_{11} + \epsilon_k)((\bm{X}_k)_{ii} + \epsilon_k)} \label{eq:ub_X1i}
\end{equation}
holds for every $i = 1,\dots,n_k$.
Indeed, $\bm{X}_k + \epsilon_k \bm{I}_{n_k}$ is positive semidefinite by $\dist(\bm{X}_k,\calS_+^{n_k}) \le \epsilon_k$ and \eqref{eq:dist_PSDcone}.
Hence, the determinant of the principal submatrix of $\bm{X}_k + \epsilon_k \bm{I}$ obtained by extracting the rows and columns indexed by $1$ and $i$ is nonnegative, which leads to \eqref{eq:ub_X1i}.

First, we show the inequalities in \eqref{eq:ub_absX1n-k} and \eqref{eq:ub_absXn-k-1} by induction.
The inequality in \eqref{eq:ub_absX1n-k} holds for $s = 0$ since
\begin{equation*}
\lvert (\bm{X}_k)_{1n_k}\rvert \le \sqrt{((\bm{X}_k)_{11} + \epsilon_k)((\bm{X}_k)_{n_kn_k} + \epsilon_k)} \le \sqrt{2(\norm{\bm{X}_k}_{\rm F} + 1)\epsilon_k} \le 2\sqrt{(\norm{\bm{X}_k}_{\rm F} + 1)\epsilon_k},
\end{equation*}
where we use \eqref{eq:ub_X1i} to derive the first inequality and use $\abs{(\bm{X}_k)_{11}} \le \norm{\bm{X}_k}_{\rm F}$, $\epsilon_k \le 1$, and \eqref{eq:ub_Xnn} to derive the second inequality.
We assume that \eqref{eq:ub_absX1n-k} holds for $s = t \in \{0,\dots,n_k-3\}$.
Then it follows that
\begin{align}
\lvert (\bm{X}_k)_{n_k-t-1,n_k-t-1}\rvert &\le \abs{(\bm{X}_k)_{1,n_k-t}} + \sqrt{3/2}\epsilon_k \nonumber\\
& \le (2 + (\sqrt{3/2} + 1)t)(\norm{\bm{X}_k}_{\rm F} + 1)^{1 - \frac{1}{2^{t+1}}}\epsilon_k^{\frac{1}{2^{t+1}}} + \sqrt{3/2}\epsilon_k \nonumber\\
& \le (2 + \sqrt{3/2} + (\sqrt{3/2} + 1)t)(\norm{\bm{X}_k}_{\rm F} + 1)^{1 - \frac{1}{2^{t+1}}}\epsilon_k^{\frac{1}{2^{t+1}}}, \label{eq:ub_Xn-l-1n-l-1}
\end{align}
where the first inequality holds by \eqref{eq:ub_diff_Xii_X1i+1}, the second inequality results from the inductive hypothesis, and the third inequality holds by $\epsilon_k \le (\norm{\bm{X}_k}_{\rm F} + 1)^{1 - \frac{1}{2^{t+1}}}\epsilon_k^{\frac{1}{2^{t+1}}}$.
Therefore, \eqref{eq:ub_absXn-k-1} holds for $s = t$.
Next, we have
\begin{align*}
\lvert (\bm{X}_k)_{1,n_k-(t+1)}\rvert &\le \sqrt{((\bm{X}_k)_{11} + \epsilon_k)((\bm{X}_k)_{n_k-(t+1),n_k-(t+1)} + \epsilon_k)}\\
&\le \sqrt{
(\norm{\bm{X}_k}_{\rm F} + 1)((2 + \sqrt{3/2} + (\sqrt{3/2} + 1)t)(\norm{\bm{X}_k}_{\rm F} + 1)^{1 - \frac{1}{2^{t+1}}}\epsilon_k^{\frac{1}{2^{t+1}}} + \epsilon_k)}\\
&\le \sqrt{(2 + (\sqrt{3/2} + 1)(t+1))}(\norm{\bm{X}_k}_{\rm F} + 1)^{1 - \frac{1}{2^{t+2}}}\epsilon_k^{\frac{1}{2^{t+2}}} \\
&\le (2 + (\sqrt{3/2} + 1)(t + 1))(\norm{\bm{X}_k}_{\rm F} + 1)^{1 - \frac{1}{2^{t+2}}}\epsilon_k^{\frac{1}{2^{t+2}}},
\end{align*}
where we use \eqref{eq:ub_X1i} to derive the first inequality, use $\abs{(\bm{X}_k)_{11}} \le \norm{\bm{X}_k}_{\rm F}$, $\epsilon_k \le 1$, and \eqref{eq:ub_Xn-l-1n-l-1} to derive the second inequality, use $\epsilon_k \le (\norm{\bm{X}_k}_{\rm F} + 1)^{1 - \frac{1}{2^{t+1}}}\epsilon_k^{\frac{1}{2^{t+1}}}$ to derive the third inequality.
Therefore, \eqref{eq:ub_absX1n-k} holds for $s = t+1$.

Next, using \eqref{eq:ub_absX1n-k}, we have
\begin{align*}
\sqrt{2}\sum_{s=0}^{n_k-2}\abs{X_{1,n_k-s}} &\le \sqrt{2} \sum_{s=0}^{n_k-2} (2 + (\sqrt{3/2} + 1)s)(\norm{\bm{X}_k}_{\rm F} + 1)^{1 - \frac{1}{2^{s+1}}}\epsilon_k^{\frac{1}{2^{s+1}}} \\
&\le \underbrace{\sqrt{2}\left(\sum_{s=0}^{n_k-2}(2 + (\sqrt{3/2} + 1)s)\right)}_{\eqqcolon a_{n_k}}(\norm{\bm{X}_k}_{\rm F} + 1)^{1-\frac{1}{2^{n_k-1}}}\epsilon_k^{\frac{1}{2^{n_k-1}}},
\end{align*}
so we obtain \eqref{eq:ub_sum_absX1n-k}.
Similarly, using \eqref{eq:ub_absXn-k-1}, by setting $b_{n_k} \coloneqq \sum_{s=0}^{n_k-3}(2 + \sqrt{3/2} + (\sqrt{3/2} + 1)s)$, we obtain \eqref{eq:ub_sub_absXn-k-1}.
\end{proof}

In the following lemma, we bound the fifth term of the right-hand side of \eqref{eq:dist^2_to_V_cap_PSDcone_ub}.

\begin{lemma}\label{lem:ub_absXn-kn-l}
There exists a positive constant $c_{n_k}$ (whose dependency on $n_k$ is polynomial) such that 
\begin{equation*}
\sqrt{2}\sum_{0\le t < s\le n_k-2}\abs{(\bm{X}_k)_{n_k-s,n_k-t}} \le c_{n_k}(\norm{\bm{X}_k}_{\rm F} + 1)^{1-\frac{1}{2^{n_k-1}}}\epsilon_k^{\frac{1}{2^{n_k-1}}}.
\end{equation*}
\end{lemma}

\begin{proof}
It follows from \eqref{eq:ub_absXn-k-1} that
\begin{equation}
(\bm{X}_k)_{n_k-s,n_k-s} + \epsilon_k \le (2 + (\sqrt{3/2} + 1)s)(\norm{\bm{X}_k}_{\rm F} + 1)^{1 - \frac{1}{2^s}}\epsilon_k^{\frac{1}{2^s}} \label{eq:ub_Xnk-s_plus_delta}
\end{equation}
for every $s = 1,\dots,n_k-2$.
The inequality in \eqref{eq:ub_Xnk-s_plus_delta} also holds for $s=0$ since \eqref{eq:ub_Xnn} implies that $(\bm{X}_k)_{n_k n_k} + \epsilon_k \le 2\epsilon_k$ holds.
Therefore, for each $0 \le t < s \le n_k-2$, we have
\begin{align*}
\abs{(\bm{X}_k)_{n_k-s,n_k-t}} &\le \sqrt{((\bm{X})_{n_k-s,n_k-s} + \epsilon_k)((\bm{X}_k)_{n_k-t,n_k-t} + \epsilon_k)} \nonumber \\
&\le \sqrt{
(2 + (\sqrt{3/2} + 1)s)(2 + (\sqrt{3/2} + 1)t)} (\norm{\bm{X}_k}_{\rm F} + 1)^{1 - \left(\frac{1}{2^{s+1}} + \frac{1}{2^{t+1}}\right)}\epsilon_k^{\frac{1}{2^{s+1}} + \frac{1}{2^{t+1}}}\nonumber\\
&\le \sqrt{
(2 + (\sqrt{3/2} + 1)s)(2 + (\sqrt{3/2} + 1)t)}(\norm{\bm{X}_k}_{\rm F} + 1)^{1 -\frac{1}{2^{n_k-1}}}\epsilon_k^{\frac{1}{2^{n_k-1}}},
\end{align*}
where we use the positive semidefiniteness of $\bm{X}_k + \epsilon_k \bm{I}_{n_k}$ to derive the first inequality, use \eqref{eq:ub_Xnk-s_plus_delta} to derive the second inequality, and use $0 \le t < s \le n_k-2$ to derive the third inequality.
Therefore, by letting
\begin{equation*}
c_{n_k} \coloneqq \sqrt{2}\sum_{0\le t < s\le n_k-2}\sqrt{
(2 + (\sqrt{3/2} + 1)s)(2 + (\sqrt{3/2} + 1)t)},
\end{equation*}
whose order is polynomial in $n_k$, we obtain the desired result.
\end{proof}

By \eqref{eq:dist^2_to_V_cap_PSDcone_ub}, \eqref{eq:ub_Xnn_2}, \eqref{eq:dist_X11_to_nno}, and Lemmas~\ref{lem:ub_Xabs} and \ref{lem:ub_absXn-kn-l}, we obtain
\begin{equation}
\dist(\bm{X}_k,\calV_{n_k}\cap\calS_+^{n_k})
\le (2 + a_{n_k} + b_{n_k} + c_{n_k})(\norm{\bm{X}_k}_{\rm F} + 1)^{1-\frac{1}{2^{n_k-1}}}\epsilon_k^{\frac{1}{2^{n_k-1}}}. \label{eq:dist_to_V_cap_PSDcone_ub}
\end{equation}

When $\norm{\bm{X}_k}_{\rm F} \ge 1$, since $\max\{\norm{\bm{X}_k}_{\rm F}^{1-\frac{1}{2^{n_k-1}}},\sqrt{\norm{\bm{X}_k}_{\rm F}}\} = \norm{\bm{X}_k}_{\rm F}^{1-\frac{1}{2^{n_k-1}}}$ holds by $n_k \ge 3$, it follows from \eqref{eq:lb_ub_eb_dPPSmax} and \eqref{eq:dist_to_V_cap_PSDcone_ub} that
\begin{equation*}
\rho_k \le \frac{(2 + a_{n_k} + b_{n_k} + c_{n_k})(\norm{\bm{X}_k}_{\rm F} + 1)^{1-\frac{1}{2^{n_k-1}}}}{12^{n_k-1}(\norm{\bm{X}_k}_{\rm F}^{1-\frac{1}{2^{n_k-1}}} + 1)} 
\le \frac{2 + a_{n_k} + b_{n_k} + c_{n_k}}{12^{n_k-1}\left(\frac{\norm{\bm{X}_k}_{\rm F}}{\norm{\bm{X}_k}_{\rm F} + 1}\right)^{1-\frac{1}{2^{n_k-1}}}}
\le \frac{2(2 + a_{n_k} + b_{n_k} + c_{n_k})}{12^{n_k-1}}.
\end{equation*}
When $\norm{\bm{X}_k}_{\rm F} < 1$, since $\max\{\norm{\bm{X}_k}_{\rm F}^{1-\frac{1}{2^{n_k-1}}},\sqrt{\norm{\bm{X}_k}_{\rm F}}\} = \sqrt{\norm{\bm{X}_k}_{\rm F}}$ holds by $n_k \ge 3$, it follows from \eqref{eq:lb_ub_eb_dPPSmax} and \eqref{eq:dist_to_V_cap_PSDcone_ub} that
\begin{equation*}
\rho_k \le \frac{(2 + a_{n_k} + b_{n_k} + c_{n_k})(\norm{\bm{X}_k}_{\rm F} + 1)^{1-\frac{1}{2^{n_k-1}}}}{12^{n_k-1}(\sqrt{\norm{\bm{X}_k}_{\rm F}} + 1)} \le \frac{2(2 + a_{n_k} + b_{n_k} + c_{n_k})}{12^{n_k-1}},
\end{equation*}
where the second inequality follows from $(\norm{\bm{X}_k}_{\rm F} + 1)^{1-\frac{1}{2^{n_k-1}}} \le 2$ and $\sqrt{\norm{\bm{X}_k}_{\rm F}} \ge 0$.
Therefore, regardless of the value of $\norm{\bm{X}_k}_{\rm F}$, we have
\begin{equation*}
\rho_k \le \frac{2(2 + a_{n_k} + b_{n_k} + c_{n_k})}{12^{n_k-1}},
\end{equation*}
and so $\liminf_{k\to \infty} \rho_k = 0$.
Since the sequences $(n_k)$ and $(\bm{X}_k)$ are arbitrary, we see that the error bound in \eqref{eq:computable_eb} is not asymptotically tight up to a dimension-free constant.

\section{Explicit error bounds for optimality system of SDPs}\label{subsec:optset_SDP}
Consider the following standard SDP
\begin{equation}
\begin{array}{c@{\quad}l}
\minimize_{\bm{X}\in \calS^n} & \langle \bm{C},\bm{X}\rangle \\
\subjectto &  {\cal A}(\bm{X})=\bm{b}, \ 
\bm{X} \in \mathcal{S}^n_+
\end{array} \label{eq:PSDP}\tag{PSDP}
\end{equation}
and its dual problem
\begin{equation}
\begin{array}{c@{\quad}l}
\maximize_{\bm{y}\in\mathbb{R}^m
} & \bm{b}^\top\bm{y} \\
\subjectto & \bm{C} - \mathcal{A}^*(\bm{y}) \in \calS_+^n,
\end{array} \label{eq:DSDP}\tag{DSDP}
\end{equation}
where $\bm{C} \in \mathcal{S}^n$, $\mathcal{A}\colon \mathcal{S}^n \to \mathbb{R}^m$ is a linear mapping, and $\bm{b} \in \mathbb{R}^m$. 
Denote the set of optimal solutions of \eqref{eq:PSDP} by $\calX^\star$, which we assume to be nonempty, and its optimal value by $\alpha^\star$.
Let
\begin{equation}
\calV^\star \coloneqq \{\bm{X}\in \calS^n \mid \calA(\bm{X})=\bm{b},\ \langle \bm{C},\bm{X}\rangle =\alpha^\star\}. \label{eq:Vstar}
\end{equation}
Then we have $\calX^\star = \calV^\star \cap \mathcal{S}^n_+$. 
It is worth noting that, even assuming the standard Slater's condition for \eqref{eq:PSDP}, that is, ${\cal A}^{-1}(\bm{b})\cap\setint \mathcal{S}^n_+\color{black}\neq\emptyset$,  the strict feasibility condition for the semidefinite feasibility problem describing the optimal solution set, that is $\Feas(\calV^\star,\calS_+^n)$, \emph{typically fails}.

Below, as an application, we provide an error bound result with an explicit estimate on the modulus of the radial-type error bound, under the commonly assumed strict complementarity condition.
We note that, in the special case where \eqref{eq:PSDP} has a unique solution (with other additional suitable assumptions), error estimates for feasible solutions can be deduced from the existing literature on the study of the conditioning of the so-called simple SDPs \cite{DYC+2021,DU2021} (see also \cite{NO1999}). See the discussions in Remark~\ref{rem:unique_sol_SDP} later for details.

Recall that for a feasible point $\bm{X}^\star$ of \eqref{eq:PSDP} and a feasible point $\bm{y}^\star$ of \eqref{eq:DSDP} with $\bm{S}^\star= {\bm C} - \mathcal{A}^*(\bm{y}^\star)$, they satisfy \emph{strict complementarity} if it follows that
\begin{equation}
\bm{X}^\star \in  \ri\calF^\star, \label{eq:strict}
\end{equation}
where $\calF^\star \coloneqq \calS_+^n \cap \{\bm{S}^\star\}^\perp$. 
In the above strict complementarity condition, $\bm{X}^\star$ and $\bm{y}^\star$ are indeed optimal solutions of \eqref{eq:PSDP} and \eqref{eq:DSDP}, respectively since \eqref{eq:strict} implies strong duality, i.e., $\langle \bm{C}, \bm{X}^\star\rangle = \bm{b}^\top \bm{y}^\star$.
In addition, by the strong duality between \eqref{eq:PSDP} and \eqref{eq:DSDP}, it follows that
\begin{equation*}
\calV^\star=\{\bm{X}\in \calS^n \mid \calA(\bm{X})=\bm{b},\ \langle \bm{S}^\star,\bm{X}\rangle = 0\}.
\end{equation*}
It can be readily verified that the above strict complementarity condition is equivalent to the statement that the sum of the ranks of $\bm{X}^\star$ and $\bm{S}^\star$ is equal to $n$.
The strict complementarity condition is a widely used condition in studying SDPs, and it is known that it holds generically~\cite{PT2002,AHO1997}.
Let $d$ be the rank of $\bm{S}^\star$.
We note that the face $\calF^\star$ is linearly isomorphic to $\calS_+^{n-d}$. Throughout this section, we always assume the following.
\begin{assumption}\label{asm:SDP}
\leavevmode
\begin{enumerate}[label=({A}\arabic*), ref={A}\arabic*]
\item The problem \eqref{eq:PSDP} and its dual \eqref{eq:DSDP} admit a pair of feasible points $(\bm{X}^\star,\bm{y}^\star)$ with $\bm{S}^\star= {\bm C} - \mathcal{A}^*(\bm{y}^\star)$ such that the strict complementarity condition \eqref{eq:strict} holds. \label{enum:strict_comp}
\item For the matrix $\bm{S}^\star= {\bm C} - \mathcal{A}^*(\bm{y}^\star)$, we assume that $\bm{S}^\star \neq \bm{O}$ (that is, the rank $d$ of $\bm{S}^\star$ is at least $1$).
\end{enumerate}
\end{assumption}

Note that if $\bm{S}^\star=\bm{O}$, then $\calV^\star=\{\bm{X}\in \calS^n \mid \calA(\bm{X})=\bm{b}\}$, and the assumption in \eqref{enum:strict_comp} implies that $\bm{X}^\star \in \calV^\star \cap \setint \mathcal{S}^n_+$.
Thus, the optimal solution set satisfies Slater's condition, and an explicit Lipschitz error bound holds (for example, by the case $d^*=0$ in Corollary \ref{cor:error_bound_constant}).
Therefore,  below
we only consider the case where $\bm{S}^\star \neq \bm{O}$. 

To formulate the explicit error bound, we first define
\begin{align}
\gamma &\coloneqq \frac{\norm{\bm{S}^\star}_{\rm F} + 1}{\lambda_{\rm min}^+(\bm{S}^\star)} + 1, \nonumber \\
\alpha &\coloneqq \sqrt{1 + 2(\gamma + 1)(n-d) + \gamma^2}, \label{eq:alpha_appl}\\
\beta  &\coloneqq \sqrt{2(\gamma + 1)\sqrt{n-d}}. \label{eq:beta_appl}
\end{align} 
{Let $\calM\colon \calS^n \to \bbR^p$ be a linear mapping with $p \ge 1$ whose kernel is $\setspan\calF^\star$.}
We define linear mappings $\overline{\calA}\colon \calS^n \to \bbR^{m+1}$ and $\widehat{\calA}\colon \calS^n \to \bbR^{m+p+1}$ respectively by
\begin{equation}
\overline{\calA}(\bm{X}) \coloneqq \begin{pmatrix}
\calA(\bm{X})\\
\langle \bm{S}^\star, \bm{X} \rangle
\end{pmatrix} \text{ and } \widehat{\calA}(\bm{X}) \coloneqq \begin{pmatrix}
\calA(\bm{X})\\
\calM(\bm{X}) \\
\langle \bm{S}^\star, \bm{X} \rangle 
\end{pmatrix}.  \label{eq:barA}
\end{equation}
By using the linear mapping $\overline{\calA}$, the affine space $\calV^\star$ can be represented as
\begin{equation}
\calV^\star = \left\{\bm{X}\in \calS^n \relmiddle| \overline{\calA}(\bm{X}) = \begin{pmatrix}
\bm{b}\\
0
\end{pmatrix}\right\}. \label{eq:Vstar_general}
\end{equation}  
Let $\sigma_{\rm min}^+(\overline{\calA})$ and  $\sigma_{\rm min}^+(\widehat{\calA})$ be the smallest positive singular values of $\overline{\calA}$ and $\widehat{\calA}$, respectively.\footnote{The linear mappings $\overline{\calA}$ and $\widehat{\calA}$ each have at least one positive singular value since the matrix $\bm{S}^\star$ is nonzero.} 
Recall that $\bbR\bm{S}^\star$ is the line along the matrix $\bm{S}^\star$, and, for a linear mapping $\mathcal{L}$ between two finite-dimensional normed spaces, we use $\norm{\mathcal{L}}_{\rm op}$ and $\norm{\mathcal{L}}_{\rm HS}$ to denote the operator norm of $\mathcal{L}$ and the Hilbert--Schmidt norm of $\mathcal{L}$, respectively. Finally, the function $\overline{\kappa}(\rho;\bm{X}^\star)$ with $\bm{X}^\star \in \calV^\star \cap \ri\calF^\star$ is defined as in \eqref{eq:def_barkappa} {for $\mathcal{V}=\calV^\star$ and $\calF_{d^*+1}=\calF^\star$, that is,
\begin{align}
\label{eq:kappa_bar}\overline{\kappa}(\rho;\bm{X}^\star) 
&= \max\{\eta(\calV^\star,\setspan\calF^\star),1\} \left(1 + 2\frac{\norm{\bm{X}^\star}_{\rm F} + \rho}{\lambda_{\rm min}^+(\bm{X}^\star)}\right) \\
& \le \frac{ \|\widehat{\calA}\|_{\rm HS}}{\sigma_{\rm min}^+(\widehat{\calA})}  \left(1 + 2\frac{\norm{\bm{X}^\star}_{\rm F} + \rho}{\lambda_{\rm min}^+(\bm{X}^\star)}\right) \nonumber,
\end{align}
where the inequality follows from \eqref{enum:eta_equality} of Lemma \ref{lem:Hoffman} and the fact that the Hilbert--Schmidt norm of a linear operator coincides with the Frobenius norm of its matrix representation.}

\begin{theorem}[\textbf{Error bounds for optimality system of SDPs with strict complementarity}]\label{KLG}
Consider \eqref{eq:PSDP} and its dual problem \eqref{eq:DSDP}.
Suppose that Assumption~\ref{asm:SDP} holds.
For each $\bm{X} \in \mathcal{S}^n$, let the optimality measure $\epsilon_{{\rm OM}}(\bm{X})$ be given by 
\begin{equation}
\epsilon_{{\rm OM}}(\bm{X}) \coloneqq \max\{|\langle \bm{S}^\star,\bm{X} \rangle|, \|\mathcal{A}(\bm{X})-\bm{b}\|_2,\dist(\bm{X},\mathcal{S}^n_+)\}. \label{eq:epsilonOM}
\end{equation}
\begin{enumerate}[label=(\roman*), ref=\roman*, font=\upshape]
\item For all $\bm{X} \in \mathcal{S}^n$, we have  
\begin{equation*}
\dist(\bm{X}, \calX^\star) \le \alpha_{\rm SDP}(\|\bm{X}\|_{\rm F}) \epsilon_{\rm OM}(\bm{X})+ \beta_{\rm SDP}(\|\bm{X}\|_{\rm F}) \sqrt{\epsilon_{\rm OM}(\bm{X})},
\end{equation*}
where $\alpha_{\rm SDP}(\cdot)$ and $\beta_{\rm SDP}(\cdot)$ are defined as
\begin{align}
\alpha_{\rm SDP}(\rho) &\coloneqq \overline{\kappa}(\rho;\bm{X}^\star)\max\left\{\frac{1}{\sigma_{\rm min}^+(\overline{\calA})}\left(1+\frac{\|\mathcal{A}|_{\bbR\bm{S}^\star}\|_{\rm op}}{\|\bm{S}^\star\|_{\rm F}}\right), \alpha+ \frac{1}{\|\bm{S}^\star\|_{\rm F}}\right\} +\frac{1}{\|\bm{S}^\star\|_{\rm F}}, \label{eq:alphaSDP}\\
\beta_{\rm SDP}(\rho) &\coloneqq \beta \sqrt{\rho} \ \overline{\kappa}(\rho;\bm{X}^\star) , \label{eq:betaSDP}
\end{align}
respectively, and $\overline{\kappa}$, $\alpha$, and $\beta$ are given as in \eqref{eq:kappa_bar}, 
\eqref{eq:alpha_appl}, and \eqref{eq:beta_appl}, respectively. \label{enum:infeasible_X}
\item For all $\bm{X} \in \calA^{-1}(\bm{b}) \cap \calS_+^n$, we have
\begin{equation*}
\dist(\bm{X},\calX^\star) \le \overline{\alpha}_{\rm SDP}(\|\bm{X}\|_{\rm F})(\langle \bm{C}, \bm{X} \rangle -\alpha^\star)+ \beta_{\rm SDP}(\|\bm{X}\|_{\rm F}) \sqrt{\langle \bm{C}, \bm{X} \rangle -\alpha^\star},
\end{equation*}
where 
\begin{equation*}
\overline{\alpha}_{\rm SDP}(\rho) \coloneqq \overline{\kappa}(\rho;\bm{X}^\star) \max\left\{ \frac{\|\mathcal{A}|_{\bbR\bm{S}^\star}\|_{\rm op}}{\sigma_{\rm min}^+(\overline{\calA})\|\bm{S}^\star\|_{\rm F}}, \alpha+\frac{1}{\|\bm{S}^\star\|_{\rm F}}\right\} +\frac{1}{\|\bm{S}^\star\|_{\rm F}}.
\end{equation*}
\label{enum:feasible_X}
\end{enumerate}
\end{theorem}
\begin{proof}
First, we prove \eqref{enum:infeasible_X}.
Let $\bm{X} \in \calS^n$.
We decompose it as 
\begin{equation*}
\bm{X}= \underbrace{\bm{X}-\frac{\langle \bm{S}^\star,\bm{X}\rangle}{\|\bm{S}^\star\|_{\rm F}^2}\bm{S}^\star}_{=P_{\{\bm{S}^\star\}^\perp}(\bm{X})} + \underbrace{\frac{\langle \bm{S}^\star,\bm{X}\rangle}{\|\bm{S}^\star\|_{\rm F}^2}\bm{S}^\star}_{=P_{\bbR\bm{S}^\star}(\bm{X})}.
\end{equation*}
Then we have
\begin{align}
&\dist(\bm{X},\calX^\star) \nonumber\\
\overset{\scriptsize \text{(a)}}\le{}& \dist (P_{\{\bm{S}^\star\}^\perp}(\bm{X}),\calX^\star) + \frac{\epsilon_{\rm OM}(\bm{X})}{\|\bm{S}^\star\|_{\rm F}}\nonumber\\
\overset{\scriptsize \text{(b)}}={} & \dist(P_{\{\bm{S}^\star\}^\perp}(\bm{X}), \calV^\star \cap \calF^\star) + \frac{\epsilon_{\rm OM}(\bm{X})}{\|\bm{S}^\star\|_{\rm F}}\nonumber\\
\overset{\scriptsize \text{(c)}}\le{} & \overline{\kappa}(\|P_{\{\bm{S}^\star\}^\perp}(\bm{X})\|_{\rm F};\bm{X}^\star)\max\{\dist(P_{\{\bm{S}^\star\}^\perp}(\bm{X}),\calV^\star), \dist(P_{\{\bm{S}^\star\}^\perp}(\bm{X}),\calF^\star)\}+ \frac{\epsilon_{\rm OM}(\bm{X})}{\|\bm{S}^\star\|_{\rm F}} \nonumber\\
\overset{\scriptsize \text{(d)}}\le{} & \overline{\kappa}(\|\bm{X}\|_{\rm F};\bm{X}^\star) \max\left\{\frac{\|\mathcal{A}(P_{\{\bm{S}^\star\}^\perp}(\bm{X}))-\bm{b}\|_2}{\sigma_{\rm min}^+(\overline{\calA})}, \dist(P_{\{\bm{S}^\star\}^\perp}(\bm{X}),\calF^\star)\right\} + \frac{\epsilon_{\rm OM}(\bm{X})}{\|\bm{S}^\star\|_{\rm F}}\nonumber \\
\overset{\scriptsize \text{(e)}}\le{} & \overline{\kappa}(\|\bm{X}\|_{\rm F};\bm{X}^\star) \max\left\{ \frac{1}{\sigma_{\rm min}^+(\overline{\calA})} \left(\|\mathcal{A}(\bm{X})-\bm{b}\|_2 + \|\mathcal{A}|_{\bbR\bm{S}^\star}\|_{\rm op}\frac{\epsilon_{\rm OM}(\bm{X})}{\|\bm{S}^\star\|_{\rm F}}\right), \dist(P_{\{\bm{S}^\star\}^\perp}(\bm{X}),\calF^\star)\right\} + \frac{\epsilon_{\rm OM}(\bm{X})}{\|\bm{S}^\star\|_{\rm F}}\label{eq:use0} \\
\overset{\scriptsize \text{(f)}}\le{} & \overline{\kappa}(\|\bm{X}\|_{\rm F};\bm{X}^\star) \max\left\{\frac{\epsilon_{\rm OM}(\bm{X})}{\sigma_{\rm min}^+(\overline{\calA})} \left(1 + \frac{\|\mathcal{A}|_{\bbR\bm{S}^\star}\|_{\rm op}}{\|\bm{S}^\star\|_{\rm F}}\right), \dist(P_{\{\bm{S}^\star\}^\perp}(\bm{X}),\calF^\star)\right\} + \frac{\epsilon_{\rm OM}(\bm{X})}{\|\bm{S}^\star\|_{\rm F}}, \label{eq:use1}
\end{align}
where (a) follows from the triangular inequality and the definition of $\epsilon_{\rm OM}(\bm{X})$, (b) follows from $\calX^\star = \calV^\star \cap \calF^\star$, (c) is obtained by applying Proposition~\ref{prop:eb_reduced} with $\mathcal{V}=\calV^\star$ and $\calF_{d^*+1}=\calF^\star$ and then using the inequality $\kappa(\norm{P_{\{\bm{S}^\star\}^\perp}(\bm{X})}_{\rm F}) \le \overline{\kappa}(\norm{P_{\{\bm{S}^\star\}^\perp}(\bm{X})}_{\rm F};\bm{X}^\star)$, (d) follows from \eqref{eq:Hoffman_eq}, \eqref{eq:Vstar_general}, and the nonexpansiveness of the projection $P_{\{\bm{S}^\star\}^\perp}$, (e) follows from the triangular inequality and the definition of $\epsilon_{\rm OM}(\bm{X})$, and (f) also follows from the definition of $\epsilon_{\rm OM}(\bm{X})$.
On the quantity $\dist(P_{\{\bm{S}^\star\}^\perp}(\bm{X}),\calF^\star)$ that appears in \eqref{eq:use1}, it follows from Proposition~\ref{prop:1FRF} with $\epsilon=\epsilon_{\rm OM}(\bm{X})$ that   
\begin{align}
\dist(P_{\{\bm{S}^\star\}^\perp}(\bm{X}),\calF^\star) &= \dist (P_{\{\bm{S}^\star\}^\perp}(\bm{X}), \calS_+^n \cap \{\bm{S}^\star\}^\perp) \nonumber \\
&\le \dist (\bm{X}, \calS_+^n \cap \{\bm{S}^\star\}^\perp) + \|\bm{X}-P_{\{\bm{S}^\star\}^\perp}(\bm{X})\|_{\rm F} \nonumber\\
& \le \left(\alpha + \frac{1}{\|\bm{S}^\star\|_{\rm F}}\right)\epsilon_{\rm OM}(\bm{X}) + {\beta}\sqrt{\epsilon_{\rm OM}(\bm{X})\lVert\bm{X}\rVert_{\rm F}},\label{eq:dist_PSperpX_Fstar}
\end{align}
where $\alpha$ and $\beta$ are given as in \eqref{eq:alpha_appl} and \eqref{eq:beta_appl}, respectively.
By \eqref{eq:use1} and \eqref{eq:dist_PSperpX_Fstar}, we obtain
\begin{align*}
&\dist(\bm{X},\calX^\star) \\
\le{}& \overline{\kappa}(\|\bm{X}\|_{\rm F};\bm{X}^\star) \max\left\{\frac{\epsilon_{\rm OM}(\bm{X})}{\sigma_{\rm min}^+(\overline{\calA})}\left(1+\frac{\|\mathcal{A}|_{\bbR\bm{S}^\star}\|_{\rm op}}{\|\bm{S}^\star\|_{\rm F}}\right), \left(\alpha + \frac{1}{{\|\bm{S}^\star\|_{\rm F}}}\right) \epsilon_{\rm OM}(\bm{X}) +  \beta \sqrt{\epsilon_{\rm OM}(\bm{X}) \|\bm{X}\|_{\rm F}}\right\} + \frac{\epsilon_{\rm OM}(\bm{X})}{\|\bm{S}^\star\|_{\rm F}}\\
\le{}& \alpha_{\rm SDP}(\|\bm{X}\|_{\rm F}) \epsilon_{\rm OM}(\bm{X})+ \beta_{\rm SDP}(\|\bm{X}\|_{\rm F}) \sqrt{\epsilon_{\rm OM}(\bm{X})}.
\end{align*}

Next, we prove \eqref{enum:feasible_X}.
Let $\bm{X} \in \mathcal{A}^{-1}(\bm{b}) \cap \mathcal{S}^n_+$.
Then it follows from $\bm{X} \in \mathcal{A}^{-1}(\bm{b}) \cap \mathcal{S}^n_+$ and $\bm{S}^\star \in \mathcal{S}^n_+$ that $\epsilon_{\rm OM}(\bm{X}) = \langle \bm{C}, \bm{X}\rangle -\alpha^\star$.
Moreover, by $\bm{X} \in \mathcal{A}^{-1}(\bm{b}) \cap \mathcal{S}^n_+$, \eqref{eq:use0} can be simplified as
\begin{equation*}
\dist(\bm{X},\calX^\star) \le \overline{\kappa}(\|\bm{X}\|_{\rm F};\bm{X}^\star) \max\left\{ \frac{\|\mathcal{A}|_{\bbR\bm{S}^\star}\|_{\rm op}\epsilon_{\rm OM}(\bm{X})}{\sigma_{\rm min}^+(\overline{\calA})\|\bm{S}^\star\|_{\rm F}}, \dist(P_{\{\bm{S}^\star\}^\perp}(\bm{X}),\calF^\star)\right\} + \frac{\epsilon_{\rm OM}(\bm{X})}{\|\bm{S}^\star\|_{\rm F}}.
\end{equation*}
Combining this inequality with \eqref{eq:dist_PSperpX_Fstar}, we obtain the desired result.
\end{proof}

\begin{remark}[\textbf{Links to existing works}]\label{rem:unique_sol_SDP}
We note that, in addition to Assumption~\ref{asm:SDP}, if we further assume that \eqref{eq:PSDP} has a unique optimal solution and ${\cal A}^{-1}(\bm{b})\cap\setint \mathcal{S}^n_+\color{black}\neq\emptyset$ (that is, standard Slater's condition holds for \eqref{eq:PSDP}), similar error estimates for feasible solutions can be deduced from the existing literature on the study of the conditioning of the so-called simple SDPs~\cite{DYC+2021,DU2021}. 

More precisely, suppose that \eqref{eq:PSDP} has a unique optimal solution, i.e., $\calX^\star = \{\bm{X}^\star\}$. 
Recall that $\calF^\star=\mathcal{S}^n_+ \cap \{\bm{S}^\star\}^\perp$.
Then the mapping $\mathcal{A}|_{\setspan\calF^\star}$ is injective.
(Otherwise, we can take $\bm{U} \in (\setspan\calF^\star)\setminus \{\bm{O}\}$ such that $\calA(\bm{U})$ is zero.
Then the matrix $\bm{X}(t)\coloneqq\bm{X}^\star + t\bm{U}$ defined for every $t>0$ belongs to $\mathcal{A}^{-1}(\bm{b}) \cap  \setspan\calF^\star$.
As $\bm{X}^\star \in \ri \calF^\star$, for sufficiently small $t>0$, we have $\bm{X}(t) \in \calF^\star$.
This implies that $\bm{X}(t) \in \calX^\star$, which contradicts the assumption that $\calX^\star$ is a singleton.)
Thus, the minimum singular value of $\mathcal{A}|_{\setspan\calF^\star}$, which is denoted by $\sigma_{\rm min}(\mathcal{A}|_{\setspan\calF^\star})$ and calculated as
\begin{equation}
\sigma_{\rm min}(\mathcal{A}|_{\setspan\calF^\star}) = \min\{\|\mathcal{A}(\bm{X})\|_2 \mid \bm{X}\in \setspan\calF^\star,\ \|\bm{X}\|_{\rm F}=1\}, \label{eq:min_singlarval}
\end{equation}
is positive.

Let $\bm{X} \in \mathcal{A}^{-1}(\bm{b}) \cap \mathcal{S}_+^n$ be arbitrary.
It follows that 
\begin{align}
\dist(\bm{X},\calX^\star) &\overset{\scriptsize \text{(a)}}= \norm{\bm{X} - \bm{X}^\star}_{\rm F} \nonumber \\
&\le \norm{P_{\setspan\calF^\star}(\bm{X}) - \bm{X}^\star}_{\rm F} + \|\bm{X}-P_{\setspan\calF^\star}(\bm{X})\|_{\rm F} \nonumber \\
&\overset{\scriptsize \text{(b)}}\le \frac{\|\mathcal{A}(P_{\setspan\calF^\star}(\bm{X})) -\bm{b}\|_2}{\sigma_{\rm min}(\mathcal{A}|_{\setspan\calF^\star})} + \|\bm{X}-P_{\setspan\calF^\star}(\bm{X})\|_{\rm F} \nonumber \\
&\overset{\scriptsize \text{(c)}}= \frac{\|\mathcal{A}(\bm{X}-P_{{\setspan\calF^\star}}(\bm{X}))\|_2}{\sigma_{\rm min}(\mathcal{A}|_{\setspan\calF^\star})} + \|\bm{X}-P_{\setspan\calF^\star}(\bm{X})\|_{\rm F} \nonumber \\
& \le \left(1+\frac{\|\mathcal{A}|_{(\setspan\calF^\star)^{\perp}}\|_{\rm op}}{\sigma_{\rm min}(\mathcal{A}|_{\setspan\calF^\star})} \right)\|\bm{X}-P_{\setspan\calF^\star}(\bm{X})\|_{\rm F}, \label{eq:dist_X_uniqueoptsol}
\end{align}
where we use the assumption that $\calX^\star = \{\bm{X}^\star\}$ to derive (a), use \eqref{eq:min_singlarval} and $\calA(\bm{X}^\star) = \bm{b}$ to derive (b), use $\calA(\bm{X}) = \bm{b}$ to derive (c).
Now, from \cite[Lemma~4.3]{DYC+2021} and $\bm{X} \in \mathcal{S}^n_{+}$, it follows that
\begin{equation}
\|\bm{X}-P_{\setspan\calF^\star}(\bm{X})\|_{\rm F} \le \frac{\langle\bm{S}^\star,\bm{X}\rangle}{\lambda_{\rm min}^+(\bm{S}^\star)} + \sqrt{2 \frac{\langle \bm{S}^\star,\bm{X}\rangle}{\lambda_{\rm min}^+(\bm{S}^\star)} \lambda_{\max}(\bm{X})}. \label{eq:norm_X_minus_PspanFX}
\end{equation}
By \eqref{eq:dist_X_uniqueoptsol} and \eqref{eq:norm_X_minus_PspanFX}, we obtain
\begin{equation*}
\dist(\bm{X}, \calX^\star) \le \left(1+\frac{\|\mathcal{A}|_{(\setspan\calF^\star)^{\perp}}\|_{\rm op}}{\sigma_{\rm min}(\mathcal{A}|_{\setspan\calF^\star})}\right)\left(\frac{\epsilon_{\rm OM}(\bm{X})}{\lambda_{\rm min}^+(\bm{S}^\star)}+\sqrt{2 \frac{\epsilon_{\rm OM}(\bm{X})}{\lambda_{\rm min}^+(\bm{S}^\star)} \|\bm{X}\|_{\rm F}}\right).
 \end{equation*}

Different from the error estimate provided in \eqref{enum:feasible_X} of Theorem~\ref{KLG}, the above estimate relies heavily on the fact that $\sigma_{\rm min}(\mathcal{A}|_{\setspan\calF^\star}) > 0$, which holds due to the uniqueness of the optimal solution of \eqref{eq:PSDP}.
\end{remark}

In the following example, we illustrate the error bound shown in \eqref{enum:infeasible_X}.
For this instance, we see that the optimal solution set of the primal problem is not a singleton (and so, the results in the existing literature mentioned above are not applicable). Moreover, this example shows that the error bound can be asymptotically tight up to a dimension-free constant.

\begin{example}[\textbf{Asymptotic tightness for explicit error bound for the optimality system of SDP}]\label{ex:infeasible_X} 
For $n \in \bbZ_{\ge 3}$, we consider the following SDP:
\begin{equation}
{\everymath{\displaystyle}
\begin{array}{c@{\quad}l}
\minimize_{\bm{X} \in \mathcal{S}^n} & \sum_{i=3}^n X_{ii}\\
\subjectto & \sum_{i=1}^n X_{ii}=1, \ 
\bm{X} \in \mathcal{S}^n_+.
\end{array}}
\label{prob:corr_mat}
\end{equation}
By letting $\bm{C} \coloneqq  \Diag(0,0,1,\dots,1)$, $\calA(\bm{X}) \coloneqq \sum_{i=1}^n X_{ii}$, and $b \coloneqq 1$, \eqref{eq:PSDP} reduces to \eqref{prob:corr_mat}.
The dual problem of Problem~\eqref{prob:corr_mat} can be written as 
\begin{equation*}
{\everymath{\displaystyle}
\begin{array}{c@{\quad}l}
\maximize_{y\in\mathbb{R}} & y\\
\subjectto & \Diag(-y,-y,1-y,\dots,1-y) \in \calS_+^n.
\end{array}}
\end{equation*}
It can be seen that an optimal solution of the primal problem is $\bm{X}^\star \coloneqq \Diag(1/2,1/2,0,\dots,0)$, and the unique optimal solution of the dual problem is $y^\star = 0$ with the corresponding slack matrix $\bm{S}^\star = \Diag(0,0,1,\dots,1)$.
The optimal value of these problems is $\alpha^\star = 0$.
We note that the matrix $\bm{X}^\star$ is the same as that used in Section~\ref{sec:dPPS_1}.
For Problem~\eqref{prob:corr_mat}, by $\alpha^\star = 0$, the affine space $\calV_n^\star$ corresponding to \eqref{eq:Vstar} can be written as
\begin{equation*}
\calV_n^\star = \left\{\bm{A} \in \calS^n \relmiddle| \sum_{i=1}^n A_{ii} = 1,\ \sum_{i=3}^n A_{ii} = 0\right\}.
\end{equation*}
We note that the affine space $\calV_n^\star$ is the same as that introduced in \eqref{eq:dPPS=1_V}.
The set $\calX^{\star} = \calV_n^\star \cap \calS_+^n$ of optimal solutions of the primal problem is
\begin{equation*}
\{\Diag(\bm{A},\bm{O}) \in \calS^n \mid \bm{A}\in \calS_+^2,\ A_{11} + A_{22} = 1\},
\end{equation*}
and the primal problem has nonunique optimal solutions.
The optimal solution $\bm{X}^\star$ of the primal problem and the optimal solution $y^\star$ of the dual problem satisfy the strict complementarity condition in \eqref{eq:strict}.
Using $\bm{X}^\star$, we see from \eqref{enum:infeasible_X} of Theorem~\ref{KLG} that
\begin{align}
\dist(\bm{X},\calX^\star) &\le \left(\sqrt{2}(1 + 2\sqrt{2} + 4\norm{\bm{X}}_{\rm F})\left(\sqrt{n + 8\sqrt{n-2} + 15} + \frac{1}{\sqrt{n-2}}\right) + \frac{1}{\sqrt{n-2}}\right)\epsilon_{\rm OM}(\bm{X}) \nonumber \\
&\quad + 2\sqrt{\sqrt{2(n-2)} + 3\sqrt{2}}(1+2\sqrt{2}+4\norm{\bm{X}}_{\rm F})\sqrt{\norm{\bm{X}}_{\rm F}}\sqrt{\epsilon_{\rm OM}(\bm{X})} \text{ for all $\bm{X} \in \calS^n$,} \label{eq:eb_infeasible_ex}
\end{align}
where $\epsilon_{\rm OM}(\bm{X})$ defined in \eqref{eq:epsilonOM} is
\begin{equation*}
\epsilon_{\rm OM}(\bm{X}) = \max\left\{\left|\sum_{i=3}^n X_{ii}\right|, \left|\sum_{i=1}^n X_{ii} - 1\right|, \dist(\bm{X},\calS_+^n)\right\}
\end{equation*}
and see Appendix~\ref{apdx:alphaSDP_betaSDP} for the calculations of $\alpha_{\rm SDP}(\rho)$ and $\beta_{\rm SDP}(\rho)$ defined in \eqref{eq:alphaSDP} and \eqref{eq:betaSDP}, respectively.
In what follows, we show that the error bound in \eqref{eq:eb_infeasible_ex} is asymptotically tight up to a dimension-free constant.

For every $k\in \bbZ_{\ge 1}$, let $n_k \coloneqq k + 3$ and $\bm{X}_k$ be the matrix defined in \eqref{eq:Xk_infeasible}.
As calculated in \eqref{eq:lhs_eb_infeasible}, the left-hand side of \eqref{eq:eb_infeasible_ex} with $\bm{X} = \bm{X}_k$ is
\begin{equation*}
\dist(\bm{X}_k,\calX^\star) = \dist(\bm{X}_k,\calV_{n_k} \cap \calS_+^{n_k}) = \sqrt{\frac{k+1}{k}} \to 1 \qquad (\text{as}\,\,k\to \infty).
\end{equation*}
We then calculate the right-hand side of \eqref{eq:eb_infeasible_ex}.
It can be seen that $\abs{\sum_{i=3}^{n_k}(\bm{X}_k)_{ii}} = 0$ and $\abs{\sum_{i=1}^{n_k}(\bm{X}_k)_{ii} - 1} = 0$.
In addition, it follows that $\dist(\bm{X}_k,\calS_+^{n_k}) = 1/\sqrt{k}$ as seen in Section~\ref{sec:dPPS_1}.
By combining them, we have $\epsilon_{\rm OM}(\bm{X}_{k}) = 1/\sqrt{k}$, and the right-hand side of \eqref{eq:eb_infeasible_ex} with $\bm{X}=\bm{X}_{k}$ {and $n=n_k$} is
\begin{align*}
\left(\sqrt{2}\left(1 + 2\sqrt{2} + 4\sqrt{\frac{2k+1}{k}}\right)\left(\sqrt{k + 8\sqrt{k+1} + 18} + \frac{1}{\sqrt{k+1}}\right) + \frac{1}{\sqrt{k+1}}\right)\frac{1}{\sqrt{k}}&\\
+ 2\sqrt{\sqrt{2(k+1)} + 3\sqrt{2}}\left(1+2\sqrt{2}+4\sqrt{\frac{2k+1}{k}}\right)\sqrt[4]{\frac{2k+1}{k}}\frac{1}{\sqrt[4]{k}} & \to 3(12+\sqrt{2})\qquad (\text{as}\,\,k\to \infty).
\end{align*}
Thus, the error bound in \eqref{eq:eb_infeasible_ex} is asymptotically tight up to the dimension-free constant $\frac{1}{3(12+\sqrt{2})}$.
\end{example}

\section{Conclusions and remarks}\label{sec:conclusion}
In this paper, we investigated radial-type H\"{o}lder error bounds for semidefinite feasibility problems $\Feas(\calV_n,\calS_+^n)$ without assuming {any constraint qualifications}.
The final objective was to furnish error bounds for which the associated constants have explicit expressions. This was 
accomplished via facial reduction and the framework of facial residual functions~\cite{LLP2023}.

Furthermore, we introduced the notion of asymptotic tightness of error bounds and analyzed the asymptotic tightness of the qualitative bounds derived in this paper in terms of $d^* = d_{\rm PPS}(\calV_n,\calS_+^n)$, the distance to the PPS condition of $\Feas(\calV_n,\calS_+^n)$.
Our analysis showed that the derived radial-type error bounds can be asymptotically tight up to a dimension-free constant factor, when $d^*$ equals $0$ or $1$.
We also discussed the more challenging case where $d^*=n-1$.

As an application, we derived an explicit error bound for the optimality system of SDPs satisfying strict complementarity.
Interestingly, this explicit bound was obtained without assuming the usual uniqueness of the optimal solution.
Moreover, we presented an example showing that it can be asymptotically tight up to a dimension-free constant factor.

{A limitation of our results is that the notion of tightness considered here is asymptotic and is defined in terms of sequences of instances; see Definition~\ref{def:tight_eb}. For $d^* \in \{0,1\}$, the error bound \eqref{eq:computable_eb} is asymptotically tight in the sense that there exist sequences of instances whose dimensions tend to infinity, for which the condition in Definition~\ref{def:tight_eb} is satisfied. As discussed in Remark~\ref{rem:int}, this shows that, for these families of instances, the derived error bound constants are best possible up to a dimension-free constant factor.

In order to dispel potential sources of confusion, it is important to emphasize what this asymptotic tightness statement does and does not imply. In particular, we do not claim that, for \emph{every} family of instances with $d^* \in \{0,1\}$, the error bound in  \eqref{eq:computable_eb} is asymptotically tight. Nor do we claim that, for a \emph{single specific instance}, that is, for fixed $\calV$ and $\calS_+^n$, the derived constants in the error bound  \eqref{eq:computable_eb}  are necessarily best possible. The error bound holds for each individual instance, but its constants may be conservative for particular instances. Indeed, as observed in Section~\ref{sec:dPPS_n-1}, in the most degenerate case where $d^* = n-1$, the derived error bound appears to be rather conservative.}

Getting a reasonable expression for the best possible constants for one specific semidefinite feasibility problem remains a significant research challenge but we hope this work may be a useful first step.
It would be of interest to investigate whether our bounds can be improved, for example, by exploiting more refined geometric properties of the positive semidefinite cone. 

On the algorithmic side, it would be interesting to understand how the explicit error bounds obtained in this paper can be used to derive explicit convergence rate estimates for some of the relevant numerical algorithms in the literature.

\vspace{0.5cm}
\noindent
{\bf Acknowledgments}
The first author is supported by Japan Society for the Promotion of Science (JSPS) Grant-in-Aid for Research Activity Start-up JP25K23344.
The second author is partially supported by the Australian Research Council under Discovery Projects DP190100555 and DP250101112.
The third author is supported by JSPS Grant-in-Aid for Early-Career Scientists JP23K16844 and Japan Science and Technology Agency ASPIRE Grant Number JPMJAP2520.

\vspace{0.5cm}
\noindent
\textbf{Statements and declarations}
\begin{itemize}
\item Competing interests: The authors declare that there are no competing interests.
\item Availability of data and materials: No data were used for the research described in the article.
\item Use of AI: Preliminary versions of Lemmas~\ref{lem:ub_hi} and \ref{lem:ub_cij} were developed with the assistance of ChatGPT.
The proofs in the present version were subsequently rewritten and verified by the authors, who assume full responsibility for their correctness. 
\end{itemize}

\begin{appendices}
\section{\large The proof of \texorpdfstring{\eqref{eq:dpps_eq_ds}}{(2.4)}} \label{apdx:dpps_eq_ds}
{\small
\begin{proof}
For simplicity, we write $d^*$ for $d_{\rm PPS}(\calV,\calS_+^n)$.
It is sufficient to show that $d_{\rm s} \le d^*$.
Let $\calF_{d^*+1} \subsetneq \cdots \subsetneq \calF_1$ be a chain of faces of $\calS_+^n$.
By the definition of $d^*$, the face $\calF_{d^*+1}$ is either polyhedral or satisfies $\calV \cap \ri\calF_{d^*+1} \neq \emptyset$.
If $\calV \cap \ri\calF_{d^*+1} \neq \emptyset$, then, by the definition of $d_{\rm s}$, we have $d_{\rm s} \le d^*$.
It remains to consider the case in which $\calV \cap \ri\calF_{d^*+1} = \emptyset$.
Then $\calF_{d^*+1}$ must be polyhedral.
The polyhedral faces of $\calS_+^n$ are precisely the $0$-dimensional face $\{\bm{O}\}$ and the $1$-dimensional faces $\bbR_+\bm{A}$ with $\bm{A}$ being a positive semidefinite matrix of rank $1$.
If $\calF_{d^*+1} = \{\bm{O}\}$, we have $\calV \cap \calS_+^n = \{\bm{O}\}$.
As shown in the proof of \cite[Proposition~27]{Lourenco2021}, we can reach the polyhedral face $\{\bm{O}\}$ of $\calS_+^n$ in only one facial reduction step, and $\calV \cap \ri\{\bm{O}\} \neq \emptyset$.
Under the assumption that $\calV \cap \ri\calF_{d^*+1} = \emptyset$, we have $\calV \cap \setint\calS_+^n = \emptyset$.
Moreover, $\calS_+^n$ is not polyhedral since $n\ge 2$.
Therefore, we obtain $d_{\rm s} = d^* = 1$.
If $\calF_{d^*+1} = \bbR_+\bm{A}$ with $\bm{A}$ being a positive semidefinite matrix of rank $1$, then $\ri\calF_{d^*+1} = \{t\bm{A} \mid t > 0\}$.
Since $\calV \cap \ri\calF_{d^*+1} = \emptyset$ and the problem $\Feas(\calV,\calS_+^n)$ is feasible, we have $\calV \cap \calS_+^n = \calV \cap \calF_{d^*+1} = \{\bm{O}\}$.
Therefore, by the same argument as in the case where $\calF_{d^*+1} = \{\bm{O}\}$, we again obtain $d_{\rm s} = d^* = 1$.
\end{proof}
}

\section{\large The proof of Lemma~\ref{lem:Hoffman}} \label{apdx:proof_lemma}
{\small 
\begin{proof}
Here, we only prove \eqref{enum:eta_inequality}.
The statement in \eqref{enum:eta_equality} can be proven in a similar fashion.
Let $\bm{a}_{ij}^\top$ denote the $j$th row of $\bm{A}_i$, and let $b_{ij}$ denote the $j$th element of $\bm{b}_i$.
For notational convenience, we define $C_{ij} \coloneqq \{\bm{x} \in \bbR^n \mid \bm{a}_{ij}^\top \bm{x} \le b_{ij}\}$ for each $i = 1,2$ and $j = 1,\dots,m_i$.
We note that
\begin{equation*}
C_1 \cap C_2 = \bigcap_{i=1}^2\bigcap_{j=1}^{m_i}C_{ij} = \{\bm{x} \in \bbR^n \mid \bm{a}_{ij}^\top \bm{x} \le b_{ij}\ (i=1,2,\ j = 1,\dots,m_i)\} = \{\bm{x}\in \bbR^n \mid \bm{A}\bm{x} \le \bm{b}\},
\end{equation*}
where $\bm{b}$ is the vector obtained by concatenating the vector $\bm{b}_1$ with the vector $\bm{b}_2$.
Then for any $\bm{x} \in \bbR^n$, we have
\begin{align*}
\dist(\bm{x},C_1 \cap C_2) &\le H(\bm{A}) \sqrt{\sum_{i=1}^2\sum_{j=1}^{m_i} \max\{\bm{a}_{ij}^\top \bm{x} - b_{ij},0\}^2}\\ 
&= H(\bm{A})\sqrt{\sum_{i=1}^2\sum_{j=1}^{m_i} \lVert\bm{a}_{ij}\rVert_2^2 \dist(\bm{x},C_{ij})^2}\\
&\le H(\bm{A})\norm{\bm{A}}_{\rm F}\max_{\substack{i=1,2,\\j=1,\dots,m_i}}\dist(\bm{x},C_{ij})\\
&\le H(\bm{A})\norm{\bm{A}}_{\rm F}\max\{\dist(\bm{x},C_1),\dist(\bm{x},C_2)\},
\end{align*}
where the first inequality follows from the definition of the Hoffman constant $H(\bm{A})$, the equality follows from the formula for the distance to a half-space~\cite[Theorem~9.39]{Deutsch2001}, and the last inequality follows from the fact that $C_i \subseteq C_{ij}$ for each $i = 1,2$ and $j = 1,\dots,m_i$.
Thus, the conclusion holds.
\end{proof}
}

\section{\large Computation of the quantities in Section~\ref{sec:dPPS_1}}\label{apdx:calc_comps}
{\small In this appendix, we illustrate the computations of $\overline{\kappa}(\rho;\bm{X}^\star)$, $\alpha_1$, and $\beta_1$ used in Section~\ref{sec:dPPS_1}.
For simplicity, we consider the case where the ambient space is $\calS^n$ instead of $\calS^{n_k}$ with $n_k = k + 3$, because we can obtain the result for the case of $\calS^{n_k}$ by substituting $n_k$ into $n$.

\noindent
\fbox{$\overline{\kappa}(\rho;\bm{X}^\star) = \sqrt{2}(1 + 2\sqrt{2} + 4\rho)$}
Recall that $\overline{\kappa}(\rho;\bm{X}^\star) = \max\{\eta(\calV_n,\setspan\calF_2),1\}\overline{\theta}(\rho;\bm{X}^\star)$.
The value $\overline{\theta}(\rho;\bm{X}^\star)$, which is defined in \eqref{eq:bar_theta}, is $1 + 2\sqrt{2} + 4\rho$.
In what follows, we show that $\eta(\calV_n,\setspan\calF_2) = \sqrt{2}$.
Before computing $\eta(\calV_n,\setspan\calF_2)$, we prepare an additional lemma regarding the direct sum decomposition of a subspace of $\calS^n$.

\begin{lemma}\label{lem:subspace_direct_sum}
Let $V_1$ and $V_2$ be subspaces of $\calS^n$.
If $P_{V_2}(V_1)$ is included in $V_1$, then we have
\begin{equation*}
V_1 = (V_1 \cap V_2) \oplus (V_1 \cap V_2^\perp).
\end{equation*}
\end{lemma}

The proof of Lemma~\ref{lem:subspace_direct_sum} is omitted because it is straightforward.
We note that the statement does not hold if $P_{V_2}(V_1)$ is not included in $V_1$.

The face $\calF_2$ of $\calS_+^n$ exposed by the matrix $\bm{Z}_1$ is provided by \eqref{eq:d*=1_F2}.
Let
\begin{equation*}
\calW \coloneqq \calV_n - \bm{X}^\star = \left\{\bm{A}\in \calS^n \relmiddle| \sum_{i=1}^n A_{ii} = 0,\ \sum_{i=3}^n A_{ii} = 0\right\}
\end{equation*}
for simplicity.
We observe that $\eta(\calV_n,\setspan\calF_2)$, denoted by $\eta$ for brevity, is the minimum positive number satisfying the following inequality:
\begin{equation*}
\dist(\bm{X},\calW \cap \setspan\calF_2) \le \eta\max\{\dist(\bm{X},\calW),\dist(\bm{X},\setspan\calF_2)\} \text{ for all $\bm{X} \in \calS^n$.}
\end{equation*}
Let $\bm{D}_1 \coloneqq \bm{E}_{11} + \bm{E}_{22}$ and $\bm{D}_2 \coloneqq \sum_{i=3}^n \bm{E}_{ii}$.
We note that $\calW = \{\bm{D}_1,\bm{D}_2\}^\perp$.
It follows that
\begin{equation*}
\calS^n = \calW \oplus \calW^\perp = \underbrace{(\calW \cap \setspan\calF_2) \oplus (\calW \cap (\setspan\calF_2)^\perp)}_{=\calW} \oplus \underbrace{\bbR\bm{D}_1 \oplus \bbR\bm{D}_2}_{=\calW^\perp},
\end{equation*}
where the second inequality follows from $P_{\setspan\calF_2}(\calW) \subseteq \calW$ and Lemma~\ref{lem:subspace_direct_sum}.
Following this (orthogonal) direct sum decomposition, we decompose a given matrix $\bm{X} \in \calS^n$ into $\bm{X} = \bm{Y} + \bm{Z} + s\bm{D}_1 + t\bm{D}_2$, where $\bm{Y} \in \calW \cap \setspan\calF_2$, $\bm{Z}\in \calW \cap (\setspan\calF_2)^\perp$, and $s,t\in \bbR$.
Let $a \coloneqq \sqrt{2}\abs{s}$, $b \coloneqq \norm{\bm{Z}}_{\rm F}$, and $c \coloneqq \sqrt{n-2}\abs{t}$.
By $\bbR\bm{D}_1 = \calW^\perp \cap \setspan\calF_2$ and $\bbR\bm{D}_2 = \calW^\perp \cap (\setspan\calF_2)^\perp$, we see that
\begin{align*}
\dist(\bm{X},\calW\cap \setspan \calF_2) &= \norm{\bm{Z} + s\bm{D}_1 + t\bm{D}_2}_{\rm F} = \sqrt{a^2 + b^2 + c^2},\\
\dist(\bm{X},\calW) &= \norm{s\bm{D}_1 + t\bm{D}_2}_{\rm F} = \sqrt{a^2 + c^2},\\
\dist(\bm{X},\setspan\calF_2) &= \norm{\bm{Z} + t\bm{D}_2}_{\rm F} = \sqrt{b^2 + c^2}. 
\end{align*}
The numbers $a$, $b$, and $c$ range over all nonnegative numbers as $\bm{X}$ ranges over the set $\calS^n$.
Therefore, $\eta$ is indeed the minimum positive number satisfying the following inequality:
\begin{equation*}
\sqrt{a^2 + b^2 + c^2} \le \eta\max\{\sqrt{a^2 + c^2}, \sqrt{b^2 + c^2}\}\ \text{ for all $a,b,c \in \bbR_+$}. 
\end{equation*}
When $c = 0$ and $a = b > 0$, since
\begin{equation*}
\frac{\sqrt{a^2 + b^2 + c^2}}{\max\{\sqrt{a^2 + c^2}, \sqrt{b^2 + c^2}\}} = \sqrt{2},
\end{equation*}
we see that $\eta \ge \sqrt{2}$.
Conversely, for any $a,b,c \in \bbR_+$ with $(a,b,c) \neq (0,0,0)$, we have
\begin{equation*}
\frac{\sqrt{a^2 + b^2 + c^2}}{\max\{\sqrt{a^2 + c^2}, \sqrt{b^2 + c^2}\}} \le \sqrt{2}, 
\end{equation*}
i.e., $\eta \le \sqrt{2}$.
Thus, we obtain $\eta = \sqrt{2}$.

\noindent
\fbox{$\alpha_1 = \sqrt{n + 8\sqrt{n-2} + 15}$ and $\beta_1 = \sqrt{\smash[b]{2\sqrt{2(n-2)} + 6\sqrt{2}}}$}
Recall that $\alpha_1$ and $\beta_1$ are defined as \eqref{eq:def_alphai} and \eqref{eq:def_betai}, respectively.
These definitions require $r_2$ and $\gamma_1$.
First, we see that $r_2 = 2$ since $\calF_2$ is linearly isomorphic to $\calS_+^2$.
Second, since $\norm{\bm{Z}_1}_{\rm F} = \sqrt{n-2}$ and $\lambda_{\rm min}^+(\bm{Z}_1) = 1$, the value $\gamma_1$ defined in \eqref{eq:def_gammai} is $\sqrt{n-2} + 2$.
Substituting these components into \eqref{eq:def_alphai} and \eqref{eq:def_betai}, we obtain the desired result.}

\section{\large Computation of the  radial modulus functions in  Example~\ref{ex:infeasible_X}}\label{apdx:alphaSDP_betaSDP}
{\small In this appendix, we illustrate the computations of $\alpha_{\rm SDP}(\rho)$ and $\beta_{\rm SDP}(\rho)$.
To compute them, we need to calculate the ingredients involved with them, i.e., $\overline{\kappa}(\rho;\bm{X}^\star)$, $\norm{\calA|_{\bbR\bm{S}^\star}}_{\rm op}$, $\sigma_{\rm min}^+(\overline{\calA})$, $\alpha$, and $\beta$.
Among them, $\overline{\kappa}(\rho;\bm{X}^\star)$, $\alpha$, and $\beta$ have been in fact already calculated in Appendix~\ref{apdx:calc_comps}; $\alpha$ is the same as $\alpha_1$ and $\beta$ is the same as $\beta_1$.
In what follows, we compute $\sigma_{\rm min}^+(\overline{\calA})$ and $\norm{\calA|_{\bbR\bm{S}^\star}}_{\rm op}$.

\noindent
\fbox{$\sigma_{\rm min}^+(\overline{\calA}) = \sqrt{n-1 - \sqrt{n^2-4n+5}}$}
For the instance discussed in Example~\ref{ex:infeasible_X}, the linear mapping $\overline{\calA}$ defined in \eqref{eq:barA} can be described as
\begin{equation*}
\overline{\calA}(\bm{X}) = \begin{pmatrix}
\sum_{i=1}^n X_{ii}\\
\sum_{i=3}^n X_{ii}
\end{pmatrix}
\end{equation*}
for every $\bm{X}\in \calS^n$.
The representation matrix of $\overline{\calA} \, \overline{\calA}^*$ with respect to the standard basis of $\bbR^2$ is $\begin{pmatrix}
n &  n-2\\
n-2 & n-2
\end{pmatrix}$, so we have $\sigma_{\rm min}^+(\overline{\calA}) = \sqrt{n-1 - \sqrt{n^2-4n+5}}$.

\noindent
\fbox{$\norm{\calA|_{\bbR\bm{S}^\star}}_{\rm op} = \sqrt{n-2}$}
Recall that $\calA$ is the linear mapping that maps $\bm{X} \in \calS^n$ to $\langle \bm{I}_n, \bm{X}\rangle$.
Then $\calA|_{\bbR\bm{S}^\star}$ is the linear mapping that maps $t\bm{S}^\star$ for $t\in \bbR$ to $\langle \bm{I}_n, t\bm{S}^\star\rangle = t(n-2)$.
Therefore, we have
\begin{equation*}
\norm{\calA|_{\bbR\bm{S}^\star}}_{\rm op} = \sup_{\substack{\bm{X} \in \bbR\bm{S}^\star,\\\bm{X}\neq \bm{O}}} \frac{\norm{\calA|_{\bbR\bm{S}^\star}(\bm{X})}_2}{\norm{\bm{X}}_{\rm F}} = \sup_{t\neq 0} \frac{\abs{t(n-2)}}{\norm{t\bm{S}^\star}_{\rm F}} = \sqrt{n-2}.
\end{equation*}

We can directly calculate $\beta_{\rm SDP}(\rho)$ defined in \eqref{eq:betaSDP} by using $\beta$ and $\overline{\kappa}(\rho;\bm{X}^\star)$.
The function $\alpha_{\rm SDP}(\rho)$ defined in \eqref{eq:alphaSDP} is calculated as
\begin{align*}
\alpha_{\rm SDP}(\rho) &= \sqrt{2}(1 + 2\sqrt{2} + 4\rho)\max\left\{\frac{2}{\sqrt{n-1 - \sqrt{n^2 -4n + 5}}}, \sqrt{n + 8\sqrt{n-2} + 15} + \frac{1}{\sqrt{n-2}}\right\} + \frac{1}{\sqrt{n-2}} \\
&= \sqrt{2}(1 + 2\sqrt{2} + 4\rho)\left(\sqrt{n + 8\sqrt{n-2} + 15} + \frac{1}{\sqrt{n-2}}\right) + \frac{1}{\sqrt{n-2}}, 
\end{align*}
where the second equality follows because
\begin{equation*}
\sqrt{n + 8\sqrt{n-2} + 15} + \frac{1}{\sqrt{n-2}} \ge \frac{2}{\sqrt{n-1 - \sqrt{n^2 -4n + 5}}}
\end{equation*}
holds under the assumption that $n\ge 3$.
}
\end{appendices}
%------------------

{\small 
\bibliographystyle{plainurl} % スタイルの指定
% https://mathlandscape.com/latex-bibstyles/：【LaTeX】BibTeXのスタイル76個一覧
\bibliography{myref} %
%------------------
}
\end{document}